\documentclass[11pt]{article}
\usepackage[utf8]{inputenc}  
\usepackage[british]{babel}
\usepackage{amsmath,amssymb,amsfonts,amsthm,bbm}
\usepackage{tikz,animate,media9}						%%%%%%%%%%
\usepackage{enumerate,enumitem,mdframed}

\usepackage{crossreftools,booktabs}

\makeatletter
\newcommand{\optionaldesc}[2]{%
  \phantomsection
  #1\protected@edef\@currentlabel{#1}\label{#2}%
}
\makeatother

\usepackage[mono=false]{libertine}
\usepackage[T1]{fontenc}
\usepackage[cmintegrals,libertine]{newtxmath}
\usepackage[cal=euler, calscaled=.95, frak=euler]{mathalfa}
\useosf
\usepackage[a4paper,vmargin={2cm,2cm},hmargin={2.25cm,2.25cm}]{geometry}
\usepackage[font=sf, labelfont={sf,bf}, margin=1cm]{caption}
\usepackage{numprint,ulem}
\newlist{compitem}{itemize}{4}
\setlist[compitem,1]{nolistsep,label=$\bullet$}

\usepackage{makecell} %pour passer des lignes dans les cases d'un tableau

\SetLabelAlign{CenterWithParen}{\makebox[1.6cm]{#1}}

\usepackage{hyperref}
\hypersetup{
	colorlinks=true,
		citecolor=blue!60!black,
		linkcolor=red!60!black,
		urlcolor=green!40!black,
		filecolor=yellow!50!black,
	breaklinks=true,
	pdfpagemode=UseNone,
	bookmarksopen=false,
}

\colorlet{link}{red!60!black}

\theoremstyle{plain}
\newtheorem{theorem}{Theorem}
\newtheorem{lemma}[theorem]{Lemma}
\newtheorem{proposition}[theorem]{Proposition}
\newtheorem{corollary}[theorem]{Corollary}

\theoremstyle{definition}

\theoremstyle{remark}
\newtheorem{remark}[theorem]{Remark}

\makeatletter
\newcommand{\labeltext}[3][]{%
    \@bsphack%
    \csname phantomsection\endcsname% in case hyperref is used
    \def\tst{#1}%
    \def\labelmarkup{\textcolor{link}}% How to markup the label itself
    \def\refmarkup{}%
    \ifx\tst\empty\def\@currentlabel{\refmarkup{#2}}{\label{#3}}%
    \else\def\@currentlabel{\refmarkup{#1}}{\label{#3}}\fi%
    \@esphack%
    \labelmarkup{#2}% visible printed text.
}
\makeatother

\renewcommand{\d}{\mathop{}\!\mathrm{d}} %pour les dt, dx, etc dans les integrales
\newcommand{\N}{\mathbb{N}} %entiers naturels
\newcommand{\Z}{\mathbb{Z}} %entiers relatifs
\renewcommand{\P}{{\mathbb{P}}} %symbole de proba (privilegier \proba)
\newcommand{\E}{\mathbb{E}} %symbole de l'esperance (privilegier \esp)
\newcommand{\Lp}{\mathbb{L}} %espace des fonctions integrables
\newcommand{\1}{\mathbbm{1}} %fonction indicatrice

\newcommand{\esp}[1]{\mathbb{E}\left[#1\right]} %esperance
\newcommand{\proba}[1]{\mathbb{P}\left(#1\right)} %proba

\newcommand{\esptilde}[1]{\widetilde{\mathbb{E}}\left[#1\right]} %esperance
\newcommand{\probatilde}[1]{\widetilde{\mathbb{P}}\left(#1\right)} %proba

\DeclareMathOperator*{\var}{\mathcal{V}ar} 
\DeclareMathOperator*{\vartilde}{\widetilde{\mathcal{V}ar}} 
 
\DeclareMathOperator*{\covtilde}{\widetilde{\mathcal{C}ov}}

\newcommand{\e}{\varepsilon} %epsilon

\newcommand{\T}{\mathcal{T}} %Arbre par l'algo 1 ou 2
\newcommand{\F}{\mathcal{F}} %foret de l'algo 2 / state of the coalescent at a given time

\newcommand{\birth}{\mathsf{b}} %temps de naissance
\newcommand{\coal}{\mathsf{c}} %temps de coalescence

\renewcommand{\H}{\mathsf{H}} %hauteur d'un sommet
\newcommand{\Height}{\mathsf{Height}} %hauteur d'un arbre
\newcommand{\h}{\mathsf{h}} %somme des 1/S_i

\newcommand{\Xb}{\mathrm{\mathbf{x}}} %suite de +/-1
\newcommand{\X}{\mathrm x} %coordonnees de la suite de +/-1
\newcommand{\XXb}{\mathrm{\mathbf{X}}} %suite aleatoire de +/-1
\newcommand{\XX}{\mathrm X} %coordonnees de la suite aleatoire de +/-1

\newcommand{\V}{\mathbb{V}} % labels des sommets de T_{n}
\newcommand{\Fbb}{\mathbb{F}} % labels des sommets geles de T_{n}
\newcommand{\A}{\mathbb{A}} % labels des sommets actifs de T_{n}

\newcommand{\cv}[1][n]{\enskip\mathop{\longrightarrow}^{}_{#1 \to \infty}\enskip}
\newcommand{\cvloi}[1][n]{\enskip\mathop{\longrightarrow}^{(d)}_{#1 \to \infty}\enskip}

\newcommand{\cvproba}[1][n]{\enskip\mathop{\longrightarrow}^{\P}_{#1 \to \infty}\enskip}
\newcommand{\cvLp}[1][p]{\enskip\mathop{\longrightarrow}^{\Lp^{#1}}_{n \to \infty}\enskip}

\usepackage{mathtools}
\DeclarePairedDelimiter\ceil{\lceil}{\rceil}

\def\dloc{\mathrm{d_{loc}}}
\newcommand\br[1]{\llbracket #1 \rrbracket}
\let\originalleft\left
\let\originalright\right
\renewcommand{\left}{\mathopen{}\mathclose\bgroup\originalleft}
\renewcommand{\right}{\aftergroup\egroup\originalright}

\DeclareSymbolFont{extraup}{U}{zavm}{m}{n}
\DeclareMathSymbol{\vardspade}{\mathalpha}{extraup}{81}
\DeclareMathSymbol{\varheart}{\mathalpha}{extraup}{86}
\DeclareMathSymbol{\vardiamond}{\mathalpha}{extraup}{87}
\DeclareMathSymbol{\varclub}{\mathalpha}{extraup}{84}

\makeatletter
\renewcommand*{\@fnsymbol}[1]{\ensuremath{\ifcase#1\or  \vardspade \or \varheart \or \vardiamond\or \varclub \or \bigstar \or
   \mathsection\or \mathparagraph\or \|\or **\or \dagger\dagger   \or \ddagger\ddagger \else\@ctrerr\fi}}
\makeatother

\author{
\'Etienne Bellin 
\thanks{CMAP, \'Ecole polytechnique, Institut Polytechnique de Paris, 91120 Palaiseau, France, \textsf{etienne.bellin@polytechnique.edu}
} \qquad  
Arthur Blanc-Renaudie 
\thanks{Tel Aviv University, Israel. Supported by the ERC consolidator grant 101001124 (UniversalMap). \newline \textsf{ablancrenaudiepro@gmail.com}
} 
\qquad  
Emmanuel Kammerer 
\thanks{CMAP, \'Ecole polytechnique, Institut Polytechnique de Paris, 91120 Palaiseau, France, \textsf{emmanuel.kammerer@polytechnique.edu}
}
 \qquad  
Igor Kortchemski 
\thanks{CNRS \& CMAP, \'Ecole polytechnique, Institut Polytechnique de Paris, 91120 Palaiseau, France, \textsf{igor.kortchemski@math.cnrs.fr}
} 
}
\title{Uniform attachment with freezing}

\numberwithin{equation}{section}
\begin{document}
\date{}
\vspace{-2cm}
\maketitle 
\begin{abstract}
In the classical model of random recursive trees, trees are recursively built by attaching new vertices to old ones. What happens if vertices are allowed to freeze, in the sense that new vertices cannot be attached to already frozen ones?  We are interested in the impact of freezing on the height of such trees. 
\end{abstract} 

\begin{figure}[h!]
 \centering
    \makeatletter\edef\animcnt{\the\@anim@num}\makeatother
\animategraphics[label=myAnim,scale=1.2,poster=last]{20}{film1/images_Partie}{1}{28} 
   \mediabutton[jsaction={anim.myAnim.playFwd();}]{\scalebox{1.5}[1.2]{\strut $\vartriangleright$}}
   \mediabutton[jsaction={anim.myAnim.pause();}]{\scalebox{1.5}[1.2]{\strut $\shortparallel$}}
    \caption{\label{simu1}Simulation of a tree of size $10000$ built by  uniform attachment with freezing, when the number of active (i.~e.~non-frozen) vertices roughly evolves as a positive fraction of the total number of vertices. Frozen vertices are blue; active  vertices are red. The  animation (played with Acrobat Reader) shows the resulting process as the number of steps increases.}
\end{figure}

\clearpage
\tableofcontents

\section{Introduction}
\label{sec:intro}

Random graphs are instrumental in the study of real-world networks. Uniform recursive trees (sometimes also called  uniform attachment trees) are one of such models. They are constructed recursively by starting with one single vertex, and successively attaching new vertices to a previous existing vertex, chosen uniformly at random. This model has been introduced in \cite{NR70} in the context of system generation, and has received considerable interest since, starting with graph-theoretical properties such as the height, number of leaves, etc. (see e.g.~\cite{Pit94} and references therein, as well as the survey \cite{SM95}).  Recursive trees have been proposed as models for the spread of epidemics  \cite{Moo74}, the family trees of preserved copies of ancient or medieval texts \cite{NH82}, pyramid schemes \cite{Gas77}, internet interface maps \cite{JKZV02} and appear in the study of Hopf algebras \cite{GL89}. Very interesting connections have been made with other probabilistic objects such as the Bolthausen-Sznitman coalescent \cite{GM05} and elephant random walks \cite{Kur16}. Uniform recursive trees have also been extended in several directions, for instance by introducing deterministic weights \cite{KR01} or random weights with a random environment \cite{BV06}.

In the present work, we introduce and study a modification of this model by introducing \emph{freezing}, in that existing vertices can freeze and new vertices cannot be attached to frozen vertices. 

Our motivation is twofold. First, in the context of real-world networks such mechanisms are natural: for instance, on the social network Twitter a user can choose to set their account to ``private'' which prevents strangers from ``following'' them; also performing an infection-tracing of an SIR epidemics falls within this framework (see Section \ref{sec:SIR}). Second, it is natural to investigate from a mathematical point of view the impact of freezing in dynamically-built random graph models. In particular, in a companion paper \cite{BBKKbis23+}, we investigate the regime where  the number of active vertices roughly evolves  as the total number of vertices to the power $\alpha$, for fixed $\alpha \in (0,1]$, we describe a phase transition where the macroscopic geometry of our model drastically changes according to the value of $\alpha$. Let us also mention the work \cite{BGY22}, which considers a growth-fragmentation-isolation process on random recursive trees in the context of contact tracing, see also \cite{Ber22}.

\paragraph{Uniform attachment with freezing.} Let us first define our model.  In order to incorporate freezing in the model of uniform attachment trees and to obtain results in a rather general setup, we shall fix beforehand the steps where an attachment takes place and the steps where freezing takes place. Specifically, our input is a deterministic sequence $\Xb=(\X_i)_{i \geq 1}$  of elements of $\{-1,+1\}$. Starting from a sole active vertex, we recursively build random trees by reading the elements of the sequence one after the other, by applying a ``freezing'' step when reading $-1$ (which amounts to freezing an active vertex chosen uniformly at random) and a ``uniform attachment'' step when reading $+1$ (which amounts to attaching a new vertex to an active vertex chosen uniformly at random). 

More precisely, given  a sequence $\Xb=(\X_i)_{i \geq 1}\in \{-1,+1\}^\N$, we set $S_{0}(\Xb) \coloneqq1$ and for every $n \geq 1$
\begin{equation}
\label{eq:defintro}S_n(\Xb) \coloneqq 1+ \sum_{i=1}^n \X_i \qquad \text{then} \qquad \tau(\Xb) \coloneqq \inf \{n \geq 1 : S_{n}(\Xb)=0\}. 
\end{equation}
Observe that $\tau(\Xb)$, if it is finite, is the first time when all the vertices are frozen, so that afterwards the tree does not evolve any more. 
For every $0 \leq n \leq  \tau(\Xb)$, let $\T_n(\Xb)$ be the  random tree recursively built in this fashion after reading the first $n$ elements of $\Xb$ (see {Algorithm \ref{algo1} in }Sec.~\ref{ssec:def} for a precise definition and Fig.~\ref{fig exemple arbre recursif avec gel} for an example). 

Let us comment on our choice of parametrization. It would have been possible to define the model starting with a random sequence $\Xb$, but the choice of a deterministic sequence defines a more general model, for which our results can then be applied.

\paragraph{Local limits.}
When $\tau(\Xb)=\infty$, the next result allows to give a meaning to $\T_{\infty}(\Xb)$ as a local limit of finite trees (see Sec.~\ref{sec:local} for background on the local topology).
\begin{theorem}
\label{thm:cvlocale}
Let $(\Xb^n)_{n\geq 0}$ be a sequence of elements of $\{-1,1\}^\N$. Suppose that there exists $\Xb$ such that $\tau(\Xb)=\infty$ and such that for all $i\geq 1$, $\Xb^n_i = \Xb_i$ for all $n$ large enough. Then the following assertions are equivalent:
\begin{enumerate}
    \item[{(I)}] The sequence of trees $(\T_n(\Xb^n))_{n\geq 0}$ converges locally, in distribution.
    \item[{(II)}] The sequence of trees $(\T_n(\Xb))_{n\geq 0}$ converges locally, almost surely, towards {\sout{a tree}} $\T_\infty(\Xb)$.
    \item[{(III)}] The sum $\sum_{i\geq 1} \frac{1}{S_i(\Xb)} \mathbbm 1_{\{\X_i=-1\}}$ diverges.
\end{enumerate}
In this case $\T_\infty(\Xb)$ is the local limit of  $(\T_n(\Xb^n))_{n\geq 0}$.
\end{theorem}

\paragraph{Examples.} It is interesting to note that this model encompasses the two classical models of random recursive trees and random uniform plane trees:
\begin{enumerate}[topsep=0pt,itemsep=-1ex,partopsep=1ex,parsep=1ex]
\item[--] when $\X_{i}=1$ for every $i \geq 1$, then $\T_{n}{(\Xb)}$ is a random recursive tree on $n$ vertices built by uniform attachment;
\item[--] when $\XXb=(\XX_{i})_{i \geq 1}$ is a sequence of {non constant} i.i.d.~uniform random variables on $ \{-1,+1\}$, then for every $n \geq 1$, conditionally given $\tau(\XXb)=n$, $\T_{n}{(\XXb)}$ is a uniform plane tree when the plane order among the vertices of $\T_{n}({\XXb})$ is chosen uniformly at random. \end{enumerate}
More generally we have the following result, in which we consider the case of random sequences.  In this case we will write $\XXb$ in upper case rather than $\Xb$ to emphasize the fact that the sequence is random. And, the law of $(\T_n(\XXb))_{n\geq 0}$ given the random sequence $\XXb$ follows the description above where $\XXb$ is considered to be fixed. In other words, we first choose $\XXb$ randomly, then, conditionally on $\XXb$, we construct $\T_n(\XXb)$ following Algorithm \ref{algo1} where we consider the sequence $\XXb$ to be deterministic. 

\begin{theorem}
\label{thm:BGW}
Let $p \in [0,1)$ and let $\XXb =(\XX_i)_{i\geq 1}$ be a  sequence of i.i.d random variables such that ${\P(X_1=+1)=p}$ and $\P(\XX_1=-1)=1-p$. Then $\T_\infty (\XXb)$ has the law of a Bienaym\'e tree with offspring distribution $\mu$ given by $\mu(k)= (1-p) p ^{k}$ for $k \geq 0$.
\end{theorem}

Roughly speaking, this comes from the fact that the subtrees grafted on the initial vertex evolve in an i.i.d.~fashion.

\paragraph{Height of uniform attachment trees with freezing.} 
Since we consider large trees, let us consider for $n\geq 1$, a sequence 
$\Xb^n \in \{-1,1\}^\N$ such that $\tau(\Xb^n)>n$, and let $\T_{n}\coloneqq\T_{n}(\Xb^{n})$. Also, to simplify notation, let $S_k^n\coloneqq S_k(\Xb^n)$ for  $0 \leq  k\leq n$.

Our next main theorem is that the height of $\T_{n}$ is of order 
\[ 
\h^+_n=\sum_{i=1}^n \frac{1}{S^{n}_i} \mathbbm 1_{\{\X^{n}_i=1\}}.
\]
\begin{theorem}
\label{thm:height}
The following results hold.
\begin{enumerate}
\item[{(I)}] Let $U^{n}$ be an active vertex of $\T_{n}$ chosen uniformly at random, and denote by $\mathsf{H}(U^{n})$ its height. Then for all $p \geq 1$:
\[
\frac{\mathsf{H}(U^{n})}{\h^+_n} \cvLp[p] 1.
\]
\item[{(II)}] For all $\e \in (0,1)$:
\[ 
\proba{1-\e \leq \frac{\Height(\T_{n})}{\h^+_n} \leq {e+\e}} \quad \mathop{\longrightarrow}_{n \rightarrow \infty} \quad 1.
\]
\item[{(III)}] Assume that ${\ln(n)}/{\h^+_n} \rightarrow 0$ as $n \rightarrow \infty$. Then for all $p\geq1$:
\[
\frac{\Height(\T_{n})}{\h^+_n} \cvLp[p] 1.
\]
\end{enumerate}
\end{theorem}

The proof of Theorem \ref{thm:height} is based on an alternative construction of $\T_{n}$, based on time-reversal, through a growth-coalescent process of rooted forests. This construction can be roughly described as follows (see Algorithm \ref{algo2} for a precise definition): start with a forest made of $S_{n}^{n}$ rooted one-vertex trees. Read successively $\X^{n}_{n},\X^{n}_{n-1}, \ldots, \X^{n}_{1}$; when reading $+1$ add a new one-vertex tree to the forest; when reading $-1$, choose successively two different forests uniformly at random, connect their roots by an edge, and root this new tree at the root of the first forest. It turns out that, in the end, one gets a tree having the same distribution as $\T^{n}$ (see Theorem \ref{thm:samelaw} for a precise statement). This is a generalization of the connection between random recursive trees and Kingman's coalescent which first appeared in \cite{DR76} (see \cite[Sec.~6]{Dev87}, \cite[Sec.~3]{Pit94}, \cite[Sec.~2.2]{AB15}, \cite{ABE18} for applications), in the context of union-find data structures (which are data structures that store a collection of disjoint sets where merging sets and finding a representative member of a set), see e.g.~\cite{KS78}. The main difference is that a growth feature must be added when freezing is introduced.

We are next interested in the specific regime where roughly speaking the number of active vertices increases linearly compared to the size of the tree.
\begin{theorem}
\label{thm:lineaire}
Let $c\in (0,1]$. Let $(A_n)_{n\in \N}$ be a sequence of positive numbers such that {$A_{n} \rightarrow\infty$} and $A_n=o(\log n)$ as $n\to  \infty$. Assume that
\begin{equation}
\label{eq:assumptions}
\lim_{\e \to 0} \limsup_{n\to \infty} \max_{A_n \le i \le \e n} \left| \frac{S^n_i}{i}-c \right| = 0
\qquad \text{and} \qquad \forall \e { \in (0,1)}, \enskip
\liminf_{n\to \infty} \min_{\e n \le i \le n } \frac{S^n_i}{n} >0.
\end{equation}
The following assertions hold.
\begin{enumerate}
\item[{(I)}] We have~~~~$\displaystyle \frac{\h^+_n}{\ln n} \xrightarrow[n\to\infty]{} \frac{c+1}{2c}$.
\item[{(II)}] 
For all $n\ge 1$, conditionally given  $\T_{n}$, let $V_{1}^n$   and $V_{2}^{n}$ be independent uniform vertices of $\T_{n}$. Then for every $p \geq 1$
\begin{equation}
\label{eq:fvprobahn}
 \frac{\H(V_{1}^n)}{\ln n} \cvLp[p] \frac{c+1}{2c} \qquad \textrm{and} \qquad  \frac{d^n(V^n_1,V^n_2)}{\ln n} \cvLp[p] \frac{c+1}{c},
\end{equation}
where $\H(V_{1}^n)$ is the height of $V_{1}^{n}$ in $\T_{n}$ and $d^{n}$ denotes the graph distance in $\T_{n}$.
\item[{(III)}]We have
\[\frac{\Height(\T_{n})}{\ln(n)} \cvproba \frac{{c+1}
}{2c}f(c),\]
where $f(c)$ is the unique solution with $f(c)>1$ to $f(c)(\ln f(c)-1)=(c-1)/(c+1)$. 
\end{enumerate}
\end{theorem}

Let us make some comments on these results. Roughly speaking, $\mathcal{T}_{n}$ looks like a ``tentacular bush'' in the sense that two typical vertices are always at the same distance of order $2\h_{n}^{+} \sim  \tfrac{c+1}{c} \cdot \ln (n)$, but the total height is of order $f(c) \cdot \h_{n}^{+}$. A typical example where the assumptions \eqref{eq:assumptions} are satisfied is e.g.~when $ \max_{A_n\le i\le n} |{S^n_i}/{i} -c| \rightarrow 0$. This is for instance the case when $p \in (1/2,1)$, $\XXb =(\XX_i)_{i\geq 1}$ is a  sequence of i.i.d random variables such that ${\P(X_1=+1)=p}$ and $\P(\XX_1=-1)=1-p$, setting $S_{i} ^{n}=S_i(\XXb)$, {conditionally given $\tau(\XXb)>n$}, Theorem \ref{thm:lineaire} applies almost surely with $c=2p-1$ thanks to the laws of large numbers. The reason we rather use \eqref{eq:assumptions} is for applications to contact-tracing in the SIR model (see Sec.~\ref{sec:SIR}).

\begin{figure}[!ht]
\label{fig:plot}
\centering
\includegraphics[scale=0.5]{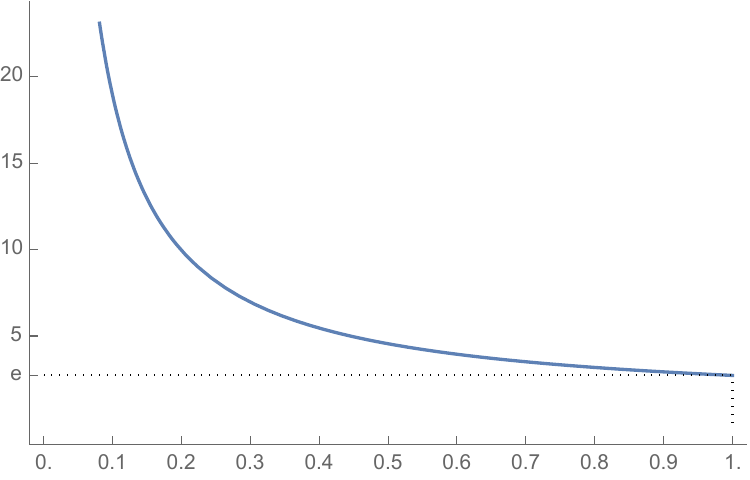}%
\caption{Plot of the function $ (\frac{{c+1}
}{2c}f(c): 0 <c \leq 1)$.
}
\end{figure}

Theorem \ref{thm:lineaire} {(III)} gives a more precise result than Theorem \ref{thm:height} {(II)} under an additional assumption. It generalizes the well known result that when $\T_{n}$  is a random recursive tree (which is obtained by taking $\X_{n}=1$  for every $n \geq 1$), we  have  \cite[Theorem 10]{Dev87}: $\Height(\T_{n})/\ln(n) \rightarrow e$ in probability (see also \cite[Theorem 1]{Pit94} for a {different approach based on continuous-time branching processes}). 

The proof of Theorem \ref{thm:lineaire} {(III)} is rather delicate: the main difficulty is that the presence of freezing impedes the direct use of branching process techniques. The alternative growth-coalescent process construction previously mentioned lies at the heart of our proof, {which is different from how corresponding results would be proved for the random recursive tree.}

Let us mention that it is possible to check that $(\T_{n}/\ln(n))_{n \geq 1}$ is not tight for the so-called Gromov--Hausdorff--Prokhorov topology. However, for the topology considered in \cite{ET21,Jan21},  $\T_{n}/\ln(n)$ converges to the so-called long dendron $\Upsilon_{\nu}$ (we use the notation of \cite[Example 3.12]{Jan21}) with $\nu$ being the Dirac mass $\nu=\delta_{(c+1)/2c}$.

{\paragraph{Application to contact-tracing in a stochastic SIR dynamics.} Our results may be applied to analyze the geometry of the so-called ``infection tree'' of a stochastic SIR dynamics, in which the vertices are individuals and where edges connect two individuals if one has infected the other. To keep this introduction at reasonable length, we refer to Section~\ref{sec:SIR} for details and to Theorem \ref{thm:SIR} for a precise statement.}

\paragraph{Plan of the paper.} The rest of the paper is organized as follows: In Section \ref{sec:discret}, we present our alternative construction, and describe a few of its properties. In Section \ref{sec:local}, we show Theorem \ref{thm:cvlocale} and \ref{thm:BGW} about local limits. In Sections \ref{sec:distances} and  \ref{sec:Linear}, we prove respectively Theorems \ref{thm:height} and \ref{thm:lineaire} about the height of our trees. We then apply our results in Section \ref{sec:SIR} to a stochastic SIR dynamics. Finally, in Section \ref{oo/oo} we give a few open problems.

\paragraph{Acknowledgments.} We thank Christina Goldschmidt for stimulating discussions at early stages of this work. We also thank the wonderful organization committee (Serte Donderwinkel, Christina Goldschmidt, Remco van der Hofstad, and Joost Jorritsma) of the RandNET Summer School and Workshop on Random Graphs, where this work was initiated. {We are also grateful to an anonymous referee, whose extremely thorough reading and numerous comments greatly improved the article.}

\section{Trees constructed by uniform attachment with freezing}
\label{sec:discret}

We start by defining our model. We also provide for future use a table of notation below (Table \ref{tab:secdiscret}).

\begin{table}[htbp]\caption{Table of the main notation and symbols introduced in Section \ref{sec:discret} and used later.}
\centering
\begin{tabular}{c c p{12cm} }
\toprule
$\N$ && $=\{1,2,3,\dots\}$ positive integers\\
$\llbracket i,j \rrbracket$ && $=\{i,i+1,\dots,j-1,j\}$ all integers between $i$ and $j$\\
$\#A$ && cardinality of a finite set $A$\\
\hline
$\Xb = (\X_n)_{n\in\N}$ && a sequence of elements of $\{-1,1\}$\\
$S_n(\Xb)$ && $=1+\sum_{i=1}^n \Xb_i$\\
$\tau(\Xb)$ && $=\inf \{n \geq 1 : S_{n}(\Xb)=0\}$ \\
\hline
$\T_n(\Xb)$ && tree built at time $n$ by Algorithm \ref{algo1}; $S_{n}(\Xb)$ is its number of active vertices \\
$N_n(\Xb)$ && total number of vertices in $\T_n(\Xb)$; $N_n(\Xb)=(S_n(\Xb)+n+1)/2$ when $n \leq \tau(\Xb)$\\
\hline
$\F_n^n(\Xb),\F_{n-1}^n(\Xb), \ldots, \F_0^n(\Xb)$ && the forest of trees built by Algorithm \ref{algo2}\\

$\T^n(\Xb)$ && $=\F_0^n(\Xb)$ the -{output of} Algorithm \ref{algo2}\\
\hline
$\Fbb_{n}(\Xb)$ && $=\{i\in\llbracket 1,n \rrbracket: \X_{i}=-1\}$ the {set of} labels of frozen vertices of $\T^n(\Xb)$\\

$\A_{n}(\Xb)$ && $=\{a_{1},\ldots,a_{S_{n}(\Xb)}\}$ the {set of} labels of active vertices of $\T^n(\Xb)$\\

$\V_{n}(\Xb)$ && $=\Fbb_{n}(\Xb)\cup\A_{n}(\Xb)$ the {set of} labels of all vertices of $\T^n(\Xb)$\\
\hline
$\birth_n(u)$ && the birth time of $u \in \V_{n}(\Xb)$ in the construction of $\T^n(\Xb)$  by Algorithm \ref{algo2} \\

$\coal_{n}(u,v)$ && the coalescence time between $u,v \in \V_{n}(\Xb)$ in the construction of $\T^n(\Xb)$  by Algorithm \ref{algo2} \\
\hline
 $\H^n_i(u)$  && the height of vertex $u$ in $\F^n_i(\Xb)$\\
\bottomrule
\end{tabular}
\label{tab:secdiscret}
\end{table}

\subsection{Uniform attachment with freezing: recursive construction}
\label{ssec:def}
Let $\Xb=(\X_i)_{i \geq 1}\in \{-1,+1\}^\N$. 
In what follows, we formally construct the random trees $(\T_n(\Xb))_{n\geq 0}$. These trees will be rooted, edge-labelled, and vertex-labelled. The label of an edge is the time it appears.
The label of a vertex {in $\T_{n}(\Xb)$} is either the time it {froze}, or the label ``$a$'' if it is still active at time $n$.

\paragraph{\labeltext[1]{Algorithm 1.}{algo1}}
\begin{compitem}
\item Start with the tree $\T_0(\Xb)$ made of a single root vertex labelled $a$.
\item For every $n \geq 1$, if $\T_{n-1}(\Xb)$ has no vertices labelled $a$, then set $\T_n(\Xb) = \T_{n-1}(\Xb)$. Otherwise let $V_n$ be a random uniform active vertex of $\T_{n-1}(\Xb)$, chosen independently from the previous ones. Then:
\begin{compitem}
\item[--] if $\X_n=-1$, build $\T_n(\Xb)$ from $\T_{n-1}(\Xb)$ simply by replacing the label $a$ of $V_n$ with the label $n$;
\item[--] if $\X_n=1$, build $\T_n(\Xb)$ from $\T_{n-1}(\Xb)$ by adding an edge labelled $n$ between $V_n$ and a new vertex labelled $a$. 
\end{compitem}
\end{compitem}
For $n \geq 0$, we view $\T_n(\Xb)$ as a rooted, double-labelled  {tree}  (that is, edge-labelled and vertex-labelled). If $N_{n}(\Xb)$ represents the total number of vertices  of $\T_n(\Xb)$, observe that by construction $N_n(\Xb)=(S_n(\Xb)+n+1)/2$ for $0 \leq n \leq \tau(\Xb)$.

\begin{figure}
\begin{center}
\begin{tikzpicture}[scale = 0.5]
\draw[->] (0,0) -- (0,3.5);
\draw[->] (0,0) -- (5.5,0);

\node[above,font=\footnotesize] at (0,3.5) {$S_n(\Xb)$};
\node[right,font=\footnotesize] at (5.5,0) {$n$};

\node[draw,circle,fill,inner sep = 1pt] at (0,1) {};
\node[draw,circle,fill,inner sep = 1pt] at (1,2) {};
\node[draw,circle,fill,inner sep = 1pt] at (2,1) {};
\node[draw,circle,fill,inner sep = 1pt] at (3,2) {};
\node[draw,circle,fill,inner sep = 1pt] at (4,3) {};
\node[draw,circle,fill,inner sep = 1pt] at (5,2) {};

\draw (0,1)--(1,2)--(2,1)--(3,2)--(4,3)--(5,2);

\node[left,font=\footnotesize] at (0,0) {0};
\node[left,font=\footnotesize] at (0,1) {1};
\node[left,font=\footnotesize] at (0,2) {2};
\node[left,font=\footnotesize] at (0,3) {3};

\node[below,font=\footnotesize] at (0,0) {0};
\node[below,font=\footnotesize] at (1,0) {1};
\node[below,font=\footnotesize] at (2,0) {2};
\node[below,font=\footnotesize] at (3,0) {3};
\node[below,font=\footnotesize] at (4,0) {4};
\node[below,font=\footnotesize] at (5,0) {5};

\draw[dotted] (0,2)--(5,2)--(5,0);
\end{tikzpicture}
\hspace{2em}
\begin{tikzpicture}[scale = 0.8,
sommet/.style = {draw,circle, font=\scriptsize,inner sep=0,minimum size=13pt},
gele/.style = {fill=cyan},
etiquete/.style = {font = \scriptsize}]

\node[sommet] (0) at (-0.5,0.4) {$a$};
\node[etiquete] (t0) at (-0.5,-0.2) {$\T_0(\Xb)$};

\node[sommet] (1) at (0.5,1) {$a$};
\node[sommet] (2) at (0.6,2) {$a$};
\draw (1)--node[left,etiquete] {1} (2);
\node[etiquete] (t1) at (0.5,0.4) {$\T_1(\Xb)$};

\node[sommet,gele] (3) at (1.6,0) {2};
\node[sommet] (4) at (1.7,1) {$a$};
\draw (3) --node[left,etiquete] {1} (4);
\node[etiquete] (t2) at (1.6,-0.6) {$\T_2(\Xb)$};

\node[sommet,gele] (5) at (3,0.5) {2};
\node[sommet] (6) at (3.1,1.5) {$a$};
\node[sommet] (7) at (2.7,2.5) {$a$};
\draw (5) --node[left,etiquete] {1} (6);
\draw (6) --node[left,etiquete] {3} (7);
\node[etiquete] (t3) at (3,-0.1) {$\T_3(\Xb)$};

\node[sommet,gele] (8) at (4.5,0) {2};
\node[sommet] (9) at (4.6,1) {$a$};
\node[sommet] (10) at (4.2,2) {$a$};
\node[sommet] (11) at (5,2) {$a$};
\draw (8) --node[left,etiquete] {1} (9);
\draw (9) --node[left,etiquete] {3} (10);
\draw (9) --node[right,etiquete] {4} (11);
\node[etiquete] (t4) at (4.5,-0.6) {$\T_4(\Xb)$};

\node[sommet,gele] (12) at (6.3,0.4) {2};
\node[sommet] (13) at (6.4,1.4) {$a$};
\node[sommet,gele] (14) at (6,2.4) {$5$};
\node[sommet] (15) at (6.8,2.4) {$a$};
\draw (12) --node[left,etiquete] {1} (13);
\draw (13) --node[left,etiquete] {3} (14);
\draw (13) --node[right,etiquete] {4} (15);
\node[etiquete] (t5) at (6.3,-0.2) {$\T_5(\Xb)$};

\end{tikzpicture}
\caption{On the left is represented the walk $(S_n(\Xb))_{n \geq 0}$ up to time $n=5$ associated with the sequence $\Xb=+1,-1,+1,+1,-1,\dots$. On the right, a possible realisation of the trees $\T_0(\Xb)$ to $\T_5(\Xb)$ given this sequence. Frozen vertices have been colored in blue.}
\label{fig exemple arbre recursif avec gel}
\end{center}
\end{figure}
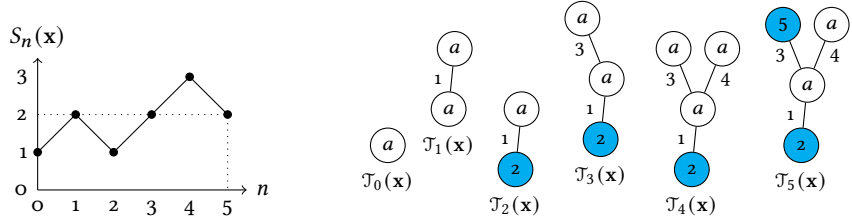

\subsection{Uniform attachment with freezing: growth-coalescent construction}
\label{ssec:growthcoalescent}

As before, let $\Xb=(\X_n)_{n \geq 1}\in \{-1,1\}^\N$. We introduce in this section an alternative time-reversed construction of uniform attachment trees with freezing. It may be seen as a growth coalescence process of forests, and as previously mentioned, is a generalization of the connection between random recursive trees and Kingman's coalescent introduced in \cite{DR76}. Most of our proofs are based on this construction.

\paragraph{\labeltext[2]{Algorithm 2.}{algo2}}
Fix $0 \leq n \leq \tau(\Xb)$. We construct a sequence  $(\F_{n}^n(\Xb),\F_{n-1}^n(\Xb), \ldots, \F_{0}^n(\Xb))$ of {(possibly empty)} forests of rooted, edge-labelled, vertex-labelled {\sout{unoriented}} trees by induction as follows.
\begin{compitem}
\item Let $\F_{n}^n(\Xb)$ be a forest made of $S_{n}(\Xb)$ one-vertex trees labelled $a_{1}, \ldots,a_{S_{n}(\Xb)}$.
\item For every $1\leq i \leq n$, if $\F^{n}_{i}(\Xb)$ has been constructed, define $\F^{n}_{i-1}(\Xb)$ as follows:
\begin{compitem}
\item[--] if $\X_{i}=-1$, $\F^n_{i-1}(\Xb)$ is obtained by adding to $\F^n_{i}(\Xb)$ a new one-vertex tree labelled $i$;
\item[--] if $\X_{i}=1$, let $(T_{1},T_{2})$ be a  pair of different random trees in $\F_{i}^n(\Xb)$ chosen uniformly at random, independently of the previous choices; then $\F^n_{i-1}(\Xb)$ is obtained from $\F^n_{i}(\Xb)$ by adding an edge labelled $i$ between the roots $r(T_{1})$ and $r(T_{2})$ of respectively $T_{1}$ and $T_{2}$, and rooting the tree thus obtained at $r(T_{1})$;
\end{compitem}
\item Let $\T^n(\Xb)$ be the only tree of $\F^n_0(\Xb)$.
\end{compitem}

 \begin{figure}[!ht] \centering
\includegraphics[width=0.95\linewidth]{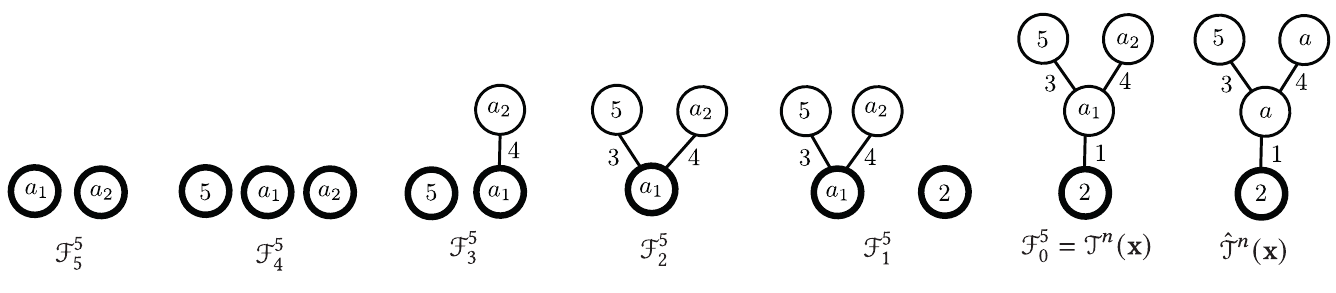}
\caption{An illustration of Algorithm \ref{algo2} with $n=5$ and $(\X_{5},\X_{4},\X_{3},\X_{2},\X_{1})=(-1,1,1,-1,1)$ (this is the same sequence as in Fig.~\ref{fig exemple arbre recursif avec gel}). For example, since $\X_{2}=-1$, $\mathcal{F}^{5}_{1}$ is obtained from $\mathcal{F}^{5}_{2}$ by adding a new tree made of a vertex labeled $2$. Since $\X_{4}=1$, to build  $\mathcal{F}^{5}_{3}$ from  $\mathcal{F}^{5}_{4}$ we have chosen in $\mathcal{F}^{5}_{3}$  two trees  $(T_{1},T_{2})$ with $T_{1}$ being the vertex $a_{1}$ and $T_{2}$ being the vertex $a_{2}$, and  we have added an edge labelled $4$ between the roots $r(T_{1})=a_{1}$ and $r(T_{2})=a_{2}$ of {\sout{respectively}} $T_{1}$ and $T_{2}$, and  {rooted} the tree thus obtained at $r(T_{1})=a_{1}$.
}
\label{fig:growthcoalescent}
\end{figure}

\begin{figure}[h!]
 \centering
    \makeatletter\edef\animcnt{\the\@anim@num}\makeatother
\animategraphics[label=myAnim2,scale=0.9,poster=last]{5}{film2/image}{1}{26} 
   \mediabutton[jsaction={anim.myAnim2.playFwd();}]{\scalebox{1.5}[1.2]{\strut $\vartriangleright$}}
   \mediabutton[jsaction={anim.myAnim2.pause();}]{\scalebox{1.5}[1.2]{\strut $\shortparallel$}}
    \caption{\label{simu1}Simulation of  Algorithm \ref{algo2}: the  animation (played with Acrobat Reader) shows the resulting process as the number of steps increases. When a new vertex appears, its color is chosen at random. When two trees merge, the resulting tree keeps the color of the largest of the two trees.}
\end{figure}

We shall soon see in Sec.~\ref{ssec:constructions} the equivalence between Algorithm \ref{algo1} and Algorithm \ref{algo2}, but before that, we introduce some important notation  and consequences which will be useful in the analysis of the height of trees built by uniform attachment with freezing.

\paragraph{Important convention} Throughout this article, when the context is clear, we will sometimes drop the parameter $\Xb$ to lighten notation: {in particular, we shall write} $\tau$, $S_n$, $\T_n$, $\F_i^n$ instead of $\tau(\Xb)$, $S_n(\Xb)$, $\T_n(\Xb)$, $\F_i^n(\Xb)$, etc.

\subsection{Laws of the birth and coalescence times}
\label{ssec:estimates}

Let $\Xb=(\X_n)_{n \geq 1}\in \{-1,1\}^\N$. First, define
\begin{equation}
\label{eq:labels}
\Fbb_{n}(\Xb) =  \{i \in \llbracket 1,n \rrbracket: \X_{i}=-1\},  \qquad \A_{n}(\Xb)=  \{a_{1}, \ldots, a_{S_{n}(\Xb)}\}, \qquad   \V_{n}(\Xb)= \Fbb_{n}(\Xb) \cup \A_{n}(\Xb).
\end{equation}
It is crucial to observe that while $\T^n$ is a random tree, the labels of its vertices are deterministic quantities that only depend on $\Xb$: $\A_{n}$ are the labels of the active vertices of $\T^n$ and $\Fbb_{n}$ are the labels of the frozen vertices of $\T^{n}$. In particular, the elements of $\V_{n}$ will be called vertices of $\T^n$.

Next, for every $u\in\V_{n}$, let $\birth_n(u)$ be the largest $i \in \{0,1, \ldots,n\}$ such that $u$ belongs to the forest $\F^n_i$. Explicitly, if $u \in \A_{n}$ is an active vertex then $\birth_n({u})=n$, and if $u \in \Fbb_{n}$ (note that $u$ is then an integer) then $\birth_n(u)=u-1$ (see Fig.~\ref{fig:growthcoalescent} for an example). We say that $\birth_n(u)$ is the \emph{birth time} of $u$, since it encodes the first time when vertex $u$ appears in Algorithm \ref{algo2}. For  $0 \leq i \leq \birth_n(u)$, let $\H^n_i(u)$ be the height of vertex $u$ in $\F^n_i$ (that is, the graph distance between $u$ and the root of {whichever tree of $\F^{n}_{i}$ contains $u$}).

Finally,  for $u,v \in \V_{n}$, let $\coal_{n}(u,v)$ be the largest $i \in \{0,1, \ldots,n\}$ such that $u$ and $v$  belong to the same tree in the forest $\mathcal{F}^{n}_{i}$ obtained when building $\T^{n}$ in Algorithm \ref{algo2}.  We say that $\coal_{n}(u,v)$ is the \emph{coalescence time} between $u$ and $v$, since it encodes the first time $u$ and $v$ belong to the same tree  in Algorithm \ref{algo2} (observe that while $\birth_{n}(u)$ is deterministic, $\coal_{n}(u,v)$ is random).

We now state several simple consequences {of Algorithm \ref{algo2}}, which will be useful to study  the geometry of $\T^{n}$.

\begin{lemma}
\label{lem:accroissementH}
Fix $1 \leq n \leq \tau$  and  $ {v} \in \V_{n}$. For  every $1 \leq i \leq \birth_{n}(v)$:
\[
\proba{\H^n_{i-1}(v)-\H^n_i(v)=1 | \F^n_n, \ldots, \F^n_i} = \frac{1}{S_{i}}\mathbbm{1}_{\{\X_i=1\}}.
\]
\end{lemma}

This is clear from the definition of Algorithm \ref{algo2}. {It follows from Lemma \ref{lem:accroissementH} that} if $(Y_i)_{1\leq i \leq n}$ are independent Bernoulli random variables of respective parameters $(1/S_i)_{1\leq i \leq n}$, we have for every $u \in \V_n$
\begin{equation}
\label{eq:H0} \H^{n}_{0}(u) \quad \mathop{=}^{(d)} \quad \sum_{i=1}^{\birth_{n}(u)} Y_i\1_{\{\X_i=1\}}.
\end{equation}

The next result identifies the law of the birth time of a vertex of $\T^{n}$ chosen uniformly at random.

\begin{lemma}
\label{lem:bunif}
Fix $1 \leq n \leq \tau$. Let $V$ be an element of $\V_{n}$ chosen uniformly at random. For every $1 \leq m \leq n$,
\[
\proba{\birth_n(V) < m}  = \frac{m+1-S_m}{n+1+S_n}.    
\]
\end{lemma}

\begin{proof}
{We noted} that $\birth_{n}(u)=n$ if $u\in \A_{n}$ and $\birth_{n}(u) =u-1$ if $u \in \Fbb_{n}$. Also, conditionally given the fact that $V \in  \Fbb_{n}$, $\birth_{n}(V)$ is uniform on $\Fbb_{n}$. Since  $\#\A_{n}=S_{n}$ and  $\#\Fbb_{n}=\sum_{i=1}^n \mathbbm{1}_{\{\X_i=-1\}}$, it follows by {the} definition of $\Fbb_{n}$ that
\[\proba{\birth_n(V) < m} = \frac{\sum_{i=1}^m \mathbbm{1}_{\{\X_i=-1\}}}{S_n+\sum_{i=1}^n \mathbbm{1}_{\{\X_i=-1\}}} = \frac{m+1-S_m}{n+1+S_n},\]
where for the last equality we have used the fact that $\sum_{i=1}^k \mathbbm{1}_{\{\X_i=-1\}}=(k-S_{k}+1)/2$ for every $1 \leq k \leq n$.
\end{proof}

 The last useful result identifies the law of the coalescence times between two vertices.

\begin{lemma}
\label{loi du temps de coalescence}
Fix $1 \leq n \leq \tau$ and consider $u,v \in \V_{n}$. Then for every $0 \leq c < \birth_{n}(u) \wedge \birth_{n}(v)$ with $\X_{c+1}=1$:
\[\proba{\coal_{n}(u,v)=c}=\frac{1}{\binom{S_{c+1}}{2}} \prod_{\substack{i=c+2 \\ \text{s.t. } \X_i=1}}^{ \birth_{n}(u) \wedge \birth_{n}(v) } \left(1-\frac{1}{\binom{S_i}{2}}\right).\]
\end{lemma}

\begin{proof}
When running Algorithm \ref{algo2}, the factors in the product correspond to the probability that the two vertices born at times $\birth_n(u)$ and $\birth_n(v)$ do not coalesce until time $c$ and the factor $1/\binom{S_{c+1}}{2}$ corresponds to the {conditional} probability that the two vertices coalesce at time $c$, {given that they have not previously done so}.
\end{proof}

\subsection{Equivalence between the two constructions}
\label{ssec:constructions}
\begin{table}[htbp]\caption{Table of the main notation and symbols introduced in Section \ref{ssec:constructions}.}
\centering
\begin{tabular}{c c p{11cm} }
\toprule
$\hat{\T}^n(\Xb)$ &  & the tree obtained from $\T^n(\Xb)$ by relabelling all its active vertices by ``$a$'' \\
\hline
\makecell[cl]{$\mathfrak T_{k,n}(A)$, $n\geq 1$ \\$0 \leq k \leq n$, $A \subset \N$ \\ $\#A= 2n-k-1$} && \makecell[cl]{set of rooted trees with $n$ vertices, such that all the $n-1$ edges \\and $n-k$ of the vertices are labelled in a one-to-one manner \\ with the integers of $A$, while the other $k$ vertices are labelled \\ with the letter ``$a$''} \\
\hline
$\mathfrak T_{k,n}^+(A)$ && \makecell[cl]{set of all trees in $\mathfrak T_{k,n}(A)$ such that edge-labels increase along \\ paths directed {away} from the root, and all vertices labelled with an \\integer (i.e.~not ``$a$'') have a label which is larger than the labels \\of their adjacent edges} \\
\hline
$\mathfrak T_{k,n}^+$ && $\mathfrak T_{k,n}^+(A)$ with $A=\llbracket1,2n-k-1\rrbracket$\\
\bottomrule
\end{tabular}
\label{tab:equivalence}
\end{table}

Observe that the active vertices of $\T_n$ are all labelled $a$ while the active vertices of $\T^n$ are labelled $a_{1}, \ldots,a_{S_{n}}$. It turns out that $\T_n$ is equal in law to $\T^n$ when its active vertices are all relabelled $a$. More precisely, let $\hat{\T}^n$ be the tree obtained from $\T^n$ by relabeling its $S_{n}$ active vertices by $a$ (see the right-most part of Fig.~\ref{fig:growthcoalescent} for an illustration).

\begin{theorem}
\label{thm:samelaw}
The two trees  $\T_n$ and $\hat{\T}^n$ have the same distribution.
\end{theorem}

To prove this result, we first identify the law of $\T_n$.  {Let us determine} the range of $\T_n$. For $n\in \N$, $0 \leq k \leq n$ and $A \subset \N$ with $\#A = 2n-k-1$, let $\mathfrak T_{k,n}(A)$ be the set of rooted trees with $n$ vertices, such that all the $n-1$ edges and $n-k$ of the vertices are labelled in a one-to-one manner with the integers of $A$, while the other $k$ vertices are labelled with the letter $a$. We then define $\mathfrak T_{k,n}^+(A)$ to be the set of all trees in $\mathfrak T_{k,n}(A)$ such that  edge labels increase along paths directed from the root, and all vertices labelled with an integer (i.e.~not $a$) have a label which is larger than the labels of their adjacent edges. We set $\mathfrak T_{k,n}^+ =\mathfrak T_{k,n}^+(\llbracket1,2n-k-1\rrbracket)$. Note that $\T_n$ is an element of $\mathfrak T_{S_n,N_n}^+$ for $0 \leq n \leq \tau$ (see the left-most tree in Fig.~\ref{fig exemple d'image par b} for an example) where $N_n = (S_n+n+1)/2$ is the total number of vertices in $\T_n$. In particular, $\mathfrak T_{k,n}^+$ is the set of all possible trees with $k$ active vertices and $n-k$ frozen vertices which can be obtained by Algorithm \ref{algo1}.

\begin{lemma}
\label{lem:loiTn}
Fix $0 \leq n \leq \tau$ and $T_n \in \mathfrak T_{S_n,N_n}^+$. Then
\[\proba{\T_n=T_{n}}= \prod_{1 \leq i \leq n} \frac{1}{S_{i-1}}.\]
\end{lemma}

\begin{proof}
For two trees $ T,T'$ with $T \in \mathfrak T_{k,n}^+$, we write $T \leadsto T'$ if either $T' \in \mathfrak T_{k+1,n}^+$ and $T'$ is obtained by replacing in $T$ the label of a vertex labeled $a$ by $2n-k$, or $T' \in \mathfrak T_{k,n+1}^+$ and $T'$ is obtained from $T$ by adding an edge labelled by $2n-k$ between a new vertex labeled $a$ and an existing $a$-labeled vertex.

Since $T_{n} \in \mathfrak T_{S_n,N_n}^+$, observe that there is a unique sequence $(T_{0}, \ldots,T_{n-1})$ such that $T_{i} \in   \mathfrak T_{S_i,N_i}^+$ for every $0 \leq i \leq n-1$ and such that $T_{i} \leadsto T_{i+1}$ for every $0 \leq i \leq n-1$. In particular, 
\[\proba{\T_n=T_{n}}=\proba{(\T_0,\T_1 \ldots, \T_n)=(T_{0},T_{1}, \ldots,T_{n}) }= \prod_{i=1}^{n} \proba{\T_i=T_{i} |\T_{i-1}=T_{i-1} }.\]
Then observe that $ \proba{\T_i=T_{i} |\T_{i-1}=T_{i-1} }= \frac{1}{S_{i-1}}$. Indeed, $T_{i-1}$ has $S_{i-1}$ active vertices. Thus, if $\X_{i}=1$, then the probability of attaching a new vertex to a given active vertex is $1/S_{i-1}$; if $\X_{i}=-1$, then the probability of freezing a given active vertex is $1/S_{i-1}$.
\end{proof}

\begin{proof}[Proof of Theorem \ref{thm:samelaw}]
First of all, {\sout{recall that}} by Lemma \ref{lem:loiTn}, 
for every $T_{n} \in  {\mathfrak T}_{S_n,N_n}^+$ we have
\[\proba{{\T}_n={T}_{n}}= \prod_{1 \leq i \leq n} \frac{1}{S_{i-1}}.\]
Now let  $\widetilde {\mathfrak T}_{k,n}^+$ be the set of all trees obtained from $\mathfrak T_{k,n}^+$ by labelling $a_{1}, \ldots, a_{k}$ their $k$ active vertices. It is enough to show that for every $\widetilde{T}_{n} \in \widetilde {\mathfrak T}_{k,n}^+$ we have
\begin{equation}
\label{eq:toshow}
\proba{{\T}^n=\widetilde{T}_{n}}=\frac{1}{S_{n}!}\prod_{1 \leq i \leq n} \frac{1}{S_{i-1}}.
\end{equation}
Indeed, since  there are $S_{n}!$ ways to relabel by $a_{1}, \ldots,a_{k}$ the $S_{n}$ active vertices of any given tree $T_{n} \in  {\mathfrak T}_{k,n}^+$, this will imply that $\proba{{\T}_n={T}_{n}}=\P(\hat{\T}^n=T_{n})$.

To establish \eqref{eq:toshow} we need to introduce some notation. Let
$\widetilde{\mathfrak F}_k^n$ be the set of all forests of rooted, edge-labelled and vertex-labelled trees, such that:
\begin{itemize}[noitemsep,nolistsep]
\item[--] $S_{n}$ vertices are labelled $a_{1}, \ldots,a_{S_{n}}$;
\item[--] all the other vertices are labelled in a one-to-one fashion by $\llbracket k+1,n\rrbracket \cap \{1 \leq i \leq n : \X_{i}=-1\} $;
\item[--] the edges are labelled in a one-to-one fashion by $\llbracket k+1,n\rrbracket \cap \{1 \leq i \leq n : \X_{i}=1\} $.
\end{itemize}
In addition, the labelling satisfies the following condition: edge-labels increase along paths directed from the roots, and all vertices labelled with an integer (i.e.~not $a$) have a label which is larger than the labels of their adjacent edges. Notice that forests in $\widetilde{\mathfrak F}_k^n$ are made of $S_{k}$ trees.

Observe that the random forest $\F^{n}_{i}$ satisfies $\F^{n}_{i} \in \widetilde{\mathfrak F}_i^n$ for every $0 \leq i \leq n$. If $F_{i} \in   \widetilde{\mathfrak F}_i^n$ and  $F_{i-1} \in   \widetilde{\mathfrak F}_{i-1}^n$, we write $F_{i} \leadsto F_{i-1}$ if $F_{i-1}$ is obtained from $F_{i}$ by either adding a new tree made of a sole vertex labeled $i$, or choosing two different trees $T_{1},T_{2}$ in $F_{i}$ and adding an edge labeled $i$ between the roots $r(T_{1})$ and $r(T_{2})$ of respectively $T_{1}$ and $T_{2}$, and  rooting the tree thus obtained at $r(T_{1})$.

Now, let $F^n_0$ be the forest made of the tree $ \widetilde{T}_{n}$. Observe that there is a unique sequence $({F}^{n}_{n}, \ldots,{F}^{n}_{1})$ such that ${F}^{n}_{i} \in \widetilde{\mathfrak F}_i^n$ for every $0 \leq i \leq n-1$ and such that ${F}^{n}_{i} \leadsto {F}^{n}_{i-1}$ for every $1 \leq i \leq n$.  Notice that $ {F}^{n}_{i}$ is made of $S_{i}$ trees. In particular, 
\begin{eqnarray*}
\proba{\T^n=\widetilde{T}_{n}}&=&\proba{(\mathcal{F}^{n}_n,\mathcal{F}^{n}_{n-1}, \ldots, \mathcal{F}^{n}_{0})=({F}^{n}_{n},{F}^{n}_{n-1}, \ldots,{F}^{n}_{0} )}\\
&=& \prod_{i=1}^{n} \proba{\mathcal{F}^{n}_{i-1}={F}^{n}_{i-1} |\mathcal{F}^{n}_{i}={F}^{n}_{i} }. 
\end{eqnarray*}
Then observe that if $\X_{i}=-1$ we have $\proba{\mathcal{F}^{n}_{i-1}={F}^{n}_{i-1} |\mathcal{F}^{n}_{i}={F}^{n}_{i} }=1$ since to build $\mathcal{F}^{n}_{i-1}$ from $\mathcal{F}^{n}_{i}$ with probability one we add a new one-vertex tree labelled $i$ to  $\mathcal{F}^{n}_{i-1}$, and since $\X_{i}=-1$, the forest ${F}^{n}_{i-1}$ is also obtained from ${F}^{n}_{i}$ by adding a new one-vertex tree labelled $i$. If $\X_{i}=1$ we have the equality $\proba{\mathcal{F}^{n}_{i-1}=\mathcal{F}^{n}_{i-1} |\mathcal{F}^{n}_{i}=\mathcal{F}^{n}_{i} }= \frac{1}{S_{i}(S_{i}-1)}$; indeed, there are $S_{i}(S_{i}-1)$ ways of choosing a pair $(T_{1},T_{2})$ of different trees in ${F}^{n}_{i}$, and one of them gives ${F}^{n}_{i-1}$ when adding an edge labeled $i$ between the roots $r(T_{1})$ and $r(T_{2})$ of respectively $T_{1}$ and $T_{2}$,  and rooting the tree thus obtained at $r(T_{1})$.

As a consequence,
\[\proba{{\T}^n=\widetilde{T}_{n}}=  \prod_{\substack{1 \leq i \leq n \\ \X_{i}=1}} \frac{1}{S_{i}(S_{i}-1)}=\frac{1}{S_{n}!}\prod_{1 \leq i \leq n} \frac{1}{S_{i-1}},\]
where the last equality is readily checked by induction. This establishes \eqref{eq:toshow} and completes the proof. 
\end{proof}

\subsection{A connection with increasing binary trees}
\label{lien binaire croissant}

\begin{table}[htbp]\caption{Table of the main notation and symbols introduced in Section \ref{lien binaire croissant}.}
\centering
\begin{tabular}{c c p{12cm} }
\toprule
$\mathfrak B_{k,n}(A)$ &&\makecell[cl]{set of rooted  binary plane trees with $2n-1$ vertices, such that \\$k$ of the leaves are labelled with the letter $a$ and the $2n-k-1$ \\ other vertices are labelled in a one-to-one manner with the \\integers of $A$}\\
\hline
$\mathfrak B_{k,n}^+(A)$  && set of trees in $\mathfrak B_{k,n}(A)$ such that the vertices have increasing labels paths directed from the root. \\
\hline
$\mathfrak B_{k,n}^+$ &&$\mathfrak B_{k,n}^+(A)$ with $A=\llbracket1,2n-k-1\rrbracket$\\
\bottomrule
\end{tabular}
\label{tab:secbinaire}
\end{table}

It turns out that $\mathfrak T_{k,n}(A)$  is in bijection with {a certain set of} increasing plane binary trees, {which are rooted trees, 
with an order among the children of internal vertices (see Sec.~\ref{sous-section lien BGW} for a precise setting), and where every vertex has either zero or two children}.

Such a bijection is well-known in the case $k=0$ (see e.g.~\cite{KS78} and \cite[Sec.~6]{Dev87}). To make this connection explicit, we need some further notation. For $n\geq 1$, $0 \leq k \leq n$ and $A \subset \N$ with cardinality $\#A = 2n-k-1$ we define $\mathfrak B_{k,n}(A)$ to be the set of rooted  binary plane  trees with $2n-1$ vertices, such that $k$ of the leaves are labelled with the letter $a$ and the $2n-k-1$ other vertices are labelled in a one-to-one manner with the integers of $A$. We also define $\mathfrak B_{k,n}^+(A)$ to be the set of trees in $\mathfrak B_{k,n}(A)$ such that the vertices have increasing labels {along} paths directed {away} from the root. We set $\mathfrak B_{k,n}^+ = \mathfrak B_{k,n}^+(\llbracket1,2n-k-1\rrbracket)$ (see the right of Figure \ref{fig exemple d'image par b} {for an example of a tree in $\mathfrak B_{4,15}^+ $}).

We are now ready to define a function $\Phi$ from $\bigcup \mathfrak T_{k,n}(A)$ to $\bigcup \mathfrak B_{k,n}(A)$ where the unions are over all possible triplets $(k,n,A)$. The definition is done by induction.
If $\tau$ is {a} trivial tree with one vertex then {\sout{we simply}} set $\Phi(\tau)=\tau$.
If $\tau$ has at least 2 vertices, then let $u$ be the child of the root $\rho$ such that the edge between $\rho$ and $u$ has the smallest label, denoted by $\ell$, among all the edges adjacent to $\rho$. Denote by $\tau_1$ the sub-tree rooted in $u$ composed of all the descendants of $u$ and $\tau_2$ the sub-tree rooted in $\rho$ composed of all the vertices that are not in $\tau_1$. Then $\Phi(\tau)$ is the binary tree starting with a root, labelled $\ell$, with two children: the first child on the left gives birth to $\Phi(\tau_1)$ and the second one on the right gives birth to $\Phi(\tau_2)$ (see Figure \ref{fig definition de la bijection b} and \ref{fig exemple d'image par b}).

\begin{figure}[h!]
\begin{center}
\begin{tikzpicture}[scale = 0.8,
sommet/.style = {draw,circle, font=\scriptsize,inner sep=0,minimum size=15pt},
etiquete/.style = {font = \footnotesize}]

\node[] (t) at (0,-0.6) {$\tau$};
\node[] (t1) at (-1.2,2.4) {$\tau_1$};
\node[] (t2) at (0.6,1.2) {$\tau_2$};
\node[sommet] (0) at (0,0) {$\rho$};
\node[sommet] (1) at (-1.1,1.2) {$u$};
\node[etiquete] (l) at (-0.7,0.5) {$\ell$};

\draw (0) -- (1);
\draw (1) to[out=140,in=50,distance=4cm] (1);
\draw (0) to[out=30,in=100,distance=4cm] (0);

\end{tikzpicture}
\hspace{-8em}
\begin{tikzpicture}[scale = 0.8,
sommet/.style = {draw,circle, font=\scriptsize,inner sep=0,minimum size=15pt},
etiquete/.style = {font = \footnotesize}]

\node[] (bt) at (0,-0.6) {$\Phi(\tau)$};
\node[] (bt1) at (-1.1,2.1) {$\Phi(\tau_1)$};
\node[] (bt2) at (1.1,2.1) {$\Phi(\tau_2)$};
\node[sommet] (0) at (0,0) {$\ell$};
\node[sommet] (1) at (-1,1) {};
\node[sommet] (2) at (1,1) {};

\draw (1) -- (0) -- (2);
\draw (1) to[out=140,in=50,distance=4cm] (1);
\draw (2) to[out=130,in=40,distance=4cm] (2);

\end{tikzpicture}
\caption{Illustration of the construction of the bijection $\Phi$ by induction.}
\label{fig definition de la bijection b}
\end{center}
\end{figure}
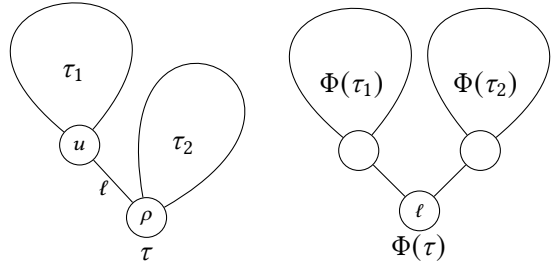

\begin{figure}
\begin{center}
\begin{tikzpicture}[scale = 0.9,
sommet/.style = {draw,circle, font=\scriptsize,inner sep=0,minimum size=15pt},
gele/.style = {fill=cyan},
etiquete/.style = {font = \scriptsize}]

\node[sommet] (0) at (0,0) {$a$};
\node[sommet,gele] (1) at (-1,1.5) {8};
\node[sommet] (2) at (1,1.5) {$a$};
\node[sommet,gele] (3) at (-1.7,3) {4};
\node[sommet] (4) at (-0.7,3) {$a$};
\node[sommet,gele] (5) at (0.7,3) {7};
\node[sommet] (6) at (1.7,3) {$a$};
\node[sommet,gele] (7) at (-0.7,4.5) {11};

\draw (0) --node[left,etiquete] {1} (1);
\draw (0) --node[right,etiquete] {3} (2);
\draw (1) --node[left,etiquete] {2} (3);
\draw (1) --node[right,etiquete] {6} (4);
\draw (2) --node[left,etiquete] {5} (5);
\draw (2) --node[right,etiquete] {10} (6);
\draw (4) --node[left,etiquete] {9} (7);

\end{tikzpicture}
\hspace{2em}
\begin{tikzpicture}[scale = 0.9,
sommet/.style = {draw,circle, font=\scriptsize,inner sep=0,minimum size=15pt},
gele/.style = {fill=cyan},
etiquete/.style = {font = \footnotesize}]

\node[sommet] (0) at (0,0) {1};
\node[sommet] (1) at (-1.5,1) {2};
\node[sommet] (2) at (1.5,1) {3};
\node[sommet,gele] (3) at (-2.5,2) {4};
\node[sommet] (4) at (-1,2) {6};
\node[sommet] (5) at (1,2) {5};
\node[sommet] (6) at (2.5,2) {$a$};
\node[sommet] (7) at (-1.5,3) {9};
\node[sommet,gele] (8) at (-0.5,3) {8};
\node[sommet,gele] (9) at (0.5,3) {7};
\node[sommet] (10) at (1.5,3) {10};
\node[sommet,gele] (11) at (-2,4) {11};
\node[sommet] (12) at (-1,4) {$a$};
\node[sommet] (13) at (1,4) {$a$};
\node[sommet] (14) at (2,4) {$a$};

\draw (11)--(7)--(4)--(1)--(0)--(2)--(5)--(10)--(13);
\draw (7)--(12);
\draw (8)--(4);
\draw (3)--(1);
\draw (2)--(6);
\draw (5)--(9);
\draw (10)--(14);

\end{tikzpicture}
\caption{An example of a tree $\tau \in \mathfrak T_{4,8}^+$ on the left and its image by the bijection $\Phi(\tau) \in \mathfrak B_{4,15}^+$ on the right.}
\label{fig exemple d'image par b}
\end{center}
\end{figure}
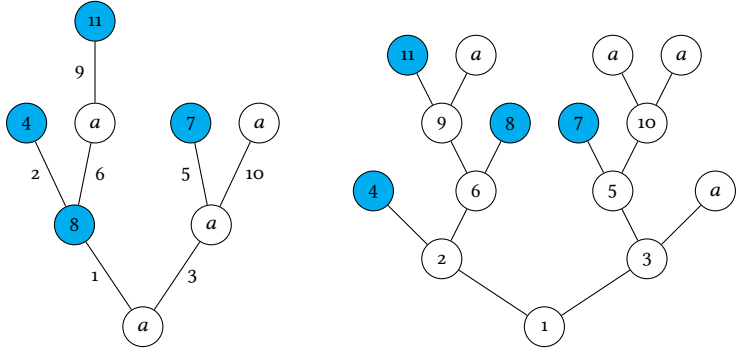

\begin{proposition}
The function $\Phi$ defined above is a bijection between $\mathfrak T_{k,n}^+$ and $\mathfrak B_{k,n}^+$.
\end{proposition}

\begin{proof}
It suffices to construct explicitly the inverse bijection $\Psi$ of $\Phi$. We proceed by induction. For a tree $\tau$ with a single vertex let $\Psi(\tau)=\tau$. If $\tau$ is a binary tree with at least 2 children and a root labelled $\ell$, let $\tau_1$ be the tree on the left and  let $\tau_2$ be the tree on the right of the root. Then $\Psi(\tau)$ is obtained by attaching the root of $\Psi(\tau_1)$ to the root of $\Psi(\tau_2)$ with an edge labelled $\ell$; the root of $\Psi(\tau_2)$ becomes the root of $\Psi(\tau)$. It is a simple matter to check that $\Psi$ is indeed the inverse bijection of $\Phi$.
\end{proof}

As a consequence, we can identify the cardinality of $\mathfrak T_{0,n}^+$, which is the set of all possible trees with $0$ active vertices and $n$ frozen vertices which can be obtained by Algorithm \ref{algo1}, {using the fact that $ \vert\mathfrak B_{0,n}^+\vert$ is counted by the so-called tangent numbers \cite{Don75}.}

\begin{corollary}
The cardinality of $\mathfrak T_{0,n}^+$ is given by
\[
\# \mathfrak T_{0,n}^+ = \frac{4^n(4^n-1)\vert B_{2n}\vert}{2n} = \tan^{(2n-1)}(0)
\]
where $B_{n}$ is the $n$-th Bernoulli number and $\tan^{(n)}$ stands for the $n$-th derivative of the tangent function. This sequence of numbers is given by OEIS A000182.
\end{corollary}

\section{Local limits}
\label{sec:local}

Let $\mathfrak T$ be the set of all rooted, vertex-labelled, edge-labelled, and locally finite (meaning that every vertex has finite degree) trees. For every $ {t} \in \mathfrak T$ {and} $ h \in \N$ let $\llbracket {t} \rrbracket_{h}$ be
the finite rooted vertex-labelled edge-labelled tree
obtained from ${t}$ by
 keeping only the vertices at distance at most $h$ from the root vertex
 together with the edges between them, and their labels.
It is standard to construct a metric  $\dloc$ on $\mathfrak T$ such that the space $(\mathfrak T, \dloc)$ is Polish (i.e. separable and complete) and
${t}_{n} \rightarrow {t}$ for $\dloc$ if and only if for every $h \ge 1$, we have $\br{{t}_{n}}_h=\br{{t}}_h$  {(as rooted, vertex-labelled, edge-labelled trees)} for $n$ large enough. Similar metrics can be defined for other families of trees 
(plane and unlabelled, plane and vertex-labelled and edge-labelled) and will be also denoted by $\dloc$. {We refer to the associated topologies as ``the local topology''.}

\subsection{Proof of Theorem \ref{thm:cvlocale} } \label{proof:thm:cvlocale}

We begin by showing that conditions  {(II)} and  {(III)}  are equivalent. Since the sequence of trees $(\T_n(\Xb))_{n\geq 0}$ is increasing, it converges almost surely for $\dloc$, if and only if, almost surely, the degree of each vertex converges (to a finite value).

First, we show that every vertex degree converges a.s.~if and only if every vertex freezes a.s.~at some finite time. {Once a vertex is frozen its degree does not change, so if every vertex freezes a.s.~at some finite time then a.s.~every vertex degree converges. For the other implication, we show that if with positive probability a vertex does not eventually freeze, then with positive probability its degree does not converge.}   Set $\mathsf{P} = \{n \geq 1 : \X_n = +1\}$ and $\mathsf{M} = \{n \geq 1 : \X_n = -1\}$. Observe that 
\begin{equation}
\label{eq:sumP}\sum_{n \in \mathsf{P}}  \frac{1}{S_{n}(\Xb)} =+\infty.
\end{equation}
Indeed, since the walk $(S_n(\Xb))_{n \geq 0}$ never touches $0$ {we have} $n^{-1}\#(\mathsf P \cap \{1,\dots,n\}) {\geq 1/2}$ {for all $n \geq1$}. Thus $\sum_{n \in \mathsf{P}} 1/S_n(\Xb) \geq \sum_{n \in \mathsf{P}} 1/(n+1)$ diverges.

Now fix $n_{0} \in \mathsf{P}$ and {denote by $v$ the vertex} that appears at time $n_{0}$. Let ${E}$ be the event ``$v$ never freezes''. We show that $\P({E})>0$ implies that with positive probability the degree of $v$ does not converge. To this end, for $n \in \mathsf{P}$, $n \geq n_{0}$, let $E_{n}$ denote the event  ``the degree of $v$ increases by one at time $n$''.
Then, since the event $E_{n}$ involves times in $\mathsf{P}$ while the event ${E}$ involves times in $\mathsf{M}$, conditionally given ${E}$ the events $(E_{n})_{n \geq 1}$ are independent and $\P(E_{n}|{E})= 1/S_{n}(\Xb)$. Thus, by \eqref{eq:sumP} and   Borel-Cantelli,  $\P(\textrm{the degree of } v  \textrm{ diverges as }n \rightarrow \infty |{E})=1$. Thus with positive probability the degree of $v$ does not converge.

Now, to show that {(II)} and {(III)} are equivalent, it remains to show that, almost surely, the fact that all the vertices freeze at some point is equivalent to the divergence of the sum $\sum_{n \in \mathsf{M}} 1/S_n(\Xb)$. Fix a vertex $v$ appearing at time $n_{0} \geq 0$. The probability that it never freezes is $\prod_{n \in \mathsf{M}, n > n_{0}} (1-1/S_n(\Xb))$ which converges towards $0$ if and only if $\sum_{n \in \mathsf{M}} 1/S_n(\Xb) = +\infty$. 

{Now we prove that  {(II)} and {(III)} imply  {(I)}}. Let $t$ be a finite tree with no vertex labelled $a$ (so all {vertices are} frozen). If $m$ is the biggest label of $t$ then for all $h$ and $n$ the event ``$\br{\T_n(\Xb)}_h = t$'' depends only on the first $m$ steps of $\Xb$. Therefore for all $n$ large enough $\proba{\br{\T_n(\Xb)}_h = t} = \proba{\br{\T_n(\Xb^n)}_h = t}$.
If $\T_n(\Xb)$ converges locally to $\T_\infty(\Xb)$, then, by {(III)}, all the vertices of $\T_\infty(\Xb)$ are frozen so we deduce that $\T_n(\Xb^n)$ also converges locally {in distribution} to $\T_\infty(\Xb)$.

{Finally we show that {(I)} implies  {(II)}.} Assume {(I)} and suppose that $\T_n(\Xb)$ does not converge locally {almost surely}. Then there exists $i$ such that $\X_i=+1$ and the vertex $v_i$ added at time $i$ in $\T_n(\Xb)$ has a degree converging towards $+\infty$ with positive probability. In other words
\[
\proba{\deg(v_i,\T_n(\Xb)) \rightarrow +\infty} = \lim_{m\to\infty} \proba{\bigcup_{n\geq i} \{\deg(v_i,\T_n(\Xb))\geq m \}} >0.
\]
Let $\varphi$ be an increasing function such that for all $k \leq n$, $\Xb_k^{\varphi(n)} = \Xb_k$. There is an obvious coupling such that $\T_n(\Xb) = \T_n(\Xb^{\varphi(n)})$ for all $n$. Therefore, for all $m$,
\[
\bigcup_{n\geq i} \{\deg(v_i,\T_n(\Xb))\geq m \} = \bigcup_{n\geq i} \{\deg(v_i,\T_n(\Xb^{\varphi(n)}))\geq m \},
\]
and
\[
\lim_{m\to\infty} \proba{\bigcup_{n\geq i} \{\deg(v_i,\T_{\varphi(n)}(\Xb^{\varphi(n)}))\geq m \}} \geq \lim_{m\to\infty} \proba{\bigcup_{n\geq i} \{\deg(v_i,\T_n(\Xb^{\varphi(n)}))\geq m \}} >0.
\]
We deduce that {under this coupling} the degree of $v_i$ in $\T_{\varphi(n)}(\Xb^{\varphi(n)})$ tends to $\infty$ with positive probability so $\T_n(\Xb^n)$ can't converge locally.\qed

\subsection{A connection with Geometric Bienaym\'e  trees: Proof of Theorem \ref{thm:BGW}}
\label{sous-section lien BGW}
\label{sec:BGW}

For the proof of Theorem \ref{thm:BGW}, we need the {Ulam--Harris} formalism of plane trees. Let $ \mathcal{U}$ be the set of labels defined by 
\[ \mathcal{U}=\bigcup_{n=0}^{\infty} \N^n,\]
where, by convention, $(\N)^0=\{\varnothing\}$. In other words, an element of $ \mathcal{U}$  is a (possibly empty) sequence $u=u_1 \cdots u_j$ of positive integers. When  $u=u_1
\cdots u_j$ and $v=v_1 \cdots v_k$ are elements of $ \mathcal{U}$, we let  $uv=u_1
\cdots u_j v_1 \cdots v_k$ be the concatenation of $u$ and $v$. In particular,  $u \varnothing=\varnothing u = u$. Finally, a {(Ulam--Harris)} \emph{plane tree} is a  subset of $ \mathcal{U}$ satisfying the following three conditions: 
\begin{enumerate}
\item[(i)] $\varnothing \in \tau$,
\item[(ii)] if $v \in \tau$ and $v=uj$ for {some} $j \in \N$, then $u
\in \tau$,
\item[(iii)] for every  $u \in \tau$, there exists an integer $k_u(\tau) \geq 0$ such that for every $j \in \N$, $uj \in \tau$ if and only if $1 \leq j \leq k_u(\tau)$.
\end{enumerate}
If $\tau$ is a plane tree and $u \in \tau$ let $\theta_{u} \tau = \{v \in \mathcal{U}: uv \in \tau\}$ be the subtree of descendants of $u$. We finally let $\mathbb{T}$ be the set of all plane trees (observe that such trees may be infinite, but are always locally finite).

Given a probability distribution $\mu$ on $\mathbb{Z}_{+}$, the law $\P_{\mu}$ of a Bienaym\'e tree with offspring distribution $\mu$ (sometimes known as a Galton--Watson tree, but we prefer to use the terminology suggested in \cite{ABHK21}) is a probability measure on  $\mathbb{T}$ that can be characterized by the following two properties (see e.g.~\cite{Nev86} for more general statements):
\begin{itemize}
\item[(i)] $\P_{\mu}(k_{\varnothing}=j)=\mu(j)$ for $j \in \Z_{+}$,
\item[(ii)] for every $j \geq 1$ with $\mu(j)>0$, the shifted trees $\theta_{1}  \tau, \ldots, \theta_{j} \tau$ are independent 
under the conditional probability  $\P_{\mu}({\, \cdot \,} | k_{\varnothing}=j)$ and their conditional
distribution is $\P_{\mu}$.
\end{itemize}

The following simple result gives a characterization of Bienaymé trees with geometric offspring distribution in terms of an invariance property involving subtrees (since we have not managed to find a reference in the literature we include the {short} proof for completeness).

\begin{lemma}
\label{lem:BGW}
Let $p \in [0,1)$. Let $ \mathcal{T}$ be a random plane tree such that $\proba{k_{\varnothing}(\mathcal{T})=0 }=1-p$ and under the conditional probability $\P(\, \cdot \, | k_{\varnothing}(\mathcal{T})>0)$ the two trees  $\theta_{1} \mathcal{T}$ and $ \mathcal{T} \backslash \theta_{1} \mathcal{T}$ are independent with conditional distribution equal to the law of $ \mathcal{T}$. Then $ \mathcal{T}$  is a Bienaym\'e tree with offspring distribution $\mu$ given by $\mu(k)= (1-p) p ^{k}$ for $k \geq 0$.  
\end{lemma}

\begin{proof}
First, for $j \geq 1$ we have $\proba{k_{\varnothing}(\mathcal{T}) \geq j}= \proba{k_{\varnothing}(\mathcal{T})>0} \proba{ k_{\varnothing}( \mathcal{T} \backslash \theta_{1} \mathcal{T}) \geq j-1 | k_{\varnothing}(\mathcal{T})>0}$, which implies $\proba{k_{\varnothing}(\mathcal{T}) \geq j}= p \proba{k_{\varnothing}(\mathcal{T}) \geq j-1}$.  Thus $\proba{k_{\varnothing}(\mathcal{T}) = k}=(1-p) p ^{k}$ for $k \geq 0$. 

Second, we argue by induction. Take $n \geq 1$ and assume that  the shifted trees $\theta_{1}  \mathcal{T}, \ldots, \theta_{n} \mathcal{T}$ are independent  under the conditional probability $\P( \, \cdot \, | k_{\varnothing}(\mathcal{T})=n)$  with conditional distribution equal to the law of $ \mathcal{T}$. Then for $f_{1},\ldots,f_{n+1} \geq 0$ measurable
\[
\esp{f_{1}(\theta_{1}  \mathcal{T}) \cdots f_{n+1}(\theta_{n+1} \mathcal{T})   | k_{\varnothing}(\mathcal{T})=n+1}= \esp{f_{1}(\theta_{1}  \mathcal{T}) g_{n} ( \mathcal{T} \backslash \theta_{1} \mathcal{T})   | k_{\varnothing}(\mathcal{T})>0}\]
with $g_{n}(\tau)=\mathbbm{1}_{k_{\varnothing}(\tau)=n} f_{2}(\theta_{1} \tau) \cdots f_{n+1}(\theta_{n} \tau) \proba{k_{\varnothing}(\mathcal{T})>0}/ \proba{k_{\varnothing}(\mathcal{T})=n+1}$. Thus, by assumption, we have
$\esp{f_{1}(\theta_{1}  \mathcal{T}) \cdots f_{n+1}(\theta_{n+1} \mathcal{T})   | k_{\varnothing}(\mathcal{T})=n+1}=\esp{f_{1}(\mathcal{T})} \esp{g_{n}(\mathcal{T})}$. But
\[ \esp{g_{n}(\mathcal{T})}=\frac{\proba{k_{\varnothing}(\mathcal{T})=n} \proba{k_{\varnothing}(\mathcal{T})>0}}{\proba{k_{\varnothing}(\mathcal{T})=n+1}} \esp{f_{2}(\theta_{1}  \mathcal{T}) \cdots f_{n+1}(\theta_{n} \mathcal{T})   | k_{\varnothing}(\mathcal{T})=n},
\]
which is equal to $\esp{f_{2}(\theta_{1}  \mathcal{T}) \cdots f_{n+1}(\theta_{n} \mathcal{T})   | k_{\varnothing}(\mathcal{T})=n}$ thanks to the first step. The desired result follows from the induction hypothesis.
\end{proof}

For every tree $\tau \in \mathfrak T_{k,n}^+(A)$, we define $\tilde \tau$ to be the unique element of $\mathfrak T_{k,n}^+$ obtained by shifting the labels of $\tau$ {to make them consecutive}; specifically if $f : \{1,\dots,2n-k-1\} \rightarrow A$ denotes the unique increasing bijection, an edge/vertex is labeled $f(i)$ in $\tau$ if and only if its corresponding edge/vertex is labeled $i$ in $\tilde \tau$ (see Figure \ref{fig preuve thm:BGW}).

We are now ready to establish Theorem \ref{thm:BGW}.

\begin{proof}[Proof of Theorem \ref{thm:BGW}]
First, we claim that almost surely the sequence $(\T_n(\XXb))_{n\geq 0}$ converges locally towards a tree denoted by $\T_\infty$.  {Observe that if $(S_n)_{n\geq 0}$ touches $0$, then the sequence of trees $(\T_n(\XXb))_{n\geq 0}$ is almost surely eventually constant and thus converges.  If $p \leq 1/2$, $(S_n)_{n\geq 0}$ touches $0$ almost surely at some point, so the sequence of trees $(\T_n(\XXb))_{n\geq 0}$ converges almost surely.} If $p > 1/2$, conditionally on the fact that $(S_n)_{n\geq 0}$ never touches $0$, this walk satisfies the condition of Theorem \ref{thm:cvlocale}, so $(\T_n(\XXb))_{n\geq 0}$ has a local limit. We conclude that, for every value of $p \in [0,1)$, almost surely the sequence $(\T_n(\XXb))_{n\geq 0}$ converges locally towards a tree $\T_\infty$.

We shall first consider $\T_\infty$ conditioned on having at least $2$ vertices (which amounts to {conditioning} on the event $\{\XX_1=+1\}$). Conditionally given $\{\XX_1=+1\}$, $\T_n(\XXb)$ is composed of two trees: the tree $\T_n^1(\XXb)$ representing all the descendants of the first child of the root of $\T_n(\XXb)$, and the tree $\T_n^2(\XXb)$ representing all the other vertices, including the root of $\T_n(\XXb)$ (see Figure \ref{fig preuve thm:BGW}). 

The main  idea is then to check that if one of the trees $\T_n^1(\XXb)$ or $\T_n^2(\XXb)$ has at least one active vertex, then it will eventually almost surely evolve {(i.e. its degree changes or it freezes)}. This will show that the two trees $\tilde \T_n^1(\XXb)$ or $\tilde \T_n^2(\XXb)$ evolve independently with the same transition probabilities, which in turn will entail that conditionally given $\{\XX_1=1\}$, $(\tilde \T_n^1(\XXb))_{n\geq 1}$ and $(\tilde \T_n^2(\XXb))_{n\geq 1}$ converge locally, jointly, towards two independent trees distributed like $\T_\infty$.  

Let us now make this discussion more formal. Notice that $(\T_n(\XXb))_{n\geq 0}$ is a Markov chain in the state space $E \coloneqq \bigcup_{n \geq 1}\bigcup_{0\leq k \leq n} \mathfrak T_{{k,n}}^+$ starting at the trivial tree, denoted by $t_\circ$, composed of a single active vertex. Denote by $\alpha$ its transition matrix. The absorbing states are exactly the trees {in $\bigcup_{n \geq 1} \mathfrak T_{{0,n}}^+$} with no active vertices. For $n\geq 2$, let $I_n$ be the random variable which is equal to $1$ when $S_k=0$ for some $k \leq n-1$ or when the $n$-th action (corresponding to the step $\XX_{n}$) occurs in the tree $\T_{n-1}^1(\XXb)$ and $2$ if it occurs in the tree $\T_{n-1}^2(\XXb)$.  {For example, when} $\XX_{n}=-1$ and $S_k>0$ for all $k \leq n-1$, an active vertex of $\T_{n-1}(\XXb)$, chosen uniformly at random, will freeze. If this vertex is chosen in $\T_{n-1}^1(\XXb)$, then $I_n = 1$, otherwise, it is chosen in $\T_{n-1}^2(\XXb)$ and $I_n = 2$. By convention we also set $I_1 = 1$. Denote by $a(t)$ the number of active vertices of a tree $t \in E$. 

The previous discussion shows that, conditionally given the event $\{\XX_1 = 1\}$, $\left( \tilde \T_n^1(\XXb), \tilde \T_n^2(\XXb), I_n \right)_{n\geq 1}$ is a Markov chain in the state space $E^2\times\{1,2\}$, starting at $(t_\circ,t_\circ,1)$, with transitions given by 
\begin{align*}
(x,y,i) \rightarrow (x',y,1) \qquad & \text{with probability } ~~k(x,y)\alpha(x,x') \\ 
(x,y,i) \rightarrow (x,y',2) \qquad & \text{with probability } ~~(1-k(x,y))\alpha(y,y').
\end{align*}
for every $x,x',y,y' \in E$ and $i \in \{1,2\}$ where 
\[
k(x,y) \coloneqq \frac{a(x)}{a(x)+a(y)} ~~ \text{ if }a(x)+a(y)>0 \qquad \text{ and } \qquad k(x,y)\coloneqq1 ~\text{ otherwise}.
\]

Let $x \in E$ be a non-absorbing state (i.e.~$a(x)>0$). Let $y \in E$ and $n \geq 1$. We show that conditionally given $\{\tilde \T_n^1(\XXb) = x \text{ and }\tilde \T_n^2(\XXb) = y\}$ the probability that $I_k = 1$ for some $k>n$ is $1$ {and that the probability that $I_k = 2$ for some $k>n$ is $1$}. The forthcoming Lemma  \ref{lemme chaines de markov altern\'ees}  then entails that conditionally given $\{\XX_1=1\}$, the trees $(\tilde \T_n^1(\XXb))_{n\geq 1}$ and $(\tilde \T_n^2(\XXb))_{n\geq 1}$ converge locally, jointly, towards two independent trees distributed like $\T_\infty$.

Now, {we show that that conditionally given $\{\tilde \T_n^1(\XXb) = x \text{ and }\tilde \T_n^2(\XXb) = y\}$ the probability that $I_k = 1$ for some $k>n$ is $1$ (for $I_{k}=2$ the argument is the same)}. Observe that since $a(x)>0$, we have $S_{k}>0$ for every $1 \leq k \leq n$. If $S_k = 0$ for some $k>n$, this implies that both $\T_k^1(\XXb)$ and $\T_k^2(\XXb)$ have no active vertices left. So in particular, there must be an integer $n < \ell \leq k$ such that $I_\ell =1$. Conditionally given $(S_{k})_{k \geq n}$, if $S_k > 0$ for all $k>n$, then the probability that $I_k = 2$ for all $k > n$ is
\[
\prod_{k\geq n} \frac{S_k - a(x)}{S_k} \leq \prod_{k \geq n} \left( 1 - \frac{1}{S_k}\right) \leq \prod_{k \geq n} \left( 1 - \frac{1}{k+1}\right) = 0.
\]
Therefore we can apply the forthcoming Lemma \ref{lemme chaines de markov altern\'ees} and deduce that, conditionally given $\{\XX_1=1\}$, the trees $(\tilde \T_n^1(\XXb))_{n\geq 1}$ and $(\tilde \T_n^2(\XXb))_{n\geq 1}$ converge locally, jointly, towards two independent trees distributed like $\T_\infty$.

To conclude, for a {vertex and edge} labelled locally finite tree $\tau$ denote by  $\overline \tau$ to be the plane tree obtained from $\tau$ by first ordering the vertices in such a way that the labels of edges connecting vertices to their children are increasing from left to right, and second erasing all the labels (for example the vertices of the tree on the left of Fig.~\ref{fig exemple d'image par b} are ordered in such a way).
By continuity of the operation $\tau \mapsto \overline{\tau}$ on the set of labelled locally finite trees with respect to the local topology, we conclude that conditionally on the fact that $\overline \T_\infty$ is not a single vertex, the two trees $\overline \T_\infty^1$ and $\overline \T_\infty^2$ are independent copies of $\overline \T_\infty$.   Since $\overline{\T}_\infty$  has only one vertex with probability $1-p$ (this happens if and only if $\XX_{1}=-1$), the desired result follows from Lemma \ref{lem:BGW}.
\end{proof}

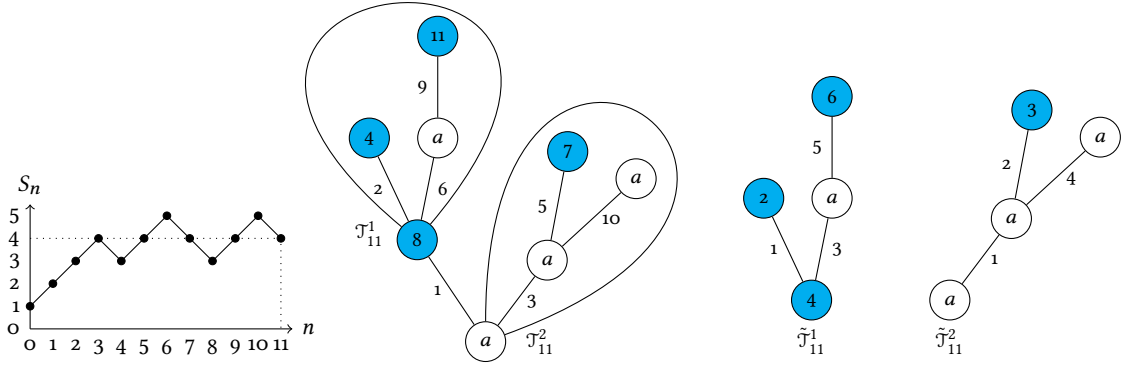
\begin{figure}
\begin{center}
\begin{tikzpicture}[scale = 0.3]
\draw[->] (0,0) -- (0,5.5);
\draw[->] (0,0) -- (11.5,0);

\node[above,font=\footnotesize] at (0,5.5) {$S_n$};
\node[right,font=\footnotesize] at (11.5,0) {$n$};

\node[draw,circle,fill,inner sep = 1pt] at (0,1) {};
\node[draw,circle,fill,inner sep = 1pt] at (1,2) {};
\node[draw,circle,fill,inner sep = 1pt] at (2,3) {};
\node[draw,circle,fill,inner sep = 1pt] at (3,4) {};
\node[draw,circle,fill,inner sep = 1pt] at (4,3) {};
\node[draw,circle,fill,inner sep = 1pt] at (5,4) {};
\node[draw,circle,fill,inner sep = 1pt] at (6,5) {};
\node[draw,circle,fill,inner sep = 1pt] at (7,4) {};
\node[draw,circle,fill,inner sep = 1pt] at (8,3) {};
\node[draw,circle,fill,inner sep = 1pt] at (9,4) {};
\node[draw,circle,fill,inner sep = 1pt] at (10,5) {};
\node[draw,circle,fill,inner sep = 1pt] at (11,4) {};

\draw (0,1)--(1,2)--(2,3)--(3,4)--(4,3)--(5,4)--(6,5)--(7,4)--(8,3)--(9,4)--(10,5)--(11,4);

\node[left,font=\footnotesize] at (0,0) {0};
\node[left,font=\footnotesize] at (0,1) {1};
\node[left,font=\footnotesize] at (0,2) {2};
\node[left,font=\footnotesize] at (0,3) {3};
\node[left,font=\footnotesize] at (0,4) {4};
\node[left,font=\footnotesize] at (0,5) {5};

\node[below,font=\footnotesize] at (0,0) {0};
\node[below,font=\footnotesize] at (1,0) {1};
\node[below,font=\footnotesize] at (2,0) {2};
\node[below,font=\footnotesize] at (3,0) {3};
\node[below,font=\footnotesize] at (4,0) {4};
\node[below,font=\footnotesize] at (5,0) {5};
\node[below,font=\footnotesize] at (6,0) {6};
\node[below,font=\footnotesize] at (7,0) {7};
\node[below,font=\footnotesize] at (8,0) {8};
\node[below,font=\footnotesize] at (9,0) {9};
\node[below,font=\footnotesize] at (10,0) {10};
\node[below,font=\footnotesize] at (11,0) {11};

\draw[dotted] (0,4)--(11,4)--(11,0);
\end{tikzpicture}
\hspace{-9em}
\begin{tikzpicture}[scale = 0.9,
sommet/.style = {draw,circle, font=\scriptsize,inner sep=0,minimum size=15pt},
gele/.style = {fill=cyan},
etiquete/.style = {font = \scriptsize}]

\node[sommet] (0) at (0,0) {$a$};
\node[sommet,gele] (1) at (-1,1.5) {8};
\node[sommet] (2) at (0.9,1.2) {$a$};
\node[sommet,gele] (3) at (-1.7,3) {4};
\node[sommet] (4) at (-0.7,3) {$a$};
\node[sommet,gele] (5) at (1.2,2.8) {7};
\node[sommet] (6) at (2.2,2.4) {$a$};
\node[sommet,gele] (7) at (-0.7,4.5) {11};

\draw (0) --node[left,etiquete] {1} (1);
\draw (0) --node[right,etiquete] {3} (2);
\draw (1) --node[left,etiquete] {2} (3);
\draw (1) --node[right,etiquete] {6} (4);
\draw (2) --node[left,etiquete] {5} (5);
\draw (2) --node[right,etiquete] {10} (6);
\draw (4) --node[left,etiquete] {9} (7);

\draw (1) to[out=140,in=50,distance=6.2cm] (1);
\draw (0) to[out=90,in=20,distance=6.2cm] (0);

\node[etiquete] (t1) at (-1.7,1.6) {$\T_{11}^1$};
\node[etiquete] (t2) at (0.8,0) {$\T_{11}^2$};
\end{tikzpicture}
\hspace{-6em}
\begin{tikzpicture}[scale = 0.9,
sommet/.style = {draw,circle, font=\scriptsize,inner sep=0,minimum size=15pt},
gele/.style = {fill=cyan},
etiquete/.style = {font = \scriptsize}]
\node[sommet,gele] (1) at (0,0.6) {4};
\node[sommet,gele] (3) at (-0.7,2.1) {2};
\node[sommet] (4) at (0.3,2.1) {$a$};
\node[sommet,gele] (7) at (0.3,3.6) {6};

\draw (1) --node[left,etiquete] {1} (3);
\draw (1) --node[right,etiquete] {3} (4);
\draw (4) --node[left,etiquete] {5} (7);

\node[etiquete] (t1) at (0,0) {$\tilde \T_{11}^1$};
\end{tikzpicture}
\hspace{2em}
\begin{tikzpicture}[scale = 0.9,
sommet/.style = {draw,circle, font=\scriptsize,inner sep=0,minimum size=15pt},
gele/.style = {fill=cyan},
etiquete/.style = {font = \scriptsize}]
\node[sommet] (0) at (0,0.6) {$a$};
\node[sommet] (2) at (0.9,1.8) {$a$};
\node[sommet,gele] (5) at (1.2,3.4) {3};
\node[sommet] (6) at (2.2,3) {$a$};

\draw (0) --node[right,etiquete] {1} (2);
\draw (2) --node[left,etiquete] {2} (5);
\draw (2) --node[right,etiquete] {4} (6);

\node[etiquete] (t2) at (0,0) {$\tilde \T_{11}^2$};
\end{tikzpicture}
\caption{Illustration of the proof of Theorem \ref{thm:BGW}}
\label{fig preuve thm:BGW}
\end{center}
\end{figure}

\begin{lemma}
\label{lemme chaines de markov altern\'ees}
Let $(X_n)_{n\geq 0}$ and $(Y_{n})_{n \geq 0}$ be two Markov chains on the countable state space $E$ with the {same} transition matrix $\alpha$. Denote by $A { := \{x\in E : \alpha(x,x)=1\}}$ their set of absorbing states. Let $k : E\times F \rightarrow [0,1]$. {Define} a Markov chain $(U_n,V_n,I_n)_{n\geq 0}$ on the state space $E^{2} \times \{1,2\}$ with transition matrix given by, for all $x,x',y,y'  \in E$ and $i \in \{1,2\}$:
\begin{align*}
(x,y,i) \rightarrow (x',y,1) \qquad & \text{with probability } ~~k(x,y)\alpha(x,x') \\ 
(x,y,i) \rightarrow (x,y',2) \qquad & \text{with probability } ~~(1-k(x,y))\alpha(y,y').
\end{align*}
Suppose that, for all $(x,y,i) \in (E\setminus A) \times E \times \{1,2\}$, $\P_{(x,y,i)}(\exists n\geq 1, \, I_n=1)=1$. Similarly, suppose that for all $(x,y,i) \in E \times (E\setminus A) \times \{1,2\}$, $\P_{(x,y,i)}(\exists n\geq 1, \, I_n=2)=1$. Let $n_0=m_0=0$. For all $k \geq 1$, if the set $\{n : n > n_{k-1} \text{ and } I_n=1\}$ is non-empty, we define $n_k$ as its minimum, otherwise we set $n_k = n_{k-1}$. Similarly, if the set $\{m : m > m_{k-1} \text{ and } I_m=2\}$ is non-empty, we define $m_k$ as its minimum, otherwise we set $m_k = m_{k-1}$. Then, for every $(x,y) \in E^{2}$, $\left( (U_{n_k})_{k\geq 0},(V_{m_k})_{k\geq 0} \right)$ under $\P_{(x,y,1)}$ has the same law as $\left( (X_k)_{k\geq 0}, (Y_k)_{k\geq 0} \right)$ under $\P_x \otimes \P_y$.
\end{lemma}

\begin{proof}[Proof of Lemma \ref{lemme chaines de markov altern\'ees}]
We show that $\left( (U_{n_k})_{0\leq k \leq r},(V_{m_k})_{0 \leq k \leq s} \right)$ under $\P_{(x,y,1)}$ has the same law as \\ $\left( (X_k)_{0 \leq k \leq r}, (Y_k)_{0 \leq k \leq s} \right)$ under $\P_x \otimes \P_y$ by double induction on $(r,s)$. If $r=s=0$, the result is obvious. Suppose it is true for some pair $(r,s)$. We show it for $(r,s+1)$ (the case $(r+1,s)$ is similar). Let ${\mathbf x = (x_1,\dots,x_r) \in E^r}$, $\mathbf y = (y_1,\dots,y_s) \in E^s$ and $y_{s+1} \in E$. Write $\mathbf U \coloneqq (U_{n_k})_{1\leq k \leq r}$, $\mathbf V \coloneqq (V_{m_k})_{1 \leq k \leq s}$, $\mathbf X \coloneqq (X_{k})_{1 \leq k \leq s}$, $\mathbf Y \coloneqq (Y_{k})_{1 \leq k \leq s}$ and $\P \coloneqq \P_{(x,y,1)}$. Suppose that $y_s \notin A$, then $m_{s+1}>m_s$ and
\begin{align*}
\P\left( \mathbf U = \mathbf x,\mathbf V= \mathbf y, V_{m_{s+1}}=y_{s+1} \right)= \P\left( \mathbf U = \mathbf x,\mathbf V = \mathbf y \right) \P_{(x_r,y_s,1)}\left(\exists n \geq 1,\, I_1=\cdots=I_{n-1}=1,\,I_n=2,\,V_n=y_{s+1} \right).    
\end{align*}
By induction hypothesis, $\P\left( \mathbf U = \mathbf x,\mathbf V = \mathbf y \right) = \P_x\left( \mathbf X = \mathbf x\right)\P_y\left(\mathbf Y = \mathbf y \right)$. Moreover
\begin{align*}
\P_{(x_r,y_s,1)}\left(\exists n \geq 1,\, I_1=\cdots=I_{n-1}=1,\,I_n=2,\,V_n=y_{s+1} \right) &= \P_{(x_r,y_s,1)}\left(\exists n \geq 1, I_n=2\right) \alpha(y_s,y_{s+1})\\
&= \alpha(y_s,y_{s+1}).
\end{align*}
This shows the formula for $(r,s+1)$ in the case $y_s \notin A$. Now, if $y_{s} \in A$ then
\begin{align*}
\P\left( \mathbf U = \mathbf x,\mathbf V= \mathbf y, V_{m_{s+1}}=y_{s+1} \right) &= \P\left( \mathbf U = \mathbf x,\mathbf V = \mathbf y\right) \P_{(x_r,y_s,1)}\left(\forall m \geq 1, I_m = 1 \right) \mathbbm{1}_{y_s = y_{s+1}} \\
&\qquad+ \P\left( \mathbf U = \mathbf x,\mathbf V = \mathbf y\right) \P_{(x_r,y_s,1)}\left(\exists m \geq 1, I_m = 2 \right) \alpha(y_s, y_{s+1}) \\
&=\P_x\left( \mathbf X = \mathbf x\right)\P_y\left(\mathbf Y = \mathbf y, Y_{s+1}=y_{s+1} \right).  
\end{align*}
This completes the proof.
\end{proof}

\section{The height of uniform attachment trees with freezing}
\label{sec:distances}

We shall now study the geometry of  uniform attachment trees with freezing. 
 It will be convenient to work with $\T^n(\Xb)$ as built using Algorithm \ref{algo2}. Recall {from} Theorem~\ref{thm:samelaw} {that}: the only difference between $\T_{n}(\Xb)$  and $\T^{n}(\Xb)$ is that all the active vertices of $\T^{n}(\Xb)$ are labeled $a_{1}, \ldots, a_{S_{n}(\Xb)}$, while all the active vertices of $ \T_{n}(\Xb)$ are labelled $a$. In particular the graph structure of both trees is the same in law, so it is equivalent to establish our main results with $\T_{n}(\Xb)$ replaced with $\T^{n}(\Xb)$.

Also recall from \eqref{eq:labels} the definition of $\mathbb{V}_{n}(\Xb)$, which is a deterministic set representing the labels of the vertices of $\T^{n}(\Xb)$, and that by a slight abuse of notation we view  elements of $\mathbb{V}_{n}(\Xb)$ as  vertices of  $\T^{n}(\Xb)$.

An important ingredient {in our proofs} is Bennett's concentration inequality. Since it will be used multiple times, we state it here, {tailored} for our purpose:

\begin{proposition}[Bennett's inequality]
\label{prop:bennett}
Let $(Y_{i})_{1 \leq i \leq n}$ be independent Bernoulli random variables of respective parameters $(p_{i})_{1\leq i \leq n}$. Set $m_{n}=\sum_{i=1}^{n} p_{i}$ and assume that $m_{n}>0$. Then for every $t>0$:
\[\proba{\sum_{i=1}^{n} Y_{i} >m_{n}+ t} \leq \exp \left( - m_{n} g \left( \frac{t}{m_{n}} \right) \right) \qquad \textrm{and} \qquad \proba{\sum_{i=1}^{n} Y_{i}< m_{n}- t} \leq \exp \left( - m_{n} g \left( \frac{t}{m_{n}} \right) \right),\]
 where $g(u)=(1+u)\ln(1+u)-u$ for $u>0$.
\end{proposition}

\subsection{Uniform bounds on the height}
\label{ssec:height}

We keep the notation introduced before (see {in particular} Table \ref{tab:secdiscret}). In particular, $\Xb=(\X_n)_{n \geq 1}$ is a fixed sequence  of elements of $\{-1,1\}$, $S_{0}(\Xb) \coloneqq1$ and for every $n \geq 1$
\begin{equation}
S_n(\Xb) \coloneqq 1+ \sum_{i=1}^n \X_i; \qquad \tau(\Xb) \coloneqq \inf \{n \geq 1 : S_{n}(\Xb)=0\}.
\end{equation}

To lighten notation we will drop {again} the $\Xb$ in parenthesis (e.g.~we will write $S_{n}$ instead of $S_{n}(\Xb)$, $\tau$ instead of $\tau(\Xb)$, $\T^{n}$ instead of $\T^{n}(\Xb)$, $\V_{n}(\Xb)$ instead of $\V_{n}$, etc.)  {\sout{but we keep in mind that all the quantities depend on the sequence $\Xb$}}.
We first start with an estimate concerning  the height of $\T^n$ (recall its construction using Algorithm \ref{algo2} in Section~\ref{ssec:growthcoalescent}).

\begin{proposition}
\label{prop:bounds}
Fix $n \geq 1$ and assume that $ \tau>n$. Let $\Height(\T^n)$ denote the height of $\T^n$. Then, for every $t>0$,
\[
\proba{\Height(\T^n)-h^{+}_{n}>t} \leq \exp\left(-h^{+}_{n}\,g\left(\frac{t}{h^{+}_{n}}\right)+\ln(n)\right)
\]
and
\[
\proba{\Height(\T^n)-h^{+}_{n}<-t} \leq \exp\left(-h^{+}_{n}\,g\left(\frac{t}{h^{+}_{n}}\right)\right)
\]
where $g(u)=(1+u)\ln(1+u)-u$ for $u>0$.
\end{proposition}

\begin{proof}
Fix $u \in \V_{n}$. From \eqref{eq:H0}, if $(Y_i)_{1\leq i \leq n}$ are independent Bernoulli random variables of respective parameters $(1/S_i)_{1\leq i \leq n}$, we have
\[
\H^{n}_{0}(u) \quad \mathop{=}^{(d)} \quad \sum_{i=1}^{\birth_{n}(u)} Y_i\1_{\{\X_i=1\}}.
\]
Thus, by a union bound over all vertices of $\T^n$ we get
\[
\proba{\Height(\T^n)>h^{+}_n+t} \leq n\proba{\sum_{i=1}^n \left(Y_i-\frac{1}{S_i}\right)\1_{\{\X_i=1\}}>t}.
\]
In addition, when $u \in \V_{n}$ is an active vertex (so that  $\birth_{n}(u)=n$), using the fact that $  \H^{n}_{0}(u) \leq \Height(\T_n)$, we get 
\[
\proba{\Height(\T^n)< h^{+}_n-t} \leq \proba{\sum_{i=1}^n \left(Y_i-\frac{1}{S_i}\right)\1_{\{\X_i=1\}}<-t}.
\]
The desired {bounds} then {both follow} from Bennett's inequality.
\end{proof}

\subsection{A simple lower bound for  $\h^+_n$}

In view of Theorem \ref{thm:height}, which tells us that $\h^+_n$ is the  order of magnitude of $\Height(\T^n)$, it is useful to have estimates on $\h_{n}^{+}$. We give here a simple lower bound on  $\h^+_n$.

\begin{lemma}
\label{lemme minoration h_n^+}
Let $n \geq 1$. If $\tau > n$ then $\h^+_n \geq \ln(n/4)$.
\end{lemma}

\begin{proof}
Set $p_n \coloneqq \#\{1 \leq i \leq  n : \X_{i}=+1\}$ and $m_n \coloneqq \#\{1 \leq i \leq  n : \X_{i}=-1\}$. If there exists $1 \leq i \leq n-1$ such that $\X_{i}=-1$ and $\X_{i+1}=+1$, set $\Xb'=(\X_{1}, \ldots,\X_{i-1},+1,-1,\X_{i+2}, \X_{i+3},\ldots)$. Then it is a simple matter to check that $\h^+_n(\Xb) \geq \h^+_n(\Xb')$. By iteration, it readily follows that $\h^+_n(\Xb) \geq \h^+_n((+1,\ldots,+1,-1, \ldots,-1, \Xb_{n+1},\Xb_{n+2}, \ldots))$ where $+1$ appears $p_n$ times and $-1$ appears $m_n$ times. Since $\tau(\Xb)>n$ entails $p_n \geq n/2$, it follows that $
\h^+_n(\Xb) \geq \sum_{i=1}^{p_n } 1/(i+1) \geq \ln((p_n+1)/2) \ge \ln(n/4)$.
\end{proof}
\subsection{Limit theorems for the height}
\label{ssec:limitheightBennet}
We are now in position to establish Theorem \ref{thm:height}.

\begin{proof}[Proof of Theorem \ref{thm:height}]

We start with {(I)}. Recall that for an active vertex $u \in \A_n$ we always have $\birth_n(u)=n$. Therefore, by \eqref{eq:H0}
\[
\H(U^n) \quad \mathop{=}^{(d)} \quad \sum_{i=1}^{n} Y^{n}_i\1_{\{\X_i=1\}}.
\]
Using Bennett's inequality (Proposition \ref{prop:bennett}), for all $\e>0$
\[
\proba{\left\vert\frac{\H(U^n)}{\h^+_n}-1\right\vert>\e} \leq 2\exp\left(-\h_n^+\,g(\e)\right).
\]
This upper bound goes to $0$ as $n$ tends to infinity since $\h_n^+$ goes to infinity by Lemma \ref{lemme minoration h_n^+}. We deduce that the sequence $\left({\H(U^n)}/{\h^+_n}-1\right)_n$ converges to $0$ in probability. To conclude it suffices to show that it is bounded in $\Lp^p$ for all $p\geq 1$. Let $p\geq 1$, we have 
\[
\esp{\left\vert\frac{\H(U^n)}{\h^+_n}-1\right\vert^p} = p\int_0^{\infty} \proba{\left\vert\frac{\H(U^n)}{\h^+_n}-1\right\vert\geq u}u^{p-1}\d u. 
\]
Using Bennett's inequality again and by splitting the integral in half at $u=7$, the above quantity is bounded by
\[
7^p + 2p\int_7^{\infty}\exp\left(-\h^+_n\,g(u)\right)u^{p-1}\d u.
\]
Notice that for $u\geq 7 \geq e^2-1$, we have $g(u) \geq u$ thus
\[
\int_7^{\infty}\exp\left(-\h^+_n\,g(u)\right)u^{p-1}du \leq \int_7^{\infty}\exp\left(-\h^+_n\,u\right)u^{p-1}\d u = \frac{1}{(\h^+_n)^p} \int_{7\h^+_n}^{\infty} \exp(-u)u^{p-1}\d u.
\]
By an iterated integration by parts, the last quantity is $\mathcal{O}\left({\exp(-7\h^+_n)}/{\h^+_n}\right)$. Therefore the boundedness in $\Lp^p$ follows. 

We now establish {(II)} and {(III)}. Fix $C>1$. By Proposition \ref{prop:bounds}, for all $n\geq1$:
\[
\proba{\frac{\Height(\T^n)}{\h^+_n}>C} \leq \exp\left(-\h^+_n\,g(C-1) + \ln(n)\right).
\]
This upper bound goes to $0$ when $\ln(n) = o(\h^+_n)$ since $g(C-1)>0$. It also goes to $0$ when $C=e+\e$ since $g(e-1+\e)>1$ and by Lemma \ref{lemme minoration h_n^+}, $\h^+_n \geq \ln(n/4)$. Similarly,
\[\proba{ \frac{\Height(\T^n)}{\h^+_n} <  1-\e} \leq  \exp \left( -\h^+_n(\Xb)g \left(\e\right) \right)  \quad \mathop{\longrightarrow}_{n \rightarrow \infty} \quad 0.\]
This completes the proof of {(II)}. It also shows that the convergence
\[
\frac{\Height(\T^n)}{\h^+_n} \cvproba 1
\]
holds in probability when $\ln(n) = o(\h^+_n)$. To fully show {(III)} we just need to prove that the sequence $\left({\Height(\T^n)}/{\h^+_n}-1\right)_n$ is bounded in $\Lp^p$ for all $p\geq 1$ when $\ln(n) = o(\h^+_n)$. It is done using the same method as in the proof of {(I)} 
{by using the fact that for $n$ sufficiently large we have 
$$\proba{\frac{\Height(\T^n)}{\h^+_n}>C} \leq    \exp\left(-\h^+_n\,g(C-1)/2\right)$$
for every $C \geq2$.}
\end{proof}

\section{Regime with a linear amount of active vertices} \label{sec:Linear}

We shall now investigate the regime where the number of active vertices grows roughly linearly in the size of the trees: the main goal of this section is to establish Theorem \ref{thm:lineaire}. Here we assume that \eqref{eq:assumptions} is in force, namely {that} there exist $c\in (0,1]$ and a sequence $(A_n)_{n\in \N}$ of positive numbers such that $A_n=o(\log n)$ as $n\to\infty$ and {such} that
\[
\lim_{\e \to 0} \limsup_{n\to \infty} \max_{A_n \leqslant i \leqslant \e n} \left| \frac{S^n_i}{i}-c \right| = 0
\qquad \text{and} \qquad \forall \e>0, \enskip
\liminf_{n\to \infty} \min_{\e n \leqslant i \leqslant n } \frac{S^n_i}{n} >0.
\]

\subsection{An asymptotic equivalent for $\h_{n}^{+}$}

Here we show the convergence of Theorem \ref{thm:lineaire} {(I)}, namely:
\begin{equation}
\label{cv hn cas lineaire}
\displaystyle \frac{\h^+_n}{\ln n} \xrightarrow[n\to\infty]{} \frac{c+1}{2c}.
\end{equation}
To simplify notation, for all $n \in \N$ and for all $\e >0$, set
\[
\eta_n(\e) \coloneqq \max_{A_n \leqslant i \leqslant \e n } \left \lvert \frac{S^n_i}{i} -c \right\rvert .
\]
By \eqref{eq:assumptions}, we know that $\limsup \eta_n(\e) \to 0$ when $\e \to 0$.
\begin{lemma}
\label{lem:estimee uniforme sur h_n+}
Suppose that the assumptions \eqref{eq:assumptions} hold with $A_n\to\infty$ when $n\to\infty$. Then, for all $C>1$
\begin{equation}
\label{eq:estimee uniforme sur h_n+}
\lim_{\e \to 0} \limsup_{n\to \infty} \max_{\substack{A_n \leqslant a< b \leqslant \e n\\aC<b}} \left| \frac{1}{\ln(b/a)}\sum_{i=a+1}^b \frac{1}{S_i^n}\mathbbm 1_{\{\X_i^n=1\}} - \frac{{c+1}
}{2c} \right| = 0.
\end{equation}
\end{lemma}

\begin{proof}[Proof of Lemma \ref{lem:estimee uniforme sur h_n+}]
We  prove \eqref{eq:estimee uniforme sur h_n+} with a double bound. First, let us start with the upper bound. Let $\e\in(0,1)$ and $\beta>1$. For all $A_n\leq a<b\leq \e n$: 
\[
\sum_{i=a+1}^{b}\frac{1}{S_i^n}\mathbbm{1}_{\{\X_i^n=1\}}
=  {\sum_{a< i\leqslant\beta^{\ell}} \frac{1}{S_i^n}\mathbbm{1}_{\{\X_i^n=1\}}}+\sum_{k=1+\ell}^{m}\sum_{\beta^{k-1}<i\leqslant\beta^k}\frac{1}{S_i^n}\mathbbm{1}_{\{\X_i^n=1\}} + \sum_{\beta^{m}<i\leqslant b} \frac{1}{S_i^n}\mathbbm{1}_{\{\X_i^n=1\}}
\]
where $\ell \coloneqq \lceil\ln a/\ln\beta\rceil$, $m \coloneqq \lfloor\ln b/\ln\beta\rfloor$ and where, in all the sums over $i$, it is implicit that $i$ is an integer. Let $P_n \coloneqq \{1\leq i \leq n : \X_i=1\}$ and $M_n \coloneqq \{1\leq i \leq n : \X_i=-1\}$. 
Note that for all $1\leq i<j\leq n$, we have
\[
\#(P_n\cap[i+1,j])-\#(M_n\cap[i+1,j]) = S_{j}^n-S_{i}^n, \quad \#(P_n\cap[i+1,j])+\#(M_n\cap[i+1,j])=j-i.
\]
Moreover, when $A_{n}\leq i < j \leq \e n$, we have $(j-i)c-2j \eta_n(\e)\le S_{j}^n-S_{i}^n\le (j-i)c+2j \eta_n(\e)$.
As a result,
\begin{equation}
\label{EqProp9}
\frac{c+1}{2}(j-i) -j \eta_n(\e)\le\#(P_n\cap[i+1,j]) = \frac{S_{j}^{n}-S_{i}^{n}+j-i}{2} \le \frac{c+1}{2}(j-i) +j \eta_n(\e).
\end{equation}
Therefore, {for every $k \in \llbracket 1+ \ell, m \rrbracket$},
\begin{align*}
\sum_{\beta^{k-1}<i\leqslant\beta^k}\frac{1}{S_i^n}\mathbbm{1}_{\{\X_i^n=1\}}
\le \sum_{\beta^{k-1}<i\leq\beta^k}\frac{1}{i(c-\eta_n(\e))}\mathbbm{1}_{\{\X_i^n=1\}} &\le \#\{ \beta^{k-1}<i\le\beta^k : \X_i^n=1\} \frac{1}{\beta^{k-1}(c-\eta_n(\e))} \\
&\le 
\left(\frac{c+1}{2}(\beta^k-\beta^{k-1}{+1})+\beta^k \eta_n(\e)\right) \frac{1}{\beta^{k-1}(c-\eta_n(\e))} \\
& {\leq} \left(\frac{c+1}{2}(\beta-1+{1/A_n})+\beta \eta_n(\e)\right) \frac{1}{(c-\eta_n(\e))} \eqcolon U_n(\e)
\end{align*}
where the third inequality comes from \eqref{EqProp9} {and where {for the fourth inequality} we used that $\beta^{k-1}\geq\beta^{\ell}\geq a\geq A_n$.} {
Similarly, since $\beta^{\ell-1} \leq a$,
$$\sum_{a< i\leqslant\beta^{\ell}} \frac{1}{S_i^n}\mathbbm{1}_{\{\X_i^n=1\}} \leq  \#\{a<i\le\beta^\ell : \X_i^n=1\} \frac{1}{\beta^{\ell-1}(c-\eta_n(\e))} \leq \frac{\beta^{\ell} - \beta^{\ell-1}}{\beta^{\ell-1}(c-\eta_n(\e))},$$
and since $\beta^{m+1} \geq b$,
$$ \sum_{\beta^{m}<i\leqslant b} \frac{1}{S_i^n}\mathbbm{1}_{\{\X_i^n=1\}} \leq   \#\{\beta^{m}<i\le b : \X_i^n=1\} \frac{1}{\beta^{\ell-1}(c-\eta_n(\e))} \leq   \frac{\beta^{m+1} - \beta^{m}}{\beta^{m}(c-\eta_n(\e))}.$$
}
Thus
\[
\sum_{i=a+1}^{b}\frac{1}{S_i^n}\mathbbm{1}_{\{\X_i^n=1\}} \leq (m-\ell)U_n(\e) + \frac{\beta^{\ell} - \beta^{\ell-1}}{\beta^{\ell-1}(c-\eta_n(\e))} + \frac{\beta^{m+1} - \beta^{m}}{\beta^{m}(c-\eta_n(\e))} \leq \frac{\ln(b/a)}{\ln\beta}U_n(\e) + \frac{2(\beta-1)}{c-\eta_n(\e)}.
\]
For the lower bound, the same ideas apply.
Indeed, for every integer $ k \in \llbracket 1 + \ell , m \rrbracket$ we have
\begin{align*}
\sum_{\beta^{k-1}<i\leqslant\beta^k}\frac{1}{S_i^n}\mathbbm{1}_{\{\X_i^n=1\}}
&\geq \#\{\beta^{k-1}<i\le\beta^k : \X_i^n=1\}\frac{1}{\beta^{k}(c+\eta_n(\e))} \\
&\geq \left(\frac{c+1}{2}\left(\beta^k-\beta^{k-1}-2\right)-\beta^k \eta_n(\e)\right) \frac{1}{\beta^{k}(c+\eta_n(\e))} \\
&\geq \left(\frac{c+1}{2}\left(1-\frac{1}{\beta}-\frac{2}{\beta A_n}\right)-\eta_n(\e)\right) \frac{1}{(c+\eta_n(\e))} \eqcolon L_n(\e)
\end{align*}
where the second inequality comes from \eqref{EqProp9} and the third one comes from {from the bound} $\beta^k\geq\beta^{\ell+1}\geq\beta a\geq\beta A_n$. Thus 
\[
\sum_{i=a+1}^{b}\frac{1}{S_i^n}\mathbbm{1}_{\{\X_i^n=1\}} \geq  (m-\ell)L_n(\e) \geq \frac{\ln(b/a)}{\ln\beta}L_n(\e) - 2L_n(\e) \geq\frac{\ln(b/a)}{\ln\beta}L_n(\e) - \frac{c+1}{c}\left(1-\frac{1}{\beta}-\frac{2}{\beta A_n}\right).
\]
{Using the assumption that $aC<b$,} we end up with the following inequalities
\[
\frac{L_n(\e)}{\ln\beta}-\frac{c+1}{c\ln C}\left(1-\frac{1}{\beta}-\frac{2}{\beta A_n}\right) \leq \frac{1}{\ln(b/a)}\sum_{i=a+1}^{b}\frac{1}{S_i^n}\mathbbm{1}_{\{\X_i^n=1\}}\leq \frac{U_n(\e)}{\ln\beta}+\frac{2(\beta-1)}{{(c}-\eta_n(\e))\ln C}.
\]
{Using the fact that $A_{n} \rightarrow\infty$}, notice that $\limsup_{{n \rightarrow \infty}} U_n(\e)\to(\beta-1)(c+1)/(2c)$ and $\liminf_{{n \rightarrow \infty}} L_n(\e)\to(1-1/\beta)(c+1)/(2c)$ {as} $\e\to0$. Therefore
\[
\limsup_{\e \to 0} \limsup_{n\to \infty} \max_{\substack{A_n \leqslant a< b \leqslant \e n\\aC<b}} \left| \frac{1}{\ln(b/a)}\sum_{i=a+1}^b \frac{1}{S_i^n}\mathbbm 1_{\{\X_i^n=1\}} - \frac{{c+1}
}{2c} \right| \leq \frac{{c+1}
}{2c}\left|\frac{\beta-1}{\ln\beta}-1\right|+\frac{c+1}{c\ln C}\left(1-\frac 1\beta\right)+\frac{2(\beta-1)}{c\ln C}.
\]
{This bound tends to $0$ as $\beta \rightarrow1$; since $\beta>1$ was arbitrary, the result follows.}
\end{proof}

\begin{proof}[Proof of Theorem \ref{thm:lineaire} {(I)}]
Let $(A'_n)_{n\in\N}$ be a sequence of positive integers such that $A_n\leq A'_n$, $A'_n=o(\ln n)$ and $A'_n\to\infty$ as $n\to\infty$. Applying Lemma \ref{lem:estimee uniforme sur h_n+} we see that
\[
\lim_{\e\to0}\limsup_{n\to\infty}\left|\frac{1}{\ln n}\sum_{A'_n < i \leqslant \e n} \frac{1}{S_i^n}\mathbbm{1}_{\{\X_i^n=1\}} - \frac{{c+1}
}{2c}\right|=0
\]
Since $S_i^n\geq1$ for all $1\leq i \leq n$ and $A'_n=o(\ln n)$, we have 
\[
\sum_{i=1}^{A'_n} \frac{1}{S_i^n}\mathbbm{1}_{\{\X_i^n=1\}}=o(\ln n).
\]
Finally, for all $0<\e<1$, by the second assumption of \eqref{eq:assumptions},
\[
\sum_{\e n \leqslant i \leqslant n} \frac{1}{S_i^n}\mathbbm{1}_{\{\X_i^n=1\}} = \mathcal{O}(1).
\]
Combining everything gives the desired convergence \eqref{cv hn cas lineaire}.
\end{proof}

\subsection{Distances between uniform vertices}

Here we prove  Theorem \ref{thm:lineaire} {(II)}. We keep the notation introduced in Section \ref{ssec:estimates}: for a vertex $v \in \V_{n}$, let $\birth_{n}(v)$ be the largest $i \in \{0,1, \ldots,n\}$ such that $v$ belongs to the forest $\mathcal{F}^{n}_{i}$ obtained when building $\T^{n}$ using Algorithm \ref{algo2}. We first focus on $\H(V_{1}^n)$ and introduce some notation. Set 
\[
Z_{n} \coloneqq \sum_{i=1}^{\birth_{n}(V_{1}^{n})} Y^{n}_{i} 1_{\{\X^{n}_i=1\}} \qquad \text{and} \qquad M_{n} \coloneqq \sum_{i=1}^{\birth_{n}(V_{1}^{n})}  {\frac{1}{S_i^n}} \mathbbm{1}_{\{\X_i^n=1\}}
\]
where $(Y_i^n)_{1 \leq i \leq n}$ are independent Bernoulli random variables of respective parameters $(1/S^n_i)_{1\leq i\leq n}$ (independent from $\birth_{n}(V_{1}^{n})$). Observe that $M_{n}=\esp{Z_{n} | \birth_{n}(V_{1}^{n})}$.
By \eqref{eq:H0}, we have
\begin{equation}
\label{eq:Hvn} \H(V_{1}^n) \quad \mathop{=}^{(d)} \quad Z_{n}.
\end{equation}
Fix $p\geq 1$. First we establish that $M_n/\h_n^+$ converges to $1$ in $\Lp^p$. Actually, since $0 \leq M_n \leq \h_n^+$ almost surely it is enough to show the convergence in probability.

Let $\e >0$ {be} small enough so that $\eta_n(\e) \le c/2$ for all $n$ large enough. Let $\gamma_\e>0$ {be} such that for all $n$ large enough we have
\[
\gamma_\e \le \min_{\e n \le i \le n} \frac{S^n_i}{n}.
\]
Such a $\gamma_\e$ exists thanks to \eqref{eq:assumptions}. 
Besides, since $A_n = o(\ln n) = o(n)$, by Lemma \ref{lem:bunif} we have $b_n(V_1^n) \geq A_n$ with high probability. Using this, as well as the fact that $\h_n^+ \geq \ln(n/4)$ (see Lemma \ref{lemme minoration h_n^+}), we obtain, with high probability as $n \to \infty$, 
\begin{align*}
0 \leq 1 - \frac{M_n}{\h_n^+} 
= \frac{1}{\h_n^+}\sum_{i=\birth_{n}(V_{1}^{n})}^{n} {\frac{1}{S_i^n}} \mathbbm{1}_{\{\X_i^n=1\}} &\leq \frac{1}{\h_n^+} \sum_{i=\birth_n(V_1^n)}^n \left( \frac{1}{i(c-\eta_n(\e))} + \frac{1}{\gamma_\e n}\right)\mathbbm{1}_{\{\X_i^n=1\}} \\
&\leq \frac{1}{\ln(n/4)} \frac{2}{c} \ln\left(\frac{n}{\birth_n(V_1^n)}\right) + \frac{1}{\gamma_\e \h_n^+}.
\end{align*}
Now, fix $0 < \delta < 1$ and take  $m_{n} = n^{1-\delta}2^\delta$. Then by Lemma \ref{lem:bunif}
\[
\proba{\ln\left(\frac{n}{b_n(V_1^n)}\right) > \delta\ln\left(\frac n2\right)} = \frac{m_{n}+1 - S_{m_{n}}^n}{n+1 + S_n^n} \xrightarrow[n\to\infty]{} 0.
\]
We conclude that $M_n/\h_n^+$ converges towards $1$ in probability (and therefore in $\Lp^p$ for all $p\geq 1$). Now we show that $Z_n/M_n$ converges to $1$ in $\Lp^p$ following the same spirit as in the proof of Theorem \ref{thm:height}. We have
\[
\esp{\left. \left\vert\frac{Z_n}{M_n}-1\right\vert^p~\right\vert~\birth_n(V_1^n)} = p\int_0^{\infty} \proba{\left.\left\vert\frac{Z_n}{M_n}-1\right\vert\geq u~\right\vert~\birth_n(V_1^n)}u^{p-1}du. 
\]
By Bennett's inequality (Proposition \ref{prop:bennett})
\[
\proba{\left.\left\vert\frac{Z_n}{M_n}-1\right\vert\geq u~\right\vert~\birth_n(V_1^n)} \leq 2\exp\left(-M_ng(u)\right)
\]
where $g(u) = (u+1)\ln(u+1)-u$. Since $M_n$ is of order $\ln n$ it is a simple matter to check that
\[
\E\,\left\vert\frac{Z_n}{M_n}-1\right\vert^p \leq 2p\esp{\int_0^{\infty} \exp\left(-M_ng(u)\right)u^{p-1}du} \xrightarrow[n\to\infty]{} 0.
\]
Since $M_n \leq \h_n^+$ we deduce that
\[
\E\,\left\vert\frac{Z_n-M_n}{\h_n^+}\right\vert^p \leq \E\,\left\vert\frac{Z_n}{M_n}-1\right\vert^p \xrightarrow[n\to\infty]{} 0.
\]
So $(Z_n-M_n)/\h_n^+$ converges to $0$ in $\Lp^p$. Since $M_n/\h_n^+$ converges to $1$ in $\Lp^p$, $Z_n/\h_n^+$ goes to $1$ as well. Finally, using Theorem \ref{thm:lineaire} {(I)}, we conclude that $Z_n/ \ln n$ converges to $(c+1)/(2c)$ in $\Lp^p$. This proves the first convergence of Theorem  \ref{thm:lineaire} {(II)}.

So as to deduce the asymptotic behaviour of the distance between two uniform points, we next show that two uniform random vertices coalesce near the root.

\begin{lemma}\label{lemme ça coalesce en zéro}
	{Under the assumptions of Theorem \ref{thm:lineaire} the following holds}. For all $n \ge 1$, let $V^n_1, V_2^n$ be two independent uniform vertices of $\V_{n}$. Then
	\[
	\frac{\coal_{n}(V^{n}_{1},V^{n}_{2})}{n}\cvproba[n] 0.
	\]
\end{lemma}

\begin{proof}
    Let $0 < \e < 1$. Let $i,j \leq n-1$ distinct with $\X^n_{i+1} = \X^n_{j+1} = -1$ and $k \leq i\wedge j$ with $\X^n_{k+1} = 1$. Denote by $B$ the event ``$\birth_n(V^n_1)=i$ and $\birth_n(V^n_2)=j$'' and by $C$ the event ``$\coal_n(V^n_1,V^n_2) = k$''. By Lemma \ref{loi du temps de coalescence}, it is clear that
    \[
    \proba{C \,\vert\, B} = \proba{C \text{ and } \coal_n(V^n_1,V^n_2) \leq k \,\vert\, B} \leq \proba{C \,\vert\, \coal_n(V^n_1,V^n_2)\leq k, B} = \binom{S_{k+1}^n}{2}^{-1}.
    \]
    Therefore
    \[
    \proba{\coal_n(V^n_1,V^n_2) \ge \e n \,\vert\, B} \leq \sum_{k=\lceil \e n \rceil}^{i\wedge j} \binom{S_{k+1}^n}{2}^{-1}\mathbbm{1}_{\{\X_{k+1}^n=1\}} \leq \sum_{k=\lceil \e n \rceil}^{n-1} \binom{S_{k+1}^n}{2}^{-1}\mathbbm{1}_{\{\X_{k+1}^n=1\}}.
    \]
    {By averaging over $i$ and $j$ it follows that}
    \[
    \proba{\coal_n(V^n_1,V^n_2) \ge \e n} \leq \sum_{k=\lceil \e n \rceil}^{n -1} \binom{S_{k+1}^n}{2}^{-1}\mathbbm{1}_{\{\X_{k+1}^n=1\}}.
    \]
	{Finally}, by the second hypothesis of \eqref{eq:assumptions}, we know that 
	\[
	\sum_{k=\lceil \e n \rceil}^{n-1} \binom{S_{k+1}^n}{2}^{-1}\mathbbm{1}_{\{\X_{k+1}^n=1\}} =\mathcal{O}\left(\frac{1}{n} \right)
	\qquad
	\text{as}
	\qquad
	n\to \infty,
	\]
	hence the desired result.
\end{proof}

The second convergence of the point {(II)} of Theorem \ref{thm:lineaire} is then an immediate consequence of the first convergence of the same point using  dominated convergence and Lemma \ref{lemme ça coalesce en zéro}.

\subsection{Limit theorem for the height}
\label{ssec:limitheight}

Here we show  Theorem \ref{thm:lineaire} {(III)}. The proof of the upper bound is readily obtained by a union bound, while the lower bound is much more delicate.

\begin{proof}[Proof of the upper bound in Theorem \ref{thm:lineaire} {(III)}] Fix $\eta>0$.
First observe that by definition of $f(c)$ we have
\[
g\left(f(c)-1\right)=f(c)\ln f(c)-f(c)+1=1+\frac{c-1}{{c+1}
}=\frac{2c}{{c+1}
},
\]
so that $g\left(f(c)-1+\eta\right)>{2c}/({c+1}
)$. By Theorem \ref{thm:lineaire} {(I)} we have $\h^+_n\sim ({c+1}
)/(2c){\ln}(n)$  as $n\to \infty$. As a consequence, there exists $\delta \in (0,1)$ such that $h^+_n\,g\left(f(c)-1+\eta\right) \geq (1+\delta) \ln(n)$ for $n$ sufficiently large.
Hence, by Proposition \ref{prop:bounds}, for $n$ sufficiently large,
\[
\proba{\Height(\T^n)>\h^+_n(f(c)+\eta)} \leq \exp\left(-\h^+_n\,g\left(f(c)-1+\eta\right)+\ln n\right) \leq \exp(-\delta \ln(n))  \quad \mathop{\longrightarrow}_{n \rightarrow \infty} \quad 0.
\]
This completes the proof of the upper bound.
\end{proof}

It therefore remains to show the lower bound on $\Height(\T_{n})$, which is the delicate part of the proof. Let us explain the main idea of our approach. Since the upper bound has been obtained by using a union bound over all vertices, one would hope to obtain the lower bound from the fact that the height of the vertices are "almost independent". However, their heights are highly correlated. To overcome this issue, we use the so-called chaining technique and estimate by induction on $k$ the height of the subtrees $\Height(\T_{k}(\Xb^n))$. More precisely, for every increasing sequence of integers (chain) $0=R_0<R_1<\dots < R_m=n$, we have
\[ \Height(\T_{n})=\sum_{0\leq i <m} \Height(\T_{R_{i+1}}(\Xb^n))- \Height(\T_{R_i}(\Xb^n)).\]
We shall choose an appropriate chain of integers $(R_i)_{0\leq i \leq m}$ and  then bound the differences appearing in the last display. To this end, the main idea is that if $R_{i+1}/R_i$ is close to some large constant $C>0$, then $\T_{R_{i+1}}(\Xb^n)$ can be approximated by a forest of independent trees attached on 
$\T_{R_i}(\Xb^n)$. We then show by induction on $i$ that many vertices have a large height in this forest, and among them many are attached on vertices of $\T_{R_i}(\Xb^n)$ with large height. 

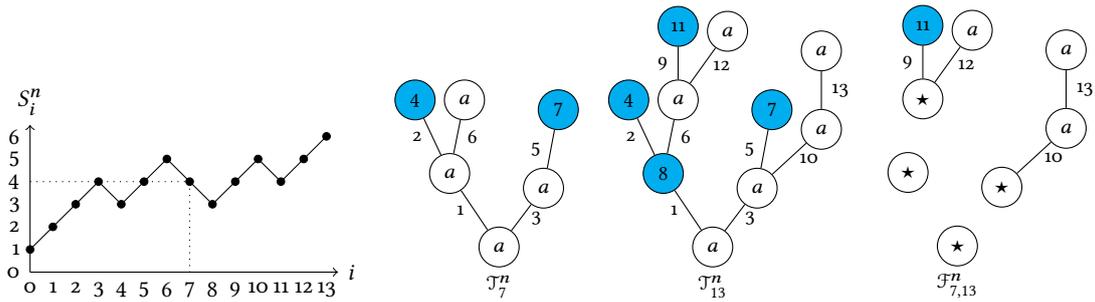
\begin{figure}[h!]
\begin{center}
\begin{tikzpicture}[scale = 0.3]
\draw[->] (0,0) -- (0,6.5);
\draw[->] (0,0) -- (13.5,0);

\node[above,font=\footnotesize] at (0,6.5) {$S_i^n$};
\node[right,font=\footnotesize] at (13.5,0) {$i$};

\node[draw,circle,fill,inner sep = 1pt] at (0,1) {};
\node[draw,circle,fill,inner sep = 1pt] at (1,2) {};
\node[draw,circle,fill,inner sep = 1pt] at (2,3) {};
\node[draw,circle,fill,inner sep = 1pt] at (3,4) {};
\node[draw,circle,fill,inner sep = 1pt] at (4,3) {};
\node[draw,circle,fill,inner sep = 1pt] at (5,4) {};
\node[draw,circle,fill,inner sep = 1pt] at (6,5) {};
\node[draw,circle,fill,inner sep = 1pt] at (7,4) {};
\node[draw,circle,fill,inner sep = 1pt] at (8,3) {};
\node[draw,circle,fill,inner sep = 1pt] at (9,4) {};
\node[draw,circle,fill,inner sep = 1pt] at (10,5) {};
\node[draw,circle,fill,inner sep = 1pt] at (11,4) {};
\node[draw,circle,fill,inner sep = 1pt] at (12,5) {};
\node[draw,circle,fill,inner sep = 1pt] at (13,6) {};

\draw (0,1)--(1,2)--(2,3)--(3,4)--(4,3)--(5,4)--(6,5)--(7,4)--(8,3)--(9,4)--(10,5)--(11,4)--(12,5)--(13,6);

\node[left,font=\footnotesize] at (0,0) {0};
\node[left,font=\footnotesize] at (0,1) {1};
\node[left,font=\footnotesize] at (0,2) {2};
\node[left,font=\footnotesize] at (0,3) {3};
\node[left,font=\footnotesize] at (0,4) {4};
\node[left,font=\footnotesize] at (0,5) {5};
\node[left,font=\footnotesize] at (0,6) {6};

\node[below,font=\footnotesize] at (0,0) {0};
\node[below,font=\footnotesize] at (1,0) {1};
\node[below,font=\footnotesize] at (2,0) {2};
\node[below,font=\footnotesize] at (3,0) {3};
\node[below,font=\footnotesize] at (4,0) {4};
\node[below,font=\footnotesize] at (5,0) {5};
\node[below,font=\footnotesize] at (6,0) {6};
\node[below,font=\footnotesize] at (7,0) {7};
\node[below,font=\footnotesize] at (8,0) {8};
\node[below,font=\footnotesize] at (9,0) {9};
\node[below,font=\footnotesize] at (10,0) {10};
\node[below,font=\footnotesize] at (11,0) {11};
\node[below,font=\footnotesize] at (12,0) {12};
\node[below,font=\footnotesize] at (13,0) {13};

\draw[dotted] (0,4)--(7,4)--(7,0);
\end{tikzpicture}
%%%%%
\hspace{0.5em}
\begin{tikzpicture}[scale = 0.65,
sommet/.style = {draw,circle, font=\scriptsize,inner sep=0,minimum size=15pt},
gele/.style = {fill=cyan},
etiquete/.style = {font = \scriptsize}]

\node[sommet] (0) at (0,0) {$a$};
\node[sommet] (1) at (-1,1.5) {$a$};
\node[sommet] (2) at (0.9,1.2) {$a$};
\node[sommet,gele] (3) at (-1.7,3) {4};
\node[sommet] (4) at (-0.7,3) {$a$};
\node[sommet,gele] (5) at (1.2,2.8) {7};

\draw (0) --node[left,etiquete] {1} (1);
\draw (0) --node[right,etiquete] {3} (2);
\draw (1) --node[left,etiquete] {2} (3);
\draw (1) --node[right,etiquete] {6} (4);
\draw (2) --node[left,etiquete] {5} (5);

\node[etiquete] (t) at (0,-0.8) {$\T^{n}_{7}$};

\end{tikzpicture}
%%%
\hspace{0.5em}
\begin{tikzpicture}[scale = 0.65,
sommet/.style = {draw,circle, font=\scriptsize,inner sep=0,minimum size=15pt},
gele/.style = {fill=cyan},
etiquete/.style = {font = \scriptsize}]

\node[sommet] (0) at (0,0) {$a$};
\node[sommet,gele] (1) at (-1,1.5) {8};
\node[sommet] (2) at (0.9,1.2) {$a$};
\node[sommet,gele] (3) at (-1.7,3) {4};
\node[sommet] (4) at (-0.7,3) {$a$};
\node[sommet,gele] (5) at (1.2,2.8) {7};
\node[sommet] (6) at (2.2,2.4) {$a$};
\node[sommet,gele] (7) at (-0.7,4.5) {11};
\node[sommet] (8) at (0.3,4.4) {$a$};
\node[sommet] (9) at (2.2,4) {$a$};

\draw (0) --node[left,etiquete] {1} (1);
\draw (0) --node[right,etiquete] {3} (2);
\draw (1) --node[left,etiquete] {2} (3);
\draw (1) --node[right,etiquete] {6} (4);
\draw (2) --node[left,etiquete] {5} (5);
\draw (2) --node[right,etiquete] {10} (6);
\draw (4) --node[left,etiquete] {9} (7);
\draw (4) --node[right,etiquete] {12} (8);
\draw (6) --node[right,etiquete] {13} (9);

\node[etiquete] (t) at (0,-0.8) {$\T^{n}_{13}$};

\end{tikzpicture}
%%%
\hspace{0.5em}
\begin{tikzpicture}[scale = 0.65,
sommet/.style = {draw,circle, font=\scriptsize,inner sep=0,minimum size=15pt},
gele/.style = {fill=cyan},
etiquete/.style = {font = \scriptsize}]

\node[sommet] (0) at (0,0) {$\star$};
\node[sommet] (1) at (-1,1.5) {$\star$};
\node[sommet] (2) at (0.9,1.2) {$\star$};
\node[sommet] (4) at (-0.7,3) {$\star$};
\node[sommet] (6) at (2.2,2.4) {$a$};
\node[sommet,gele] (7) at (-0.7,4.5) {11};
\node[sommet] (8) at (0.3,4.4) {$a$};
\node[sommet] (9) at (2.2,4) {$a$};

\draw (2) --node[right,etiquete] {10} (6);
\draw (4) --node[left,etiquete] {9} (7);
\draw (4) --node[right,etiquete] {12} (8);
\draw (6) --node[right,etiquete] {13} (9);

\node[etiquete] (f) at (0,-0.8) {$\F_{7,13}^n$};

\end{tikzpicture}
\caption{Illustration of the construction of $\F_{k,k'}^n$ with $k=7$ and $k'=13$. From left to right: the sequence $\Xb^n$, a  realization of $\T_{k}^n$,  a  realization of $\T_{k'}^n$  and  the forest $\F_{k,k'}^n$ composed of four trees all rooted at their unique vertex labelled $\star$. {If one matches and merges uniformly at random every $\star$ vertex of  $\F_{k,k'}^n$ with every active vertex of $\T_{k}^n$, one obtains a tree having the same distribution as $\T_{k'}^n$.}}
\label{fig chainage}
\end{center}
\end{figure}

To make this intuition rigorous we first need to define some objects. For $n\in \N$ and $0\leq k \leq n$, recall the definition of $\T_k^n\coloneqq\T_k(\Xb^n)$ from Algorithm \ref{algo1}. We denote by $\mathcal A_k^n$  the set of all vertices in $\T_n^n$ that {were active at time $k$ when running Algorithm \ref{algo1} (these are vertices of ${\T_{n}^{n}}$ that are either labelled with $a$   {and have an adjacent edge labelled by $j \leq k$}, or labelled with $i > k$  and have an adjacent edge labelled by $j \leq k$)}.
For $0\leq k\leq k'\leq n$, let $\F_{k,k'}^n$ be the forest obtained as follows: 

\begin{itemize}[topsep=0pt,itemsep=-1ex,partopsep=1ex,parsep=1ex]
\item[(i)] {Consider $\T_{k'}^n$ and relabel by $\star$  all the vertices of $\T_{k'}^n$ belonging to $\mathcal A_{k}^n$ (see Fig.~\ref{fig chainage} for an example);}
\item[(ii)] then remove all the edges of $\T_{k'}^n$ that belongs to $\T_{k}^n$ (i.e. all the edges with labels in $\llbracket1,k\rrbracket$);
\item[(iii)] finally remove all the vertices with labels in $\llbracket1,k\rrbracket$. 
\end{itemize}
{Note that all the trees in $\F_{k,k'}^n$ have exactly one vertex labelled $\star$  {(see Fig.~\ref{fig chainage} for an example), which by convention is considered to be active}. We will thus consider the vertices labelled $\star$ as the roots of the trees in $\F^{n}_{k,k'}$. Observe that one cannot  deterministically reconstruct $\T_{k'}^n$ from $\{\F_{k,k'}^n,\T_k^n\}$, but in distribution it is possible: if one merges uniformly at random every active vertex of $\T_{k'}^n$ with one vertex labelled $\star$ from $\F^{n}_{k,k'}$, then one obtains a tree having the same distribution as  $\T_{k'}^n$.} {This comes from the fact that conditionally given  $\T_{k}^n$, at any future time the subtrees grafted on active vertices of  $\T_{k}^n$ after time $k$ are exchangeable}. Furthermore, it will be crucial to keep in mind that that $\T_k^n$ and $\F^{n}_{k,k'}$ are independent. Indeed $\F^{n}_{k,k'}$ may be constructed following Algorithm \ref{algo1}, independently of the steps $1\leq i \leq k$.

For every $0\leq k\leq k'\leq n$, for every vertex $v\in \F^{n}_{k,n}$ we write $H^n_{k,k'}(v)$ for the height of $v$ in the forest $\F^{n}_{k,n}$ minus its height in $\F^n_{k',n}$ (where, by convention, the height of $v$ in $\F^n_{k',n}$ is $0$ if $v$ is not in this forest).  In words, if $v$ belongs to a certain tree $\tau$ of  $\F^n_{k',n}$, the quantity $H^n_{k,k'}(v)$ represents the height in $\F^{n}_{k,k'}$  of the vertex associated with the root $\tau$. 

For $\eta>0$, for every integers $n,k,C\in\N$ such that $C^k\leq n/C$, set 
\[ 
N_k^n(C,\eta)= \# \left\{v\in \mathcal A_{C^{k+1}}^n : H^n_{C^k,C^{k+1}}(v) > \frac{{c+1}
}{2c}\ln(C)(f(c)-\eta) \right\}. 
\]
The set of vertices involved in the definition of $N_{k}^{n}$ will play an important role in our approach. {Roughly speaking, they correspond to  vertices active at time $C^{k+1}$ which are ``quite far'' from a vertex active at time $C^{k}$. The first main input in the proof of the lower bound of the height is the following result, which shows that $ C>0$ can be chosen such that if $R_{i+1}/R_i$ is close to $C>0$, then there are many active vertices with a ``large'' height in the forest $\F_{R_{i},R_{i+1}}^n$. Its proof is deferred to Section \ref{ssec:lem}.}

\begin{lemma} 
\label{Forestier} 
For every $\eta \in (0,1)$, there exists an integer ${C  \geq 2}$, there exist $\lambda>0$  {and $\e \in (0,1)$ such that} for every $n$ large enough, for every $k\in \N$, with $A_n \leq C^k\leq \e n/C$, we have
\[ 
{\proba{N_k^n(C,\eta) \leq  4 C^k} \leq   \frac{\lambda}{C^{k}}
}\]
\end{lemma} 

For every $n,k\in \N$, with $C^k\leq n/C$, let $M^{n}_{k}$ be the maximal number of active vertices of a tree in the forest $\F^{n}_{C^{k},C^{k+1}}$.  {The second main input in the proof of the lower bound of the height is the following result, which shows that the this quantity cannot be too large. Its proof is deferred to Section \ref{ssec:lem}.}

\begin{lemma}
\label{Forestjump}
For every $n$ large enough, for every $k\in \N$, with $A_n \leq C^k\leq  n/C$,
\[ \proba{M^{n}_{k}> k^3}\leq \frac{1}{k^{2}}. \]
\end{lemma}

Let us now explain how the lower bound for the height follows from Lemmas \ref{Forestier} and \ref{Forestjump}.

\begin{proof}[Proof of the lower bound for Theorem \ref{thm:lineaire} {(III)}] 
{Fix $\eta \in (0,1)$ and take ${C  \geq 2}, \lambda>0, \e\in(0,1)$ such that the conclusion of Lemma \ref{Forestier} holds} {and write $N_k^n$ instead of $N_k^n(C,\eta)$ to simplify notation.}
{For all integers $n,k\in\N$ such that $A_n\leq C^k\leq n/C$,} set
\[
\overline{N}^n_k = \# \left\{v\in \mathcal A^{n}_{C^{k}}, \quad \forall \ell \in {\llbracket1,k-1\rrbracket \text{ s.t. } A_n\leq C^\ell}, \quad   H^n_{C^\ell,C^{\ell+1}}(v) > \frac{{c+1}
}{2c}\ln(C)(f(c)-\eta) \right\}.
\]
We shall show that  {if $ {K_{n}}$ is the largest integer $k$ such that $C^{{k}}\leq {\e} n/C$, then
\begin{equation}
\label{eq:proba} \proba{\overline N^{n}_{{{K_{n}}}}\geq 1}  \quad \mathop{\longrightarrow}_{n \rightarrow \infty} \quad 1.
\end{equation}
If \eqref{eq:proba} holds, then with probability tending to $1$} there exists a vertex $v_{{{K_{n}}}} \in {\mathcal{A}}^{n}_{C^{{{K_{n}}}}}$ and vertices $v_{i} \in {\mathcal{A}}^{n}_{C^{i}}$ for every $i\in{\llbracket1,{{K_{n}}}-1\rrbracket\text{ such that } A_n\leq C^i}$, $v_{i}$ is an ancestor of $v_{i+1}$ and $d^{n}(v_{i},v_{i+1}) >\frac{{c+1}
}{2c}\ln(C)(f(c)-\eta)$. Thus the height of $v_{{{K_{n}}}}$ in $\mathcal{T}^{n}$ is at least
{
\[
\# \{ i \in \N : \ln(A_{n})/\ln(C) \leq i < K_{n}\} \cdot \frac{{c+1}
}{2c}\ln(C)(f(c)-\eta),
\]
which is equal to}
\[ \left( \left\lfloor \frac{\ln(n)}{\ln(C)} {+ \frac{\ln(\e/C)}{\ln(C)}} \right\rfloor - \left\lceil\frac{\ln(A_n)}{\ln(C)} \right\rceil  \right) \cdot \frac{{c+1}
}{2c} \ln(C) \cdot  (f(c)-\eta)   \quad \mathop{\sim}_{n \rightarrow \infty} \quad  \ln(n) \frac{{c+1}
}{2c}(f(c)-\eta). \]
Since $\eta \in (0,1)$ was chosen {arbitrarily}, this will imply the desired result.

{The main step of the proof {to prove \eqref{eq:proba} is to} establish that  for every $k\in \N$ with $A_n \leq C^k\leq \e n/C$ we have
\begin{equation}
\label{eq:Nbar}
\proba{ \overline N^{n}_{k+1} <2^{k+1}  \middle| \overline N^{n}_{k} \geq 2^k}=\mathcal{O}\left(\frac{1}{k^2}\right),
\end{equation}
{where here, and in the rest of the proof, the $\mathcal{O}$ is uniform in  $A_n \leq C^k\leq \e n/C$ . This means that there exists a constant $\Delta>0$ such that the term $\mathcal{O}({1}/{k^2})$ can be bounded from above by $\Delta/k^{2}$ for every $A_n \leq C^k\leq \e n/C$ and $n$ large enough.}
If this holds, 
\[
\proba{\exists k \textrm{ with } A_n \leq C^k\leq \e n/C \textrm{ and } \overline N^{n}_{k}< 2^k}  \leq	 \proba{N^{n}_{k_{n}}< 2^{k_{n}}} + \mathcal{O} \left( \sum_{k=k_{n}}^{\infty} \frac{1}{k^{2}}\right)
\]
with $k_{n}=\lceil  \ln(A_{n})/\ln(C) \rceil$. By Lemma \ref{Forestier},
{
\[
 \proba{N^{n}_{k_{n}}< 2^{k_{n}}} \leq \proba{N_{k_{n}}^n  \leq  4 C^ {k_{n}}} \leq   \frac{\lambda}{C^{k_{n}}} \leq \frac{\lambda}{A_{n}}  \quad \mathop{\longrightarrow}_{n \rightarrow \infty} \quad 0,
\]
where the first inequality holds because $C \geq 2$.
}
{As a consequence,}  with probability tending to $1$ as $ n \rightarrow \infty$, for every integer $k$ such that $A_n \leq C^k\leq \e n/C$, we have $\overline N^{n}_{k}\geq 2^k$, {which implies \eqref{eq:proba}.}
}

{It thus remains to establish \eqref{eq:Nbar}.}  Note that the random variable $\overline N^{n}_{k}$ is $\T^{n}_{C^k}$ measurable, and so is independent of the random variables $N^{n}_{k}$ and $M^{n}_k$ which are $\F^{n}_{C^k,C^{k+1}}$ measurable. As a consequence, we get {for every $k\in \N$ with $A_n \leq C^k\leq \e n/C$}: 
\begin{eqnarray}
&&\proba{ \overline N^{n}_{k+1} <2^{k+1}  \middle| \overline N^{n}_{k} \geq 2^k} \notag \\
&& \qquad\qquad \leq  \proba{N^{n}_{k}< 4C^k}+\proba{M^{n}_{k}>k^3} + \proba{ \overline N^{n}_{k+1} <2^{k+1},  N^{n}_{k}\geq 4C^k, M^{n}_{k}\leq k^3 \middle| \overline N^{n}_{k}\geq 2^k} \notag\\
&& \qquad\qquad\leq  \mathcal{O} \left(  \frac{1}{C^{k}}\right)+ \mathcal{O} \left( \frac{1}{k ^{2}}\right)+\proba{ \overline N^{n}_{k+1} <2^{k+1},  N^{n}_{k}\geq 4C^k, M^{n}_{k}\leq k^3 | \overline N^{n}_{k}\geq 2^k},
\label{14/07/21h}
\end{eqnarray}
where we have used  Lemmas \ref{Forestier} and \ref{Forestjump}.

To bound the third term, we  shall estimate $\overline N^{n}_{k}$   by using a second moment technique. To this end, we recall that one may reconstruct a tree with the same distribution as $\T^{n}_{C^{k+1}}$, by  merging uniformly at random each active vertex from $\T^{n}_{C^{k}}$ with one vertex labelled $\star$ from $\F^{n}_{C^k,C^{k+1}}$. 
With this in mind, it follows that 
\begin{equation}
\E \left [\left . \overline N^{n}_{k+1} \right | \T^{n}_{C^k}, \F^{n}_{C^k,C^{k+1}} \right ] = \frac{\overline N^{n}_{k} N^{n}_{k}}{S^n_{C^k}},
\label{14/07/19h}
\end{equation}
where we recall that $S_{C^k}^n= S_{C^{k}}(\Xb^n)$ represents the number of active vertices of $ \T^{n}_{C^{k}}$. 

Next, denote by $\mathcal{A}^{n}_{C^k,C^{k+1}}$ the set of all active vertices of the forest $\F^{n}_{C^k,C^{k+1}}$. For  $v\in \mathcal{A}^{n}_{C^k,C^{k+1}}$, let $E^{n}_{k}(v)$ be the event defined by
\[
E^{n}_{k}(v)=\left\{ \forall \ell \in {\llbracket1,k-1\rrbracket \text{ s.t. } A_n\leq C^\ell}, \quad   H^n_{C^\ell,C^{\ell+1}}(v) > \frac{{c+1}
}{2c}\ln(C)(f(c)-\eta) \right\};
\] 
{if this event holds say that $v$ is $k$-good, and $k$-bad otherwise. Observe that by definition, for every tree of $\F^{n}_{C^k,C^{k+1}}$, either all its active vertices are $k$-good, or all its vertices are $k$-bad.}

{Note that for all $v, v' \in \mathcal{A}^{n}_{C^k,C^{k+1}}$, conditionally on $ \T^{n}_{C^k}$ and $\F^{n}_{C^k,C^{k+1}}$, the events $E^{n}_{k+1}(v)$ and $E^{n}_{k+1}(v')$ are negatively correlated unless $v,v'$ are in the same tree in $\mathcal{F}^n_{C^k,C^{k+1}}$.  Indeed, on the one hand, if the trees of $\F^{n}_{C^k,C^{k+1}}$ containing $v$ and $v'$ are different, then when we attach the tree containing $v$ to an active vertex of $ \mathcal{A}^{n}_{C^k}$ {in such a way that $v$ is $k$-good} then the tree containing $v'$ will have less chance to be attached to {another} vertex of $\mathcal{A}^{n}_{C^k}$ {in such a way that $v'$ is $k$-good}, and vice versa. On the other hand, if $v$ and $v'$ are in the same tree and if we have the inequalities $H^n_{C^k,C^{k+1}}(v) > \frac{{c+1}
}{2c}\ln(C)(f(c)-\eta)$ and $H^n_{C^k,C^{k+1}}(v') > \frac{{c+1}
}{2c}\ln(C)(f(c)-\eta)$, if we denote by $w$ the (random) active vertex of $\T_{C^k}^n$ on which this tree is attached, then $v$ and $v'$ are both $k+1$-good if and only if $w$ is$ k$-good.}

As a result, {writing $\widetilde{\mathbb{P}}$ and $\widetilde{\mathbb{E}}$ for the conditional probability and expectation conditionally given $ \T^{n}_{C^k}$ and $\F^{n}_{C^k,C^{k+1}}$, we get}
\begin{align}
\vartilde \left [ \overline N^{n}_{k+1}  \right ]  = \vartilde \left [  \sum_{v\in \mathcal{A}^{n}_{C^k,C^{k+1}}} \1_{E^{n}_{k+1}(v)}  \right ] &=\sum_{v,v'\in  \mathcal{A}^{n}_{C^k,C^{k+1}}} \covtilde ( \1_{E^{n}_{k+1}(v)} , \1_{E^{n}_{k+1}(v')} ). \notag
\\ & \leq M^{n}_{k} \sum_{v\in  \mathcal{A}^{n}_{C^k,C^{k+1}}} \probatilde { {E^{n}_{k+1}(v)}}   = M^{n}_{k}  \cdot  {\esptilde{  \overline N^{n}_{k+1}}}.
\label{14/07/19hb}
\end{align}
Then,  by \eqref{14/07/19h} {on} the event 
$\left\{\overline N^{n}_{k}\geq 2^k, N^{n}_{k}\geq 4C^k, M^{n}_{k}\leq k^3\right\}$ (which is measurable with respect to $ \T^{n}_{C^k}, \F^{n}_{C^k,C^{k+1}}$), {using the bound $S^{n}_{C^{k}} \leq C^{k}$,}
we have
\[ \esptilde{  \overline N^{n}_{k+1}}\geq \frac{\overline N^{n}_{k} N^{n}_{k}}{S^n_{C^k}}\geq 2^{k+2}.\]
{Thus, by \eqref{14/07/19hb}, still on the event 
$\left\{\overline N^{n}_{k}\geq 2^k, N^{n}_{k}\geq 4C^k, M^{n}_{k}\leq k^3\right\}$,} we have 
\begin{eqnarray*}
\probatilde{  \overline N^{n}_{k+1} <  2^{k+1} } & \leq& \probatilde{  \overline N^{n}_{k+1} < \frac{1}{2}\esptilde{ \overline N^{n}_{k+1}}  } 
  \leq 4 \frac{ \vartilde \left[ \overline N^{n}_{k+1} \right]} { \esptilde{ \overline N^{n}_{k+1}}^2}
  \leq  {4 \frac{ M^{n}_{k} }{\esptilde{  \overline N^{n}_{k+1}}}}  \\
  &{=}&  \frac{4 M^{n}_{k} S^n_{C^k}}{N^{n}_{k} N^{n}_{k}}
  \leq \frac{4k^3 C^k}{2^k C^k}=\mathcal{O}\left( \frac{1}{k^{2}}\right).
  \end{eqnarray*}
{As a consequence,
\begin{eqnarray*}
&&\proba{ \overline N^{n}_{k+1} <2^{k+1}, N^{n}_{k}\geq 4C^k, M^{n}_{k}\leq k^3 \middle| \overline N^{n}_{k}\geq 2^k} \\
&& \qquad \qquad	= \esp{\probatilde{  \overline N^{n}_{k+1} < 2^{k+1}}  \mathbbm{1}_{N^{n}_{k}\geq 4C^k, M^{n}_{k}\leq k^3} \middle| \overline N^{n}_{k}\geq 2^k }=\mathcal{O}\left( \frac{1}{k^{2}}\right),
\end{eqnarray*}
and \eqref{eq:Nbar} follows from \eqref{14/07/21h}. This completes the proof.}
\end{proof}

\subsection{Proof of Lemmas \ref{Forestier} and \ref{Forestjump}}
\label{ssec:lem}
We keep the notation introduced in Section \ref{ssec:limitheight}. The following estimate is the first step for the proof of Lemma \ref{Forestier}.
\begin{lemma}
\label{lem:minoration}
For every $\eta \in (0,1)$, there exists an integer $C>1$ {and $\e\in(0,1)$ such that} for every $n$ large enough for every $k\in \N$, with $A_n \leq C^k\leq \e n/C$, we have
\[
\esp{N_k^n(C,\eta)} \geq 5 C^{k}.
\]
\end{lemma}
\begin{proof}[Proof of Lemma \ref{lem:minoration}]
Recall that $g(x)=(1+x)\ln(1+x)-x$. Set $\delta=  1- \frac{{c+1}
}{2c}\cdot g(f(c)-1-\eta)$, which belongs to $(0,1)$ since $ g(f(c)-1-\eta) < (2c)/({c+1}
)$, and then choose an integer $C>1$ {large enough} such that 
\begin{equation}
\frac{1}{\ln(C)^{4}} \exp \left(-  (1-\delta) \ln(C)   \right) \geq  \frac{6}{c C} \qquad {\textrm{and} \qquad \ln(C) \ge  \frac{{c+1}
}{2c}(f(c)-\eta)}. \label{13/06/7h}
\end{equation}
{It remains to check that} this value of $C$ satisfies the desired conclusion. By Lemma \ref{lem:accroissementH} for every $v\in \mathcal A_{C^{k+1}}^n$,
\[ 
H^{n}_{C^{k},C^{k+1}}(v)\quad \mathop{=}^{(d)} \quad \sum_{C^k<i\leq C^{k+1}} Y_i^n, 
\]
where $(Y_i^n)_{1\leq i \leq n}$ are independent Bernoulli random variables of respective parameters $(1/S_i^n\1_{\X^n_i=1})_{1\leq i \leq n}$. 
{Lemma \ref{lem:estimee uniforme sur h_n+} yields \[\lim_{\e \to 0} \limsup_{n\to \infty}  \max_{A_{n} \leq C^{k}\leq \e n} \left| \sum_{C^k<i\leq C^{k+1}} \frac{1}{S_i^n} \1_{\X^n_i=1} -   \frac{{c+1}
}{2c}\ln(C) \right| \quad \mathop{\longrightarrow}_{n \rightarrow \infty} \quad 0.\]}
Hence, by  \cite[Theorem 1]{BH84}, we have
\[ \lim_{\e \to 0} \limsup_{n\to \infty}  \max_{A_{n} \leq C^{k}\leq \e n}  d_{\textrm{TV}} \left(H^{n}_{C^{k},C^{k+1}}(v), \mathcal P \text{oi} \left(  \frac{{c+1}
}{2c}\ln(C) \right) \right)   =  0,\]
where $\mathcal{P} \text{oi}(\lambda)$ denotes a Poisson random variable with parameter $\lambda$ and $d_{\textrm{TV}}$ stands for the total variation distance. As a consequence, 
setting
\[
P(C) \coloneqq  \P \left(  \mathcal P \text{oi} \left(  \frac{{c+1}
}{2c}\ln(C)\right) >\frac{{c+1}
}{2c}\ln(C)(f(c)-\eta)\right),
\]
we have
\begin{equation}
\label{eq:cvPC}  \lim_{\e \to 0} \limsup_{n\to \infty}  \max_{A_{n} \leq C^{k}\leq \e n}  \left| \P \left( H^{n}_{C^{k},C^{k+1}}(v)> \frac{{c+1}
}{2c}\ln(C)(f(c)-\eta) \right)  - P(C) \right| = 0.
\end{equation}
Then, for $\lambda>2$ and $t >1$, using the bound $\ceil{u}! \leq  u^{3}(u/e)^{u}$ for $u \geq 2$,
observe that
\[ 
\proba{\mathcal P \text{oi}\left (\lambda \right )>\lambda t} \geq e^{-\lambda}  \frac{\lambda^{\ceil{\lambda t}}}{\ceil{\lambda t}!} \geq e^{-\lambda}  \lambda^{ \lambda t} \left(\frac{e}{\lambda t}\right )^{\lambda t}\frac{1}{(\lambda t)^3} = e^{-\lambda g(t-1)}  \frac{1}{(\lambda t)^3},
\]
Taking into account our choice of $C$ (see \eqref{13/06/7h}), it follows that 
\begin{equation}
\label{eq:bornePC}P(C) \geq \frac{1}{\ln(C)^{4}} \exp \left(- \frac{{c+1}
}{2c}\ln(C)  g(f(c)-1-\eta) \right) \geq \frac{6}{cC}.
\end{equation}
Then observe that by linearity we have
\[\esp{{N_k^n(C,\eta)}}= \#\mathcal{A}^n_{C^{k+1}} \P \left( H^{n}_{C^{k},C^{k+1}}(v)> \frac{{c+1}
}{2c}\ln(C)(f(c)-\eta) \right) .\]
But $|\#\mathcal{A}^n_{C^{k+1}} / (c C^{k+1})-1| =|S^n_{C^{k+1}} / (c C^{k+1})-1| \leq  \eta_{n}(\e)/c$ for $C^{k+1} \leq \e n$, 
where we recall that
\[
\eta_n(\e) = \max_{A_n \leqslant i \leqslant \e n } \left \lvert \frac{S^n_i}{i} -c \right\rvert .
\]
The desired result follows by using \eqref{eq:cvPC} and \eqref{eq:bornePC}.
\end{proof}

 {Observe that by \eqref{eq:assumptions} we may choose ${\e_{0} \in (0,1)}$ such that
\[
\limsup_{n\to \infty} \max_{A_n \le i \le \e_{0} n} \left| \frac{S^n_i}{i}-c \right| < \frac{c}{2}
\qquad \text{and} \qquad 
\liminf_{n\to \infty} \min_{\e_{0} n \le i \le n } \frac{S^n_i}{n} >0,
\]
which implies the existence of a constant $c_{0}>0$ such that:
\begin{equation}
\label{eq:c0}
\textrm{for every $n$ sufficiently large, \quad for every $1 \leq i \leq n$,} \qquad  c_{0} i \leq S^{n}_{i} \leq  i+1,
\end{equation}
where the upper bound comes from the fact that $\X^{n}_{i} \leq 1$.}

We are now ready to prove Lemma \ref{Forestier}.

\begin{proof}[Proof of Lemma \ref{Forestier}]
Fix $\eta \in (0,1)$. Consider $C>1$ given by Lemma \ref{lem:minoration}. Take $\varepsilon\in(0,1)$ small enough such that the conclusion of Lemma \ref{lem:minoration} holds.
The main idea of the proof is to show that for some fixed $\lambda>0$ we have uniformly for every $k\in \N$ with $A_n\leq C^k\leq \e n/C$
\begin{equation}
\label{eq:var}\var( N_k^n(C,\eta))\leq \lambda C^k.
\end{equation}
 Indeed, it will then directly follow from Lemma \ref{lem:minoration} and  Bienaymé-Chebyshev's inequality that 
\[
\proba{ N_k^n(C,\eta) <  4 C^k} \leq \proba{ \left | N_k^n(C,\eta)-\E[N_k^n(C,\eta)] \right | > C^k}\leq \var( N_k^n(C,\eta))/C^{2k} \leq  \lambda/C^k.
\]
{
We shall show that
\begin{equation}
\label{eq:esp}\esp{( N^n_k(C,\eta))^2}  \le  \left( \esp{ N^n_k(C,\eta)} \right)^2  + \mathcal{O}(C^k).
\end{equation}
This indeed implies \eqref{eq:var}.} Observe that by definition of $N_k^n(C,\eta)$, setting $t=\frac{{c+1}
}{2c}\ln(C)(f(c)-\eta)$, we have
\[
N_k^n(C,\eta)= \sum_{v \in \mathcal{A}_{C^{k+1}}^n} \mathbbm{1}_{H^n_{C^{k},C^{k+1}}(v)>t}.
\]
Note from Algorithm \ref{algo2}, that for every $v,w \in \mathcal{A}_{C^{k+1}}^n$ {with $v \neq w$} 
\[ 
\left( H^n_{C^{k},C^{k+1}}(v)\quad , \quad H^n_{C^{k},C^{k+1}}(w)\right)\quad \mathop{=}^{(d)} \quad \left ( \sum_{C^k<i\leq C^{k+1}} Y_i^n \quad , \quad \sum_{C^k<i\leq C^{k+1}} Z_i^n \right )\]
where $(Z_i^n)_{1\leq i \leq n}$ are independent Bernoulli random variables of respective parameters $(1/S_i^n\1_{\X^n_i=1})_{1\leq i \leq n}$. The problem is that $(Z_i^n)_{1\leq i \leq n}$ and $(Y_i^n)_{1\leq i \leq n}$ are not independent. The idea is to say that, conditionally on the fact that $v$ and $w$ do not coalesce (i.e they do not belong to the same tree) before time $C^{k+1}$, then, when the height of $v$ goes up by $1$ in the coalescence process, the height of $w$ doesn't change and vice versa. Therefore their height{s} are negatively correlated. One can see that for all $v, w \in \mathcal{A}^n_{C^{k+1}}$,
\begin{align}
	\P&\left( H^n_{C^{k},C^{k+1}}(v) >t\text{ and }H^n_{C^{k},C^{k+1}}(w) >t \right) \notag\\
	& \le \proba{C^k \le \coal_n(v,w)\le C^{k+1}}
	+\proba{\coal_n(v,w) <C^k \text{ and } H^n_{C^{k},C^{k+1}}(v) >t \text{ and }  H^n_{C^{k},C^{k+1}}(w) >t}. \label{ligne}
\end{align}
The first probability of \eqref{ligne} is bounded from above as follows:
\begin{equation}\label{eq:majoration ligne1}
	\P\left(C^k \le \coal_n(v,w)\le C^{k+1}
	\right)
	\le  1- \! \!  \! \prod_{C^k\le i \le C^{k+1}} \left(1-\1_{\X^n_i=1} \frac{1}{\binom{S^n_i}{2}}\right)
	\le 1-\exp\left( \sum_{C^k \le i \le C^{k+1}} \ln\left(1- \frac{1}{\binom{S^n_i}{2}}\right)\right)
	=\mathcal{O}\left(\frac{1}{C^k}\right),
\end{equation}
where the last equality comes from \eqref{eq:assumptions} and where the $\mathcal{O}(1/C^k)$ is uniform in $v,w$ and in $n\ge C^{k+1}$. 

We then focus on the event  involved in the second probability of \eqref{ligne}. One can write 
\[
H^n_{C^{k},C^{k+1}}(v)=\sum_{C^k<i\le C^{k+1}} Y_i^n,
\] 
where $(Y_i^n)_{1\leq i \leq n}$ are independent Bernoulli random variables of respective parameters $(1/S_i^n\1_{x^n_i=1})_{1\leq i \leq n}$. Moreover, we note that if we know that $Y_i^n=1$ and that $\coal_n(v,w)\le i-1$, then it implies that there is no coalescence with the tree containing $w$. More precisely, one can write 
\[
H^n_{C^{k},C^{k+1}}(w)=\sum_{C^k<i\le C^{k+1}} Z_i^n.
\] 
with 
\[
Z_i^n= Y_i^n\1_{\coal_n(v,w)\ge i} + \xi^n_i \1_{\coal_n(v,w)< i \text{ and } Y_i^n=0},
\]
where the $\xi^n_i$'s are independent Bernoulli r.v. of parameter $(1/(S^n_i-1))\1_{\X^n_i=1}$ respectively and are taken independently from the $Y_i^n$'s. As a result, the second probability in \eqref{ligne} is bounded from above by 
\begin{align}
	\P&\left(\coal_n(v,w) <C^k, \  \sum_{C^k<i\le C^{k+1}} Y_i^n >t \enskip \text{ and }
	\sum_{C^k<i\le C^{k+1}} Z_i^n >t
	 \right)\notag\\
	 &= \P\left(
	 \coal_n(v,w)<C^k, \  \sum_{C^k<i\le C^{k+1}} Y_i^n >t \enskip \text{ and }
	 \sum_{C^k<i\le C^{k+1}}  \xi^n_i \1_{ Y_i^n=0}>t
	 \right) \notag\\
	 &\le \P\left(
	 \sum_{C^k<i\le C^{k+1}} Y_i^n >t \enskip \text{ and }
	 \sum_{C^k<i\le C^{k+1}}  \xi^n_i >t
	 \right) \notag \\
	 &=   \P\left(
	 \sum_{C^k<i\le C^{k+1}} Y_i^n >t \right) \ 
	 \P \left(
	 \sum_{C^k<i\le C^{k+1}}  \xi^n_i >t
	 \right) \le   \P \left(
	 \sum_{C^k<i\le C^{k+1}}  \xi^n_i >t
	 \right)^2. \label{eq:ksi} 
\end{align}
Now, thanks to the fact that $1/S^n_i\le 1/(S^n_i -1)$, one can define some random variables $\widetilde{Y}^n_i$'s for $i\ge 1$ such that for all $i\ge 1$, we have $\widetilde{Y}^n_i \le \xi^n_i$ and $(\widetilde{Y}^n_i)_{i\ge 1}$ is a family of independent Bernoulli random variables of parameter $(1/S^n_i)\1_{\X^n_i=1}$, having thus the same law as  $({Y}^n_i)_{i\ge 1}$. It follows that
\[
	\P\left( 
	\sum_{C^k < i \le C^{k+1}} \xi^n_i > t
	\right)
	-
	\P\left(
	\sum_{C^k < i \le C^{k+1}} Y^n_i > t
	\right)=
	\P\left(
	\sum_{C^k < i \le C^{k+1}} \xi^n_i > t\ge \sum_{C^k < i \le C^{k+1}} \widetilde{Y}^n_i 
	\right)
	\le
	\P\left(
	\sum_{C^k < i \le C^{k+1}} (\xi^n_i-\widetilde{Y}^n_i) >0
	\right)
\]
Thus
\[
		\P\left( 
	\sum_{C^k < i \le C^{k+1}} \xi^n_i > t
	\right)
	-
	\P\left(
	\sum_{C^k < i \le C^{k+1}} Y^n_i > t
	\right)
	\le 
	\E \left[\sum_{C^k < i \le C^{k+1}} (\xi^n_i-\widetilde{Y}^n_i) \right]
	=
	\sum_{C^k< i\le C^{k+1}} \frac{1}{S^n_i(S^n_i-1)}\1_{\X^n_i=1}=\mathcal{O}\left(\frac{1}{C^k}\right),
\]
where the $\mathcal{O}(1/C^k)$ is uniform in $n \ge C^{k+1}$. {By combining this estimate with \eqref{ligne} and \eqref{eq:ksi}, we conclude that
\[
\P\left( H^n_{C^{k},C^{k+1}}(v) >t\text{ and }H^n_{C^{k},C^{k+1}}(w) >t \right)  \leq  \P\left( H^n_{C^{k},C^{k+1}}(v) >t\right) {\P\left( H^n_{C^{k},C^{k+1}}(w) >t\right)}+ \mathcal{O}\left(\frac{1}{C^k} \right)
\]
}
{Therefore, since  $\#\mathcal{A}^n_{C^{k+1}} \leq C^{k+1}+1$,}
%\begin{eqnarray*}
%&& \sum_{ {v \neq w} \in \mathcal{A}_{C^{k+1}}^n} \P\left( H^n_{C^{k},C^{k+1}}(v) >t\text{ and }H^n_{C^{k},C^{k+1}}(w) >t \right)\\
% &&=\sum_{ {v \neq w} \in \mathcal{A}_{C^{k+1}}^n}   \P\left( H^n_{C^{k},C^{k+1}}(v) >t\right) {\P\left( H^n_{C^{k},C^{k+1}}(w) >t\right)}+ \left( \#\mathcal{A}^n_{C^{k+1}} \right)^{2} \mathcal{O}\left(\frac{1}{C^{k}}\right)\\
% &&=
%\end{eqnarray*}
{\begin{eqnarray}
&&\esp{( N^n_k(C,\eta))^2} \notag\\
&& \qquad = \sum_{v \in \mathcal{A}_{C^{k+1}}^n} \P\left( H^n_{C^{k},C^{k+1}}(v) >t\right)+ \sum_{ {v \neq w} \in \mathcal{A}_{C^{k+1}}^n}\P\left( H^n_{C^{k},C^{k+1}}(v) >t\text{ and }H^n_{C^{k},C^{k+1}}(w) >t \right) \notag\\
& & \qquad =   {\mathcal{O}\left({C^{k}}\right)+ \sum_{ {v \neq w} \in \mathcal{A}_{C^{k+1}}^n}   \P\left( H^n_{C^{k},C^{k+1}}(v) >t\right) {\P\left( H^n_{C^{k},C^{k+1}}(w) >t\right)}+  \left( \#\mathcal{A}^n_{C^{k+1}} \right)^{2} \mathcal{O}\left(\frac{1}{C^k}\right)}  \notag \\
&& \qquad = \sum_{ {v \neq w} \in \mathcal{A}_{C^{k+1}}^n}   \P\left( H^n_{C^{k},C^{k+1}}(v) >t\right) {\P\left( H^n_{C^{k},C^{k+1}}(w) >t\right)}+\mathcal{O}\left({C^{k}}\right)\notag\\
&& \qquad =\sum_{ {v , w} \in \mathcal{A}_{C^{k+1}}^n}   \P\left( H^n_{C^{k},C^{k+1}}(v) >t\right) {\P\left( H^n_{C^{k},C^{k+1}}(w) >t\right)}+\mathcal{O}\left({C^{k}}\right) \notag\\
&&\qquad = \esp{ N^n_k(C,\eta)}^2+\mathcal{O}\left({C^{k}}\right).\notag
\end{eqnarray}
This implies \eqref{eq:var} and completes the proof.}
\end{proof}

To establish Lemma \ref{Forestjump} and bound $M^{n}_{k}$ we use the following estimate on P\'olya urns, which may be shown by following verbatim the proof of Lemma A.1 from \cite{Bla22}.

\begin{lemma} 
\label{P\'olya}
Fix an integer $z_{0}>0$ and set $U_{0}=1$. Let $(z_n)_{n\geq 1}\in \{-1,1\}^\N$, set $Z_{n}=\sum_{i=0}^{n}z_{i}$. Assume that $Z_{n}>0$ for every $n \geq 1$ and  $\sum_{n=0}^\infty 1/Z_n^2<\infty$. Let $(U_n)_{n\geq 1}$ be a sequence of random non-negative integers such that for every $n\geq 0$, 
\[ \proba{ \left . U_{n+1}=U_n+z_{n+1} \right | U_n }=\frac{U_n}{Z_n} \quad ; \quad \proba{ \left . U_{n+1}=U_n \right | U_n }=\frac{Z_n-U_n}{Z_n}.\]
Almost surely for every $t \geq 0$
\[ \proba{   \sup_{i\geq 1} \left | \frac{U_i}{Z_i} - \frac{1}{z_0} \right  | >   t\frac{1}{z_0}   } \leq 2 \exp \left (-\frac{(t^2/4)(1/z_0)}{\sum_{n \geq 1} 1/Z_n^2+t\max \left (\sum_{n \geq 1} 1/Z_n^2,  \max_{n \geq 1} 1/Z_n \right )}\right ). \]
\end{lemma}

\begin{proof}[Proof of Lemma \ref{Forestjump}]
{We apply Lemma \ref{P\'olya} with $t=k^2$, $z_{0}=S^{n}_{C^{k}}$, $z_{i}=x^{n}_{C^{k}+i}$ for $1 \le i\le C^{k+1}-C^k$ and $z_i=1$ for all $i>C^{k+1}-C^k$, and $U_{i}$ being the number of active vertices in a given tree of $\F_{C^{k},C^{k}+i}^n$ . Thus,  writing $X_{k}$ for the number of active vertices of a {any fixed} tree of the forest $\F_{C^{k},C^{k+1}}^n$, writing $\delta_k^n\coloneqq\sum_{C_k<n\leq C^{k+1}} 1/(S^n_i)^2+ \sum_{j\ge 1} 1/(S^n_{C^{k+1}} + j)^2$, and $\e_k^n\coloneqq \max_{C^k<i\leq C^{k+1}} 1/S^n_i$,

\[ \proba{X_k/S^n_{C^{k+1}}>  (1+ k^2)/S^n_{C^k}  } \leq 2 \exp \left (-\frac{(k^4/4)/S^n_{C^k}}{\delta^n_k+k^2\max (\delta^n_k,  \e^n_k)} \right ). \]
And by \eqref{eq:c0}, as $k,n\to \infty$ with $A_n\leq C^k\le n/C$, we have $S^n_{C^k}\le C^k+1$, $1/S^n_{C^k}=\mathcal{O}(1/C^k)$, $\e_k^n=\mathcal{O}(1/C^k)$ and $\delta_k^n=\mathcal{O}(1/C^k)$,
thus for every $k$ large enough, 
\[ \proba{X_k>   k^3  } \leq 2 e^{-k^2/\mathcal{O}(1)}. \]
Then, since there are at most $S^n_{C^k}=\mathcal{O}(C^k)$ trees in $\F_{C^{k},C^{k+1}}^n$, for $n$ large enough, for every $k\in \N$, with $A_n \leq C^k\leq  n/C$,
\[ \proba{M^{n}_{k}>  k^3  } \leq S^n_{C^k} 2 e^{-k^2/\mathcal{O}(1)}\leq 1/k^2. \qedhere\]}
\end{proof}

\section{Application: contact-tracing in a stochastic SIR dynamics}
\label{sec:SIR}

Our results can be applied to study the so-called ``infection tree'' of a stochastic SIR dynamics, which is a classical model for the evolution of epidemics (for background on stochastic  epidemic models, see \cite{AB12,BPBLST19}). We assume that initially there is $1$ infectious individual (that has just become infected) and $n$ susceptible individuals. The infectious periods of different infectives are i.i.d  distributed according to an exponential random variable of parameter $1$. During {its} infectious period an infective makes contacts with a given individual at the time points of a time homogeneous Poisson process with intensity $\lambda_{n}$. If a contacted individual is still susceptible, then {it} becomes infectious and is immediately able to infect other individuals. An individual is considered ``removed'' once {its} infectious period has terminated, and is then immune to new infections, playing no further part in the epidemic spread. The epidemic ceases as soon as there are no more infectious individual present in the population. All Poisson processes are assumed to be independent of each other; they are also independent of the infectious periods.

We call a ``step'' an event where either a susceptible individuals becomes infective, or where an {individual's} infectious period terminates. Denote by $\tau_{n}$  the number of steps made when the epidemic ceases. For $0 \leq k \leq \tau_{n}$, let ${\mathscr{T}^{n}_{k}}$ be the infection tree after $k$ steps,  in which the vertices are individuals and where {an edge is present} between two individuals if one has infected the other. We are interested in the evolution of the associated ``infection tree'' $( {\mathscr{T}_k^{n}})_{0 \leq k \leq  \tau_{n}}$, as well as in the shape of the full infection tree $ {\mathscr{T}^{n}_{\tau_{n}}}$   when the epidemic ceases.

Let $(U^{n}_{k}, I^{n}_{k})_{k \geq 0}$ be a Markov chain with initial state $(U^{n}_{0},I^{n}_{0})=(n,1)$ and  transition probabilities given by
\[(U^{n}_{k+1},I^{n}_{k+1}) = \begin{cases}
	(U^{n}_{k}-1,I^{n}_{k}+1)  &\textrm{with probability } \frac{\lambda_{n} U^{n}_{k}}{1+ \lambda_{n} U^{n}_{k}}\\
	(U^{n}_{k},I^{n}_{k}-1)  &\textrm{with probability } \frac{1}{1+\lambda_{n} U^{n}_{k}}
\end{cases}
\]
with  $ \{(k,0) : 0 \leq k \leq n\}$ as absorbing states {(we use ``$U$'' for ``uninfected'')}.  Observe that evolution of   the number of susceptible individuals and the number of infectious individuals in the infection tree evolves according to this Markov chain. Then define the random sequence $\XXb^{n}=(\XX^{n}_{i})_{1 \leq i \leq \tau'_{n}}$ of $\pm 1$ as follows:  let $\tau'_{n}$ be the absorption time of the Markov chain, and for $1 \leq i \leq \tau'_{n}$ set $\XX^{n}_{i}=I^{n}_{i}-I^{n}_{i-1}$.

Then by construction, it is clear that:
\[
({\mathscr{T}^{n}_{k}})_{0 \leq k \leq \tau_{n}}  \quad \mathop{=}^{(d)}  \quad  (\mathcal{T}_{k}(\XXb^n))_{ 0 \leq k \leq \tau'_{n}}. 
\]

To simplify notation, let $ {\mathscr{T}^{n}}= {\mathscr{T}^{n}_{\tau_{n}}}$ be the full infection tree when the epidemic ceases. We also identify the random variables in the above equality in distribution. {For all $p \in [0,1)$, we denote by $\mathcal{G}(1-p)$ the geometric distribution $\mu$ given by $\mu(k) = (1-p)p^k$ for every integer $k\ge 0$.}
\begin{theorem}
	\label{thm:SIRlocallimit}
	The following assertions hold.
	\begin{enumerate}
		\item[{(I)}] Assume that   $\lambda_n \sim\lambda/n$ for some $\lambda >0$. Then $ {\mathscr{T}^n}$ converges in distribution locally towards a Bienaym\'e tree with offspring distribution $\mathcal{G}(1/(1+\lambda))$. 
		\item[{(II)}] Assume that  $\lambda_n \gg 1/n$. Then $ {\mathscr{T}^n}$ is not tight for the topology of the local convergence.
	\end{enumerate}
\end{theorem}

\begin{proof}
	
	For {(I)},  assume that $\lambda_n \sim \lambda/n$ as $n \to \infty$ for some $\lambda >0$. {We first prove} that $(I^n_k)_{k\ge 0}$ converges in distribution for the product topology towards the random walk $(S_k)_{k\ge 0}$ such that $S_{k+1}-S_k=1$ with probability $\lambda/(1+\lambda)$ and $S_{k+1}-S_k=-1$ with probability $1/(1+\lambda)$. Indeed, the local convergence of $\T_n(\XXb^n)$ implies the local convergence of ${\mathscr{T}^n}$. So as to achieve this, one can first see that since $U^n_k \le n$ for all $k \ge 0$, we have
	\[
	\frac{\lambda_n U^n_k}{1+\lambda_n U^n_k} \le \frac{n \lambda_n}{1+ n \lambda_n} \cv[n] \frac{\lambda}{1+\lambda}.
	\]
	But we also have
	\[
	\inf_{k \le \sqrt{n}} \frac{\lambda_n U^n_{k}}{1+ \lambda_n U^n_{k}}
	\ge \frac{\lambda_n (n-\sqrt{n})}{1+ n\lambda_n } \cv[n] \frac{\lambda}{1+\lambda}.
	\]
	For all $k_0 \ge 0$, one can thus build a coupling between $(I^n_k)_{0\le k\le k_0}$ and $(S_k)_{0\le k\le k_0}$ such that the two walks coincide w.h.p. until time $k_0$ as $n \to \infty$. 
	
	{Thus, by Skorokhod's representation theorem we may assume that almost surely, for every $k_{0} \geq1$, $(I^n_k)_{0\le k\le k_0}=(S_k)_{0\le k\le k_0}$ for $n$ sufficiently large. By Theorem \ref{thm:cvlocale}, ${\mathscr{T}^n}$ converges locally in distribution to a uniform attachment tree with freezing built using the sequence $(S_{k+1}-S_k)_{k \geq 0}$. By Theorem \ref{thm:BGW}, this is a Bienaym\'e tree with offspring distribution $\mathcal{G}(1/(1+\lambda))$. This proves {(I)}.}
	
{For {(II)}, we argue by contradiction and assume that $({\mathscr{T}^n})_{n \geq 0}$ is tight. By Prokhorov's theorem, let $\varphi$ be an extraction such that $(\mathscr{T}^{\varphi(n)})_{n \geq 0}$   converges in distribution. Similarly to the proof of (I), observe that  $(I^{\varphi(n)}_k)_{k\ge 0}$ converges in distribution with respect to the product topology towards $(k+1)_{k\ge 0}$ since for all fixed $k \ge 0$, the transition probability ${\lambda_n U^{\varphi(n)}_k}/({1+\lambda_{\varphi(n)} U^{\varphi(n)}_k})$ converges to $1$ as $n\to \infty$. Thus, by Skorokhod's representation theorem we may assume that almost surely, for every $k_{0} \geq1$, $(I^{\varphi(n)}_k)_{0\le k\le k_0}=(k+1)_{0\le k\le k_0}$ for $n$ sufficiently large. But if $\mathbf{x}_{i}=1$ for every $i \geq 1$, the sum  $\sum_{i\geq 1} \frac{1}{S_i(\Xb)} \mathbbm 1_{\{\X_i=-1\}}$ converges, so by Theorem \ref{thm:cvlocale} the sequence $({\mathscr{T}^{\varphi(n)}})_{n \geq 0}$ does not converge locally in distribution, a contradiction.}
\end{proof}

\begin{theorem}
	\label{thm:SIR}
	The following assertions hold.
	\begin{enumerate}
		\item[{(I)}] Assume that $\lambda_n \sim\lambda/n$ for some $\lambda> 1$. Let $A$ be the event of survival of a Bienaym\'e process with offspring distribution $\mathcal{G}(1/(1+\lambda))$.
		\begin{enumerate}
			\item[(a)]The convergence
			\begin{equation}
				\label{limite fluide I}\left(\frac{I^n_{\lfloor nt \rfloor}}{n}\right)_{t \ge 0} \cvloi 
				\left( \max(2-2g_\lambda(t)-t,0) \mathbbm{1}_{A}\right)_{t\ge 0},
			\end{equation}
			holds in distribution for the topology of uniform convergence on compact sets, where {$g_\lambda$} is the solution of the ordinary differential equation $g_\lambda'(t)=-\lambda g_\lambda(t)/(1+\lambda g_\lambda(t))$ with $g_\lambda(0)=1$.
			\item[(b)] For all $n\ge 1$, conditionally given  $\mathscr{T}^n$, let $V_{1}^n$   and $V_{2}^{n}$ be independent uniform vertices of $\mathscr{T}^n$. Then
			\begin{equation}
				\label{eq:cvprobaHn SIR}
				\frac{\H(V_{1}^n)}{\ln n} \cvloi \frac{\lambda}{\lambda-1}  \mathbbm{1}_{A}\qquad \textrm{and} \qquad  \frac{d^n(V^n_1,V^n_2)}{\ln n} \cvloi  \frac{2\lambda}{\lambda-1}  \mathbbm{1}_{A},
			\end{equation}
			where $\H(V_{1}^n)$ is the height of $V_{1}^{n}$ in $\mathscr{T}^n$ and $d^{n}$ denotes the graph distance in $\mathscr{T}^n$.
		\end{enumerate}
		\item[{(II)}] Assume that  $\lambda_n \gg 1/n$. 
		\begin{enumerate}
			\item[(a)] The convergence
			\begin{equation}\label{eq:limite fluide presque recursif}
				\left(\frac{I^n_{\lfloor nt \rfloor}}{n}\right)_{0\le t \le 2} \cvproba (\min(t,2-t))_{0\le t \le 2},
			\end{equation}
			holds in probability for the topology of uniform convergence.
			\item[(b)] We have
			\begin{equation}\label{eq:cv Hn cas presque recursif} \frac{\H(V_{1}^n)}{\ln n} {\cvproba}  1 , \qquad  \frac{d^n(V^n_1,V^n_2)}{\ln n} {\cvproba} 2 \qquad \textrm{and} \qquad \frac{\Height(\mathscr{T}^n)}{\ln(n)} \cvproba e.
			\end{equation}
		\end{enumerate}
	\end{enumerate}
\end{theorem}

\begin{figure}[!ht]
	\label{fig:SIR2}
	\centering
	\includegraphics[scale=0.35]{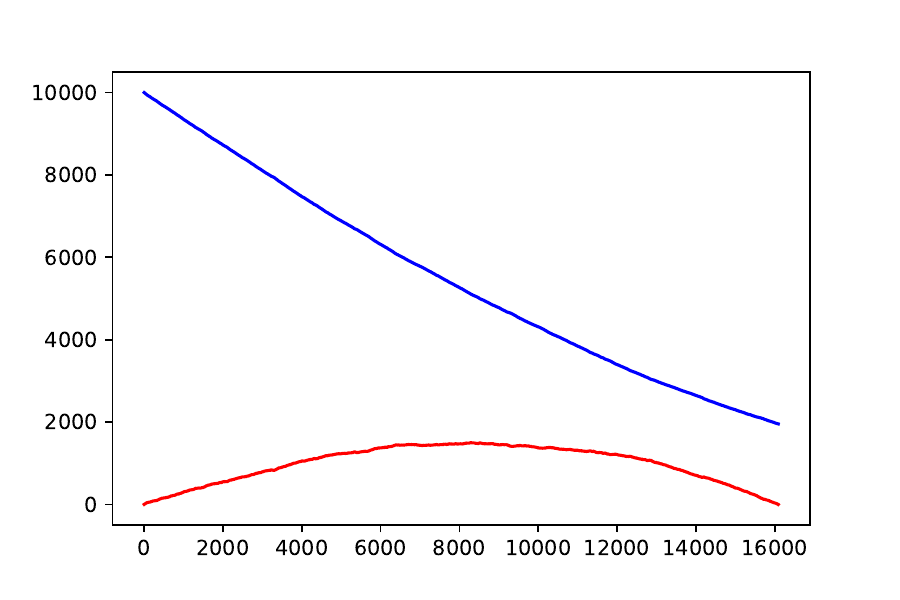}
	\includegraphics[scale=0.42]{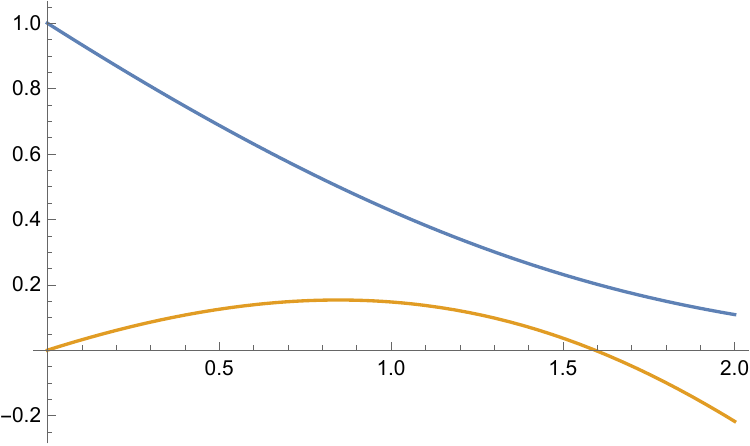}
	\includegraphics[scale=0.35]{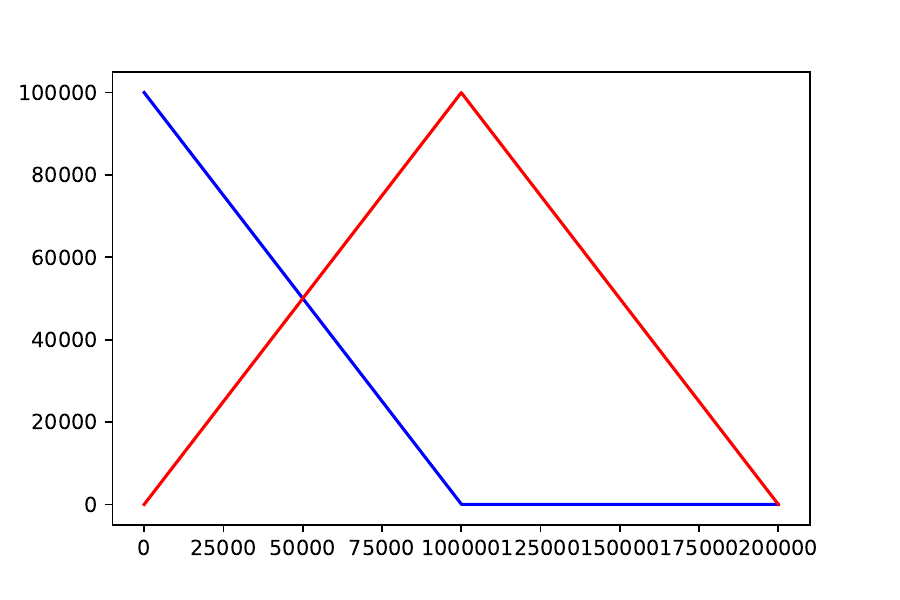}
	\caption{Left: simulation of $(U^{n},I^{n})$ for $\lambda_n=2/n$ and $n=10000$ ($U^{n}$ in blue and $I^{n}$ in red). Centre: the corresponding fluid limit. Right: simulation of $(U^{n},I^{n})$ for $\lambda_n=2$ and $n= 100000$.}
\end{figure}

{We observe that the solution $g_\lambda$ of the differential equation can be expressed using the principal branch of the Lambert $W$ function defined as the inverse function $W: [-1/e, \infty)\to [-1,\infty)$ of the function $x \mapsto x e^x$ which is increasing on $[-1,\infty)$. See e.g.\@ \cite{CGHJK96} for more details about this function.  Explicitly, using the fact that $W'(x) =  W(x)/(x(1+W(x)))$, we check that for all $t \ge 0$,
	\begin{equation}\label{eq fonction de Lambert}
		g_\lambda(t) = \frac{1}{\lambda} W(\lambda e^{\lambda}e^{-\lambda t}).\end{equation}
In particular, $g_\lambda$ is infinitely differentiable and one can see that $g_\lambda$ is convex, decreasing and goes to $0$ at $\infty$.}

{ 
	\begin{remark}
		\label{rem:SIR}
		It would be very interesting to obtain a limit theorem for $\Height(\mathscr{T}^n)$ in the setting of Theorem \ref{thm:SIR} {(I)} (b). This requires more work, because the behavior of the number of infected/active vertices near the extinction time $\tau'_{n}$ should give a non-negligible contribution to the height.
	\end{remark}
}

\begin{proof}[Proof of Theorem \ref{thm:SIR}]
	For {(I)} (a), we assume $\lambda>1$ and we determine the fluid limit of the chain $(U^n,I^n)$. Let $(\widetilde{U}^n,\widetilde{I}^n)$ be a Markov chain which has the same initial conditions as $(U^n,I^n)$ and the same transition probabilities but which is not stopped when $\widetilde{I}^n$ reaches zero. 
	We check that \eqref{limite fluide I} holds using the classical theory of fluid limits of Markov chains (see e.g. \cite{Kur81} or \cite{DN08}), and more precisely a combination of Grönwall's inequality with Doob's maximal inequality. 
	
	We claim that it is enough to show that
	\begin{equation}\label{limite fluide H}
		\left(\frac{\widetilde{U}^n_{\lfloor nt \rfloor}}{n}\right)_{t \ge 0} \cvproba[n] 
		\left(g_\lambda(t)\right)_{t\ge 0},
	\end{equation}
	uniformly on compacts. Indeed, it is easy to see that $\widetilde{I}^n_k= 2(n-\widetilde{U}^n_k) -k+1$ so that \eqref{limite fluide H} entails the convergence
	\begin{equation}\label{limite fluide I tilde}
		\left(\frac{\widetilde{I}^n_{\lfloor nt \rfloor}}{n}\right)_{t \ge 0} \cvproba[n] 
		\left(2-2g_\lambda(t)-t \right)_{t\ge 0},
	\end{equation}
	uniformly on compacts. Let us check that this implies  \eqref{limite fluide I}. Since $t\mapsto 2- 2g_\lambda(t)-t$ is concave, starts from zero at time zero, stays positive and then goes below zero, it is thus enough to identify the limit of the probability that the chain $\widetilde{I}^n$ reaches zero before time $\e n$ as {first} $n\to \infty$ and {then} $\e \to 0$ as  the probability of extinction of a $\mathcal{G}(1/(1+\lambda))$-Bienaym\'e process. To this end, for all $\e >0$, for $n$ large enough, for all $k\le \e n$, since $n\ge \widetilde{U}^n_k\ge (1-\e)n$ we have
	\[
	\frac{\lambda}{1+\lambda} -\e  \le 
	(1-\e)\frac{n \lambda_n}{1+ n \lambda_n} \le \frac{\lambda_n \widetilde{U}^n_k}{1+\lambda_n \widetilde{U}^n_k}
	\le \frac{n \lambda_n}{1+ n \lambda_n} 
	\le \frac{\lambda}{1+\lambda} +\e.
	\]
	Thus, one can couple $(I^n, U^n)$ until time $\e n$ with two simple random walks $(\overline{Y}^\e_k)_{k\ge 0}$ and $(\underline{Y}^\e_k)_{k\ge 0}$ starting from one such that
	\[
	\P(\overline{Y}^\e_1 = 1) =  \frac{\lambda}{1+\lambda}  + \e \qquad \text{and} \qquad \P(\underline{Y}^\e_1 = 1) =\frac{\lambda}{1+\lambda} -\e
	\]
	so that for all $n$ large enough, for all $k\le \e n$,
	\begin{equation}\label{eq couplage marches}
		\underline{Y}^\e_k \le \widetilde{I}^n_k \le \overline{Y}^\e_k.
	\end{equation}
	Since the probabilities that the random walks $\underline{Y}^\e$ and $\overline{Y}^\e$ starting from $1$ reach the negative integers {both} converge as $\e \to 0$ to the probability of extinction of a $\mathcal{G}(1/(1+\lambda))$-Bienaym\'e process, we get the desired result.

	It remains to check \eqref{limite fluide H}. This is a rather direct application of Lemma 6.5 in \cite{Cur} (which remains clearly true if we add a superscript $n$ to all the objects). More precisely, let $t>0$, for all $k\ge 0$ let $X^n_k=\widetilde{U}^n_k - ng_\lambda(k/n)$ and let $(\mathcal{F}^n_k)_{k\ge 0}$ be the natural filtration associated with $X^n$. This lemma states {(in greater generality)} that if there exists a constant $C>0$ such that for all $n\ge 1$, for all $0\le k \le nt$, 
	\begin{equation}\label{hypothèse Gronwall stochastique}
		\left\lvert
		\E
		\left(\left. X^n_{k+1}-X^n_k \right\vert \mathcal{F}^n_k \right)
		\right\rvert
		\le \frac{C}{n} (1+\sup_{0\le j \le nt} \lvert X^{n}_j \rvert) \qquad
		\text{a.s. and} \qquad
		\E\left( \left( X^n_{k+1}-X^n_k\right)^2\right) \le C,
	\end{equation}
	then 
	$
	\sup_{0\le k \le nt}  {\lvert X^n_k \rvert}/{n} \rightarrow 0
	$
	in probability as $n \to \infty$, {which implies} \eqref{limite fluide H}. 
	
	To show \eqref{hypothèse Gronwall stochastique}, we first compute
	\[
	\E\left( \left.  X^n_{k+1}-X^n_k \right\vert \mathcal{F}^n_k \right)
	=- \frac{\lambda \widetilde{U}^n_k / n}{1 + \lambda \widetilde{U}^n_k/n} - n\left(g_\lambda\left(\frac{k+1}{n} \right) - g_\lambda\left( \frac{k}{n}\right) \right).
	\]
	{Moreover, since the function $g_\lambda$ is twice differentiable, } using {a Taylor expansion} at order $2$ and using the definition of $g_\lambda$ one gets that{
		\begin{align*}
			\left\lvert
			\E
			\left( \left.
			X^n_{k+1}-X^n_k + \frac{\lambda \widetilde{U}^n_k / n}{1 + \lambda \widetilde{U}^n_k/n} -\frac{\lambda g_\lambda(k/n)}{1+\lambda g_\lambda(k/n)}
			\right\vert \mathcal{F}^n_k
			\right)
			\right \rvert 
			&= n\left\vert  g_\lambda\left(\frac{k+1}{n} \right) - g_\lambda\left( \frac{k}{n}\right) + \frac{1}{n}\frac{\lambda g_\lambda(k/n)}{1+\lambda g_\lambda(k/n)} \right\vert\\
			&\le n \frac{1}{n^2} \sup_{0\le s \le t+1} \vert g''_\lambda(s)\vert\\
			&\le \frac{C}{n},
	\end{align*}}
	for some constant $C>0$.
	Moreover, the function $x \mapsto x/(1+x)$ is $1$-Lipschitz on $[0,\infty)$. Thus, the first inequality of \eqref{hypothèse Gronwall stochastique} is satisfied. The second one is clearly satisfied since $ \widetilde{U}^n_{k+1}-\widetilde{U}^n_k \in \{0,-1\}$ and since $g_\lambda$ is differentiable {and $g'_\lambda$ is bounded on $[0,t+1]$}.
	
	We now turn to {(I)} (b). We essentially apply Theorem \ref{thm:lineaire} with $c=(\lambda-1)/(\lambda+1)$ and $A_n= \log \log n$. However, the rub is that the fluid limit {$t\mapsto \max(2-2g_\lambda(t)-t, 0)$ of $I^n$ reaches zero at some time $t_0$ which is the unique positive time at which $2-2g_\lambda(t_0)-t_0=0$} so that we can not apply it directly. To overcome this issue, we shall check that we are in position to apply Theorem \ref{thm:lineaire} to the subtree $\T_{\lfloor(t_0-\e) n\rfloor}(\XXb^n)$. This will indeed enable us to conclude since if $V^n_1$ and $V^n_2$ are two independent uniform random vertices of ${\mathscr{T}^n}$, then by Lemma \ref{lem:bunif},
	\[
	\P(\birth_{\tau'_n}(V^n_1) \le(t_0-\e) n) = \P(\birth_{\tau'_n}(V^n_2) \le (t_0-\e) n) \ge 1-\e+o(1) \enskip \text{ as } n\to \infty.
	\]
	{By Skorokhod's representation theorem, we may assume that the convergence \eqref{limite fluide I} holds almost surely. Therefore, the right-hand side of \eqref{eq:assumptions} then holds almost surely on the event that $I^n_k>0$ for all $k\le (t_0-\e)n$ thanks to \eqref{limite fluide I}.} To check that Theorem \ref{thm:lineaire} applies to the subtree $\T_{\lfloor(1-\e)\tau'_n\rfloor}(\XXb^n)$ with $c=(\lambda-1)/(\lambda+1)$, it {thus suffices to show that the left hand side of \eqref{eq:assumptions} holds almost surely on the event that $I^n_k>0$ for all $k\le (t_0-\e)n$ with $A_n= \log \log n$. By Skorokhod's representation theorem it is enough to check} that for all $\e>0$, there exists $\delta >0$ such that on the event {that} $I^n_k>0$ for all {$k\le (t_0-\e) n$},
	\begin{equation}\label{eq demarrage I}
		\max_{\log \log n \le k \le \delta n} \left| \frac{I^n_k}{k} - \frac{\lambda-1}{\lambda+1}\right| \le \e 
	\end{equation}
	{with high probability} as $n\to \infty$. In turn, it is enough to prove that for every $\varepsilon>0$ we have
	\begin{equation}\label{eq demarrage I tilde}
		\max_{\log \log n \le k \le \e n} \left| \frac{\widetilde{I}^n_k}{k} - \frac{\lambda-1}{\lambda+1}\right| \le 2\e 
	\end{equation}
	{with high probability} as $n\to \infty$. {Indeed, note that by definition of $g_\lambda$, we have $g_\lambda'(0)= -\lambda/(\lambda+1)$ so that as $t \to 0$,
\begin{equation}
\label{eq:gbound}		2-2g_\lambda(t)-t  =\frac{\lambda -1}{\lambda+1} t  + o(t).
\end{equation}
		Let $\e>0$. Take $\delta>0$ small enough that for all $t \in [0,\delta]$, we have $\vert2-2g_\lambda(t)-t- ((\lambda-1)/(\lambda+1))t\vert \le \e t$. On the event that $I^n_k>0$ for all $k\le (t_0-\e) n$,
		$$
		\max_{\log \log n \le k \le \delta n} \left| \frac{I^n_k}{k} - \frac{\lambda-1}{\lambda+1}\right|  \le \max_{\log \log n \le k \le \e n} \left| \frac{\widetilde{I}^n_k}{k} - \frac{\lambda-1}{\lambda+1}\right| + \max_{\e n \le k \le \delta n} \left| \frac{\widetilde{I}^n_k}{k} - \frac{\lambda-1}{\lambda+1}\right|.
		$$
		The second term is bounded from above by $2\e$ with high probability as $n\to \infty$ thanks to \eqref{limite fluide I tilde} and \eqref{eq:gbound}. Thus, \eqref{eq demarrage I tilde} implies \eqref{eq demarrage I}.
	}
	The upper bound \eqref{eq demarrage I tilde} follows by combining the  coupling \eqref{eq couplage marches} with the law of large numbers.
	
	We now turn to {(II)} (a) and assume $\lambda_n \gg 1/n$. Let us first establish that if $\delta >0$, then
	\begin{equation}\label{limite fluide jusqu'au temps 1}
		\left(\frac{I^n_{\lfloor nt \rfloor}}{n}\right)_{0\le t \le 1} \cvproba[n] (t)_{0\le t \le 1}
		\qquad \text{and} \qquad \sup_{\log \log n \le k \le (1-\delta)n} \left|\frac{I^n_k}{k}-1\right| \cvproba[n] 0,
	\end{equation}
	where the first convergence holds for the topology of the uniform convergence.
	Let $\delta >0$. Note that for all $k \le (1-\delta)n$ we have $U^n_k \ge \delta n$. As a result, for all $\e >0$, for $n$ large enough, for all $k\le (1-\delta)n$.
	$$
	1-\e \le \frac{\lambda_n U^n_k}{1+\lambda_n U^n_k} \le 1.
	$$
	One can thus obtain the {right-hand side} of \eqref{limite fluide jusqu'au temps 1} by coupling $I^n_k$ with a simple random walk (whose step is $+1$ with probability $1-\e$). {The left-hand side of \eqref{limite fluide jusqu'au temps 1} stems from the right-hand side and from the fact that $I^n$ only makes $\pm 1$ steps, so that for all $k \in \llbracket \lfloor(1-\delta)n \rfloor, n \rrbracket$, we have $\vert I^n_k - I^n_{\lfloor (1-\delta)n \rfloor} \vert \le n\delta$.}
	
	We now explain why \eqref{limite fluide jusqu'au temps 1} implies \eqref{eq:limite fluide presque recursif}.  By \eqref{limite fluide jusqu'au temps 1} {and by definition of the Markov chain $(I^n,U^n)$}, we have  $U^n_n {\le n+1-I^n_n }=o(n)$ with high probability, so that after time $n$, w.h.p.\@ there are $o(n)$ remaining steps of $+1$ for $I^n$ and all the other steps are $-1$. Hence, for the topology of the uniform convergence,
	\[
	\left(\frac{I^n_{\lfloor nt \rfloor}}{n}\right)_{1\le t \le 2} \cvproba (2-t)_{1\le t \le 2},
	\]
	{which} concludes the proof of \eqref{eq:limite fluide presque recursif}. 
	
	We finally turn to {(II)} (b). As in the proof of {(I)} (a), the first two convergences of \eqref{eq:cv Hn cas presque recursif} are proved by applying Theorem \ref{thm:lineaire} to the tree $\T_{\lfloor(2-\e)n \rfloor}(\XXb^n)$ for some $\e>0$ small enough, {with $c=1$ and $A_n= \log \log n$. Indeed, by \eqref{limite fluide jusqu'au temps 1},
		\begin{equation}\label{eq lgn infection cas recursif}
			\sup_{\log \log n \le k \le n} \left|\frac{I^n_k}{k}-1\right| \cvproba[n] 0.
		\end{equation}
		After applying Skorokhod's representation theorem, \eqref{eq lgn infection cas recursif} implies the left hand side of \eqref{eq:assumptions} and the right hand side follows from \eqref{eq:limite fluide presque recursif}.}
		
	 The last convergence of {(II)} (b) is more delicate. The idea is to apply Theorem \ref{thm:lineaire} to the tree $\T_n(\XXb^n)$ and then to show that the vertices added between times $n$ and $\tau'_n$, in a first approximation, do not affect the height. 
	
	More precisely, one can first see that by Theorem \ref{thm:lineaire} the height of $\T_n(\XXb^n)$ divided by $\ln n$ converges in probability towards $e$. {
		Indeed, by \eqref{eq lgn infection cas recursif}, after applying Skorokhod's representation theorem, entails that \eqref{eq:assumptions} holds almost surely for $\T_n(\XXb^n)$ with $c=1$ and $A_n= \log \log n$.
	}
	Next, to understand how the height of ${\mathscr{T}^n}$ behaves compared to $\T_{n}(\XXb^n)$ we need to estimate the quantity
	\[
	\h^n \coloneqq \sum_{i=n}^{\tau'_n-1} \frac{1}{I_i^n} \mathbbm{1}_{\{\XX^n_i=1\}}.
	\]
	Indeed, as we have already seen multiple times, roughly speaking this quantity  represents ``the height'' added between times $n$ and $\tau'_n$. 
	%In a similar way that we have proved that $(I_i^n)_{0\leqslant i\leqslant n}$ satisfies the assumptions \eqref{eq:assumptions} with high probability, one shows that $(\hat S_i^n)_{0\leqslant i\leqslant \tau'_n-n}$ also satisfies those assumptions with high probability. 
	
	{Let us prove that
		\begin{equation}\label{eq convergence des transitions cas presque recursif}
			\sup_{3n/2 \le k \le 2n} \frac{\lambda_n U^n_k}{1+ \lambda_n U^n_k} \le \frac{\lambda_n U^n_{\lfloor 3n/2 \rfloor}}{1+ \lambda_n U^n_{\lfloor 3n/2 \rfloor}} \cvproba[n] 0.
		\end{equation}
		The inequality holds since $(U^n_k)_{k \ge 0}$ is non-increasing. For the above convergence, assume by contradiction that there exists $\e>0$ and an increasing sequence of integers $(n_j)_{j\ge 1}$ such that for all $j\ge 1$,
		$$
		\P \, \left( \frac{\lambda_{n_j} U^{n_j}_{\lfloor 3n_j/2 \rfloor}}{1+ \lambda_{n_j} U^{n_j}_{\lfloor 3n_j/2 \rfloor}} \ge \e \right) \ge \e.
		$$
		On this event, by monotonicity for every $k \in \llbracket n_j , \lfloor 3n_j/2 \rfloor \rrbracket$ we have 
		$$\frac{\lambda_{n_j} U^{n_j}_k}{1 + \lambda_{n_j} U^{n_j}_k} \ge \e.$$
		Then, using a coupling between  $(I^{n_j}_k)_{k\ge n}$ and a simple random walk $(\widetilde{Y}^\e_{k})_{k\ge n_j}$ starting at $I^{n_j}_{n_j}$ at time $n_j$ whose step is $+1$ with probability $\e$ and $-1$ with probability $1-\e$, by the law of large numbers and since $I^n_n/n \to 1$ in probability as $n\to \infty$ by \eqref{eq:limite fluide presque recursif}, one obtains that on the same event,
		$$
		\frac{I^{n_j}_{\lfloor 3n_j/2\rfloor }}{n_j} \ge \frac{\widetilde{Y}^\e_{\lfloor 3n_j/2\rfloor}}{n_j}  \cvproba[n] 1+ \frac{2\e -1 }{2},
		$$
		which is absurd in view of \eqref{eq:limite fluide presque recursif}. This proves \eqref{eq convergence des transitions cas presque recursif}.
		
		Next, let us prove that 
		\begin{equation}\label{eq h n tend vers zero}
			\frac{\h^n}{\log n} \cvproba[n]0.
		\end{equation}
		Define $\hat\XX_i^n = -\XX_{\tau'_n-i+1}^n$ {for all $i \in \llbracket 1, \tau'_n \rrbracket$} and $\hat S_i^n = I_{\tau'_n-i}^n = \hat\XX_1^n + \cdots + \hat\XX_i^n$ for all $i \in \llbracket 0, \tau'_n \rrbracket$. Then
		\[
		\h^n = \sum_{i=1}^{\tau'_n-n} \frac{1}{\hat S_i^n} \mathbbm{1}_{\{\hat\XX^n_i=-1\}}.
		\]
		By \eqref{eq:limite fluide presque recursif}, to prove \eqref{eq h n tend vers zero}, it suffices to prove that 
		\begin{equation}\label{eq somme jusque n sur 2 tend vers zero}
			\frac{1}{\log n} \sum_{i=1}^{\lfloor n/2 \rfloor \wedge (\tau'_n -n)}  \frac{1}{\hat S_i^n} \mathbbm{1}_{\{\hat\XX^n_i=-1\}} \cvproba[n] 0.
		\end{equation}
		For all $\e>0$, for all $n\ge 1$, define the event
		$$
		\mathcal{A}^n_\e = \left\{ \sup_{3n/2 \le k \le 2n} \frac{\lambda_n U^n_k}{1+ \lambda_n U^n_k} \le \e\right\}.
		$$
		By \eqref{eq convergence des transitions cas presque recursif}, we know that
		\begin{equation}\label{eq limite proba A n epsilon tend vers 1}
			\lim_{\e \to 0} \liminf_{n\to \infty} \P(\mathcal{A}^n_\e)  = 1.
		\end{equation}
		Besides, on the event $\mathcal{A}^n_\e$, one can couple the $(\hat{\XX}^n_i)_{1 \le i \le n/2 \wedge (\tau_n -n)}$'s with i.i.d.\@ random variables $(\xi^\e_i)_{i\ge 1}$ such that $\P(\xi^\e_i = -1)= \e$ and $\P(\xi^\e_i = 1) = 1-\e$ so that for all $i \le  n/2\wedge (\tau_n-n)$, we have $\xi^\e_i \le \hat{\XX}^n_i \le 1$. For all $i \ge 0$, let $\hat{Y}^\e_i = \xi^\e_1+ \ldots + \xi^\e_i$. Then, for all $i \le n/2  \wedge (\tau_n-n)$, we have $\hat{Y}^\e_i \le \hat{S}^n_i $. In particular, on the event $\mathcal{A}^n_\e \cap \{\forall i \in \llbracket 1, n/2 \rrbracket , \ \hat{Y}^\e_i >0 \}$,
		$$
		\sum_{i=1}^{\lfloor n/2 \rfloor \wedge (\tau'_n -n)}  \frac{1}{\hat S_i^n} \mathbbm{1}_{\{\hat\XX^n_i=-1\}}
		\le
		\sum_{i=1}^{\lfloor n/2 \rfloor}  \frac{1}{\hat{Y}_i^\e} \mathbbm{1}_{\{\xi^\e_i=-1\}}.
		$$
		But, by the strong law of large numbers, we have
		\begin{equation}\label{eq limite proba Y epsilon toujours plus grand que i sur 2}
			\P\left( \forall i \ge 0, \ \hat{Y}^\e_i \ge  i/2\right)
			\mathop{\longrightarrow}\limits_{\e\to 0} 1 .
		\end{equation}
		For all $\e>0$, on the event $\mathcal{A}^n_\e \cap \{\forall i \ge 0, \hat{Y}^\e_i \ge i/2 \}$, we obtain that 
		$$
		\sum_{i=1}^{\lfloor n/2 \rfloor \wedge (\tau'_n -n)}  \frac{1}{\hat S_i^n} \mathbbm{1}_{\{\hat\XX^n_i=-1\}}
		\le
		\sum_{i=1}^{\lfloor n/2 \rfloor}  \frac{2}{ i} \mathbbm{1}_{\{\xi^\e_i=-1\}}.
		$$
		Besides,
		$$
		\E\left[\sum_{i=1}^{\lfloor n/2 \rfloor}  \frac{2}{i} \mathbbm{1}_{\{\xi^\e_i=-1\}}\right] =2\e \sum_{i=1}^{\lfloor n/2 \rfloor} \frac{1}{i}.
		$$
		Therefore, we obtain that
		$$
		\E \left[\1_{\mathcal{A}^n_{\e} \cap \{\forall i \ge 0, \ \hat{Y}^{\e}_i \ge  i/2\}}\sum_{i=1}^{\lfloor n/2 \rfloor \wedge (\tau'_n -n)}  \frac{1}{\hat S_i^n} \mathbbm{1}_{\{\hat\XX^n_i=-1\}}
		\right] \le 2\e \sum_{i=1}^{\lfloor n/2 \rfloor} \frac{1}{i}.
		$$
		After letting $\e\to 0$, using \eqref{eq limite proba A n epsilon tend vers 1} and \eqref{eq limite proba Y epsilon toujours plus grand que i sur 2}, we deduce \eqref{eq somme jusque n sur 2 tend vers zero}.
	}

	%Therefore, one could easily adapt the convergence of Theorem \ref{thm:lineaire} {(I)} to show that that $\h^n$ divided by $\ln n$ converges to $0$ in probability (indeed, if in Theorem \ref{thm:lineaire} we look at $\h_n^-$ instead of $\h_n^+$, where $\h_n^-$ consists of the sum over the $-1$ steps instead of the $+1$ steps, then, we have that $\h_n^-$ divided by $\ln n$ converges to $(1-c)/(2c)$). 
	It remains to show how {\eqref{eq h n tend vers zero}} implies that $\Height({\mathscr{T}^n})=\Height(\T_{n}(\XXb^n))+o(\ln n)$. To do so we will use, once again, the coalescence construction and Bennett's inequality. By Algorithm \ref{algo2}, conditionally on $\tau'_n$ and $(I_i^n)_{0\leqslant i\leqslant \tau'_n}$, the height of any new vertex added between time $n$ and $\tau'_n$ is stochastically dominated by $\Height(\T_n(\XXb^n))+\sum_{i=n}^{\tau'_n-1} Y_i$ where $(Y_i^n)_{n\leqslant i\leqslant \tau'_n-1}$ are i.i.d Bernoulli random variables of respective parameters $(1/I_i^n\mathbbm{1}_{\{\XX^n_i=1\}})_{n\leqslant i\leqslant \tau'_n-1}$. Using Bennett's inequality and a union bound{, as in Sec.~\ref{ssec:limitheightBennet},}  for all $\e>0$ we obtain that
	\[
	\proba{\Height({\mathscr{T}^n}) - \Height(\T_n(\XXb^n)) \geq \h^n + \e \ln n \, \middle| \,  {\tau'_n, (I_i^n)_{0\leqslant i\leqslant \tau'_n}} } \leq n\exp\left(-\h^n g\left(\frac{\e \ln n}{\h^n}\right)\right) 
	\]
	where $g(u) = (u+1)\ln(u+1)-u \sim u\ln u$ as $u\to+\infty$. Since $\h^n=o(\ln n)$ {in probability}, the above upper bound goes to $0$ when $n$ goes to $\infty$ and thus $\Height({\mathscr{T}^n})$ is of order $e\ln n$.
	\end{proof}

\section{Perspectives and extensions} \label{oo/oo}

We mention here some perspectives for future research in the direction of studying the impact of freezing in random graph models.

\begin{enumerate}[noitemsep,nolistsep]
\item[(1)] What is the number of ends of $\T_{\infty}({\mathbf{x}})$? 
\item[(2)] It would be interesting to study the evolution of the sizes of the trees in the forests obtained by Algorithm \ref{algo2} (in the particular case where $\Xb_{n}=1$ for every $n \geq 1$ this corresponds to Kingman's coalescent), as well as the evolution of degree vertices (in the particular case where $\Xb_{n}=1$ for every $n \geq 1$ there are nice connections with P\'olya urns). 
\item[(3)] {How small can be the typical height of $\T_{n}(\X)$ when $\tau(\X) \geq n$? We conjecture that for any sequence $(\X^{n})_{n \geq 1}$ such that  $\tau(\X^{n}) \geq n$ for every $n \geq1$, for every $\varepsilon>0$ we have
\[
\proba{\Height(\T_{n}(\X)) \geq (e-\varepsilon) \ln(n)}  \quad \mathop{\longrightarrow}_{n \rightarrow \infty} \quad1.
\]}
{\item[(4)] A natural question is to obtain a limit theorem for the height of the epidemic tree ${\mathscr{T}}^{n}$ in the setting of Theorem \ref{thm:SIR} {(I)} (b) (see Remark \ref{rem:SIR}).}
\item[(5)] It would be very interesting to extend our results to other attachment mechanisms, such as preferential attachment, where new vertices attach to existing vertices proportional to their degree.
\item[(6)] Similarly one may wonder what happens when the frozen vertices are not chosen uniformly, notably when vertices with high degrees are more or less likely to freeze. 
\item[(7)]  An important question arising in the context of growing real-world networks or in the context of infection tracing is how to find the root, or patient zero \cite{CDPCBVY20}. This is  a very active area of research, sometimes called \emph{network archeology} \cite{NK11}. Such questions have been for instance considered for uniform attachment and preferential attachment trees \cite{Hai70,BDG17,LP19}, as well as Bienaym\'e trees \cite{BDG22}.
\end{enumerate}

\bibliographystyle{abbrv}

\begin{thebibliography}{10}

\bibitem{AB15}
L.~Addario-Berry.
\newblock Partition functions of discrete coalescents: from {C}ayley's formula
  to {F}rieze's $\zeta$ (3) limit theorem.
\newblock In {\em XI Symposium on Probability and Stochastic Processes}, pages
  1--45. Springer, 2015.

\bibitem{ABHK21}
L.~Addario-Berry, A.~Brandenberger, J.~Hamdan, and C.~Kerriou.
\newblock {Universal height and width bounds for random trees}.
\newblock {\em Electronic Journal of Probability}, 27(none):1 -- 24, 2022.

\bibitem{ABE18}
L.~Addario-Berry and L.~Eslava.
\newblock High degrees in random recursive trees.
\newblock {\em Random Structures \& Algorithms}, 52(4):560--575, 2018.

\bibitem{AB12}
H.~Andersson and T.~Britton.
\newblock {\em Stochastic epidemic models and their statistical analysis},
  volume 151.
\newblock Springer Science \& Business Media, 2012.

\bibitem{BGY22}
V.~Bansaye, C.~Gu, and L.~Yuan.
\newblock A growth-fragmentation-isolation process on random recursive trees
  and contact tracing.
\newblock {\em To appear in \emph{Ann.~Appl.~Probab.}}

\bibitem{BH84}
A.~D. Barbour and P.~Hall.
\newblock {On the rate of Poisson convergence}.
\newblock {\em Mathematical Proceedings of the Cambridge Philosophical
  Society}, 95(3):473--480, 1984.

\bibitem{BBKKbis23+}
E.~Bellin, A.~Blanc-Renaudie, E.~Kammerer, and I.~Kortchemski.
\newblock Uniform attachment with freezing: scaling limits.
\newblock {\em Ann. Inst. H. Poincar{\'e} Probab. Statist}, (to appear).

\bibitem{Ber22}
J.~Bertoin.
\newblock A model for an epidemic with contact tracing and cluster isolation,
  and a detection paradox.
\newblock {\em Journal of Applied Probability}, pages 1--17, 2023.

\bibitem{Bla22}
A.~Blanc-Renaudie.
\newblock Compactness and fractal dimensions of inhomogeneous continuum random
  trees.
\newblock {\em Probability Theory and Related Fields}, 185(3-4):961--991, 2023.

\bibitem{BV06}
K.~A. Borovkov and V.~A. Vatutin.
\newblock {On the Asymptotic Behaviour of Random Recursive Trees in Random
  Environments}.
\newblock {\em Advances in Applied Probability}, 38(4):1047--1070, 2006.

\bibitem{BDG22}
A.~M. Brandenberger, L.~Devroye, and M.~K. Goh.
\newblock {Root estimation in Galton--Watson trees}.
\newblock {\em Random Structures \& Algorithms}, 61(3):520--542, 2022.

\bibitem{BPBLST19}
T.~Britton, E.~Pardoux, F.~Ball, C.~Laredo, D.~Sirl, and V.~C. Tran.
\newblock {\em Stochastic epidemic models with inference}.
\newblock Springer, 2019.

\bibitem{BDG17}
S.~Bubeck, L.~Devroye, and G.~Lugosi.
\newblock {Finding Adam in random growing trees}.
\newblock {\em Random Structures \& Algorithms}, 50(2):158--172, 2017.

\bibitem{CGHJK96}
R.~M. Corless, G.~H. Gonnet, D.~E.~G. Hare, D.~J. Jeffrey, and D.~E. Knuth.
\newblock On the {Lambert} {{\(w\)}} function.
\newblock {\em Adv. Comput. Math.}, 5(4):329--359, 1996.

\bibitem{Cur}
N.~Curien.
\newblock A random {W}alk among random {G}raphs.
\newblock Lecture notes available online
  \url{https://www.dropbox.com/s/bf6a61x3mpad4um/cours-GA-online.pdf}, 2023.

\bibitem{DN08}
R.~W.~R. Darling and J.~R. Norris.
\newblock Differential equation approximations for {Markov} chains.
\newblock {\em Probab. Surv.}, 5:37--79, 2008.

\bibitem{Dev87}
L.~Devroye.
\newblock Branching processes in the analysis of the heights of trees.
\newblock {\em Acta Informatica}, 24(3):277--298, 1987.

\bibitem{Don75}
R.~Donaghey.
\newblock Alternating permutations and binary increasing trees.
\newblock {\em Journal of Combinatorial Theory, Series A}, 18(2):141--148,
  1975.

\bibitem{DR76}
J.~Doyle and R.~L. Rivest.
\newblock Linear expected time of a simple union-find algorithm.
\newblock {\em Inf. Process. Lett.}, 5(5):146--148, 1976.

\bibitem{ET21}
G.~Elek and G.~Tardos.
\newblock {Convergence and Limits of Finite Trees}.
\newblock {\em Combinatorica}, 42(6):821--852, 2022.

\bibitem{Gas77}
J.~L. Gastwirth.
\newblock {A Probability Model of a Pyramid Scheme}.
\newblock {\em The American Statistician}, 31(2):79--82, 1977.

\bibitem{GM05}
C.~Goldschmidt and J.~Martin.
\newblock {Random Recursive Trees and the Bolthausen-Sznitman Coalesent}.
\newblock {\em Electronic Journal of Probability}, 10(none):718 -- 745, 2005.

\bibitem{GL89}
R.~Grossman and R.~G. Larson.
\newblock Hopf-algebraic structure of families of trees.
\newblock {\em Journal of Algebra}, 126(1):184--210, 1989.

\bibitem{Hai70}
J.~Haigh.
\newblock The recovery of the root of a tree.
\newblock {\em Journal of Applied Probability}, 7(1):79--88, 1970.

\bibitem{JKZV02}
M.~Janic, F.~Kuipers, X.~Zhou, and P.~Van~Mieghem.
\newblock Implications for qos provisioning based on traceroute measurements.
\newblock In B.~Stiller, M.~Smirnow, M.~Karsten, and P.~Reichl, editors, {\em
  From QoS Provisioning to QoS Charging}, pages 3--14, Berlin, Heidelberg,
  2002. Springer Berlin Heidelberg.

\bibitem{Jan21}
S.~Janson.
\newblock Tree limits and limits of random trees.
\newblock {\em Combinatorics, Probability and Computing}, 30(6):849--893, 2021.

\bibitem{KS78}
D.~E. Knuth and A.~Sch{\"o}nhage.
\newblock The expected linearity of a simple equivalence algorithm.
\newblock {\em Theoretical Computer Science}, 6(3):281--315, 1978.

\bibitem{KR01}
P.~L. Krapivsky and S.~Redner.
\newblock Organization of growing random networks.
\newblock {\em Physical Review E}, 63(6):066123, 2001.

\bibitem{Kur16}
R.~K\"ursten.
\newblock Random recursive trees and the elephant random walk.
\newblock {\em Phys. Rev. E}, 93:032111, Mar 2016.

\bibitem{Kur81}
T.~G. Kurtz.
\newblock {\em Approximation of population processes}, volume~36 of {\em
  CBMS-NSF Reg. Conf. Ser. Appl. Math.}
\newblock Philadelphia, PA: Society for Industrial {and} Applied Mathematics
  (SIAM), 1981.

\bibitem{LP19}
G.~Lugosi and A.~S. Pereira.
\newblock {Finding the seed of uniform attachment trees}.
\newblock {\em Electronic Journal of Probability}, 24:1 -- 15, 2019.

\bibitem{Moo74}
J.~W. Moon.
\newblock {\em The distance between nodes in recursive trees}, pages 125--132.
\newblock London Mathematical Society Lecture Note Series. Cambridge University
  Press, 1974.

\bibitem{NR70}
H.~S. Na and A.~Rapoport.
\newblock Distribution of nodes of a tree by degree.
\newblock {\em Mathematical Biosciences}, 6:313--329, 1970.

\bibitem{NH82}
D.~Najock and C.~C. Heyde.
\newblock On the number of terminal vertices in certain random trees with an
  application to stemma construction in philology.
\newblock {\em Journal of Applied Probability}, 19(3):675--680, 1982.

\bibitem{NK11}
S.~Navlakha and C.~Kingsford.
\newblock Network archaeology: uncovering ancient networks from present-day
  interactions.
\newblock {\em PLoS computational biology}, 7(4):e1001119, 2011.

\bibitem{Nev86}
J.~Neveu.
\newblock Arbres et processus de {G}alton-{W}atson.
\newblock In {\em Annales de l'IHP Probabilit{\'e}s et statistiques},
  volume~22, pages 199--207, 1986.

\bibitem{Pit94}
B.~Pittel.
\newblock Note on the heights of random recursive trees and random m-ary search
  trees.
\newblock {\em Random Structures \& Algorithms}, 5(2):337--347, 1994.

\bibitem{CDPCBVY20}
C.~Shah, N.~Dehmamy, N.~Perra, M.~Chinazzi, A.-L. Barab{\'a}si, A.~Vespignani,
  and R.~Yu.
\newblock {Finding patient zero: Learning contagion source with graph neural
  networks}.
\newblock {\em arXiv preprint arXiv:2006.11913}, 2020.

\bibitem{SM95}
R.~T. Smythe and H.~M. Mahmoud.
\newblock A survey of recursive trees.
\newblock {\em Theory of Probability and Mathematical Statistics}, (51):1--28,
  1995.

\end{thebibliography}

\end{document}